\documentclass[11pt]{amsart}

\usepackage{epsfig}

\usepackage[utf8]{inputenc}
\usepackage[T1]{fontenc}
\usepackage{todonotes}
\usepackage{amssymb}
\usepackage{tikz-cd}
\usepackage{comment}
\usepackage{dsfont}
\usepackage{amsthm}
\DeclareMathAlphabet{\mymathbb}{U}{BOONDOX-ds}{m}{n}

\usepackage[shortlabels]{enumitem}
\usepackage[all]{xy}
\usepackage[colorlinks=true, urlcolor=rltblue, citecolor=drkgreen, linkcolor=drkred] {hyperref}
\usepackage{color}
\usepackage{pdfsync}
\definecolor{rltblue}{rgb}{0,0,0.4}
\definecolor{drkred}{rgb}{0.6,0,0}
\definecolor{drkgreen}{rgb}{0,0.4,0}
\usepackage{graphicx}
\usepackage{mathdots}
\usepackage{enumitem}
\setenumerate[1]{label=\arabic*.}

\newtheorem{thm}{Theorem}[section]
\newtheorem{theorem}[thm]{Theorem}
\newtheorem{lemma}[thm]{Lemma}
\newtheorem{proposition}[thm]{Proposition}
\newtheorem{corollary}[thm]{Corollary}
\newtheorem{statement}[thm]{Statement}

\theoremstyle{definition}
\newtheorem{definition}[thm]{Definition}

\theoremstyle{remark}

\newtheorem{remark}[thm]{Remark}
\newtheorem{example}[thm]{Example}

\newtheorem{historic}[thm]{Historic Remark}

\theoremstyle{plain}

\newcounter{contenumi}

\def\la{\langle}
\def\ra{\rangle}

\newcommand{\numberset}{\mathbb}
\newcommand{\N}{\numberset{N}}

\newcommand{\ot}{\operatorname{{o.}t.}}
\newcommand{\height}{\operatorname{ht}}
\newcommand{\ran}{\operatorname{ran}}
\newcommand{\dom}{\operatorname{dom}}
\newcommand{\base}{\operatorname{base}}

\newcommand{\Sub}{\operatorname{Sub}}
\newcommand{\one}{\operatorname{\mymathbb{1}}}
\newcommand{\two}{\operatorname{\mymathbb{2}}}

\def\A{\mathcal{A}}
\def\B{\mathcal{B}}
\def\C{\mathcal{C}}
\def\D{\mathcal{D}}
\def\E{\mathcal{E}}
\def\F{\mathcal{F}}
\def\SD{\mathsf{SD}}
\def\ON{\mathrm{ON}}
\def\om{\omega}
\def\a{\alpha}
\def\b{\beta}
\def\G{\Gamma}
\def\g{\gamma}
\def\d{\delta}
\def\e{\epsilon}
\def\z{\zeta}
\def\1{\mathsd{1}}

\def\pbar{\bar{p}}

\def\conc{{}^\smallfrown}

\def\ggeq{\mathrel{\geq\!\!\!\geq}}

\makeatletter
\newcommand{\dotminus}{\mathbin{\text{\@dotminus}}}

\newcommand{\@dotminus}{%
	\ooalign{\hidewidth\raise1ex\hbox{.}\hidewidth\cr$\m@th-$\cr}%
}
\makeatother

\newcommand{\Ram}{\operatorname{Ram}}
\def\ctimes{\mathbin{{\times}\! \! \!\! {\times}}}
\def\cplus{\mathbin{{\texttt{\#}}}}

\title{The barrier Ramsey theorem}

\author{Alberto Marcone}
\address{Dipartimento di Scienze Matematiche, Informatiche e Fisiche, Universit\`a \newline \indent di Udine, Italy}
\email{alberto.marcone@uniud.it}
\urladdr{\href{http://users.dimi.uniud.it/~alberto.marcone/}{users.dimi.uniud.it/$\sim$alberto.marcone}}

\author{Antonio Montalb\'an}
\address{Department of Mathematics\\
University of California, Berkeley\\
USA}
\email{antonio@math.berkeley.edu}
\urladdr{\href{http://www.math.berkeley.edu/~antonio/index.html}{www.math.berkeley.edu/$\sim$antonio}}

\author{Andrea Volpi}
\address{Dipartimento di Scienze Matematiche, Informatiche e Fisiche, Universit\`a \newline \indent di Udine, Italy}
\email{andrea.volpi@uniud.it}
\urladdr{\href{https://andreasdfghj.github.io/andreavolpi/}{andreasdfghj.github.io/andreavolpi}}

\thanks{The first and third authors were partially supported by the Italian PRIN 2022 \emph{Models, sets and classifications}, prot.\ 2022TECZJA, funded by the European Union - Next Generation EU. The first author is a member of INdAM-GNSAGA.
The second author was partially supported by NSF grant DMS 1954062.}

\begin{document}

\begin{abstract}
In this paper we study a very general finite Ramsey theorem, where both the sets being colored and the homogeneous set must satisfy some largeness notion.
For the homogeneous set this has already been done using the notion of $\alpha$-largeness, where $\alpha$ is a countable ordinal equipped with a system of fundamental sequences.
To extend this approach the more appropriate notion is barrier largeness.
Since the complexity of barriers can be measured by countable ordinals, we define Ramsey ordinals and, using appropriate iterations of the Veblen functions, we are able to compute them.
\end{abstract}

\maketitle

\section{Introduction}\label{sec:introduction}

It is known since Gödel's first incompleteness theorem that there are true statements about the standard natural numbers that cannot be proved with the axioms of Peano arithmetic.
In \cite{MR3727432} Paris and Harrington proved that a certain statement in finite Ramsey theory, expressible in Peano arithmetic, is not provable in this system.
This has been claimed to be the first \lq\lq natural\rq\rq\ statement independent from Peano arithmetic.

\begin{statement}[Paris-Harrington 1977]\label{PH}
Let $n, k, m \in \N$ be such that $m \ge n$.
There exists $N \in \N$ such that for each $s \subseteq \N$ of cardinality $N$ and each coloring of the $n$-size subsets of $s$ in $k$ colors, there exists a set $t \subseteq s$ which is homogeneous for the coloring and such that $|t| > \max (m, \min t)$.
\end{statement}

After this seminal result, Ketonen and Solovay \cite{MR607894} introduced the notion of $\a$-largeness for $\a < \e_0$ (recall that $\e_0$ denotes the first fixed point of the ordinal exponential function and is the proof theoretic ordinal of Peano arithmetic) and improved the computation of $N$ in Statement \ref{PH}.

Given a countable ordinal $\G$ for which we have a system of fundamental sequences (i.e.\ sequences of ordinals approaching any ordinal $< \Gamma$ from below: see Subsection \ref{subsec:systems} for the precise definition), we can define what it means for a finite set of natural numbers to be $\a$-large for $\a < \G$.
Here we mention that typically (though not always, since largeness depends on the chosen system of fundamental sequences) a set $s$ is $n$-large (for $n \in \om$) when $|s| \ge n$, $\om$-large if $|s| > \min s$, $\om^2$-large if it can be partitioned into $\min s$ sets, each entirely preceding the next one and $\om$-large.
Notice that $\a < \b$ does not necessarily imply that every $\b$-large set is also $\a$-large (e.g.\ $\{2,5,8\}$ is $\om$-large but not $n$-large for $3 < n < \om$).
In this parlance, Paris-Harrington theorem states the existence of homogeneous $\om$-large sets for $k$-colorings of $n$-large subsets of a $N$-large set. 

More recently, various authors considered statements extending Statement \ref{PH} to largeness notions and established some relationships between the various parameters \cite{MR1694401, MR1969630, MR2231881, MR2335083, MR2510280, MR4125988, dSW}.
All of these results consider $\alpha$-largeness notions for $\a < \e_0$ (\cite{dSW} for $\a < \e_\om$) and Ramsey like statements that restrict the sizes of the tuples being colored to be $n$ for some fixed natural number $n$. 
%A finite set $t \subset \N$ is $\g$-size if it is $\g$-large but $t \setminus \{\max t\}$ is not.
%Given a set $s \subseteq \N$ and $\g \le \G$, we denote by $[s]^\g$ the collection of all $\g$-size subsets of $s$ (this extends the usual notation that $[s]^n$ is the collection of all subsets of $s$ of size $n < \om$).
In this paper we extend the previous results by studying colorings of all $\g$-size subsets of some finite set $s$ also when $\g \ge \om$.
In \cite{MR677888, MR2346474} the infinite Ramsey theorem has been extended to colorings of the $\g$-size subsets of an infinite set (the homogeneous set is required to be infinite).
%\todo{carlucci and zdanowski studied colorings of $\om$-size sets in reverse math; carlucci et al.\ studied colorings of $\om$-size sets for free set, thin set and rainbow ramsey in reverse math}

%We need to consider also notions of largeness that are more complex than $\a$-largeness: given $\a, \b < \G$ we say that $s$ is $(\a \uplus \b)$-large if we can write $s = s_\b \cup s_\a$ where $s_\b < s_\a$ (meaning that every element of $s_\b$ is smaller than every element of $s_\a$), $s_\b$ is $\b$-size and $s_\a$ is $\a$-large. 
%We show that this notion is stronger, sometimes strictly, than $(\a + \b)$-largeness.
%We also iterate the use of $\uplus$, thus obtaining the notion of $(\a \uplus \b \uplus \g)$-large set.

%Largeness notions induced by fundamental sequences of ordinals have been widely used.
We introduce a new, more flexible, framework: largeness notions induced by blocks and barriers.
%Blocks and barriers are combinatorial objects introduced by Nash-Williams in \cite{MR173640} to define the notion of better-quasi-orderings, which refines well-quasi-orderings (see also \cite{MR4291876}).
We give the formal definition of blocks and barriers in Subsection \ref{subsec:barriers}.
%We denote them with letters such as $\A, \B, \C$.
Largeness notions induced by blocks generalize largeness notions induced by systems of fundamental sequences.
Moreover, each block has a countable ordinal associated to it, its height, which measures the complexity of the block and of its largeness notion.
%On the other hand, the height of the block induced by the notion of $\a$-largeness is exactly $\a$.

Notice also that a system of fundamental sequences is defined on some (typically constructive) ordinal $\G$ and once we fix the system we can work only with ordinals below $\G$.
On the other hand, we can consider blocks with arbitrary (countable) height.

%Moreover such system must have nice properties such as the Bachmann property or the nestedness (which will be defined in Subsection \ref{subsec:systems}) which in general are difficult to prove.

A useful tool to state results in Ramsey theory is the arrow notation: in Section \ref{sec:ramseytheorem} we adapt this notation to our framework and write $\B \rightarrow (\A)^\C_k$ where $\A$, $\B$ and $\C$ are blocks and $k \in \N$.
Then, given ordinals $\a$ and $\g$, we define $\Ram(\a)^\g_k$ as the least $\b$ such that for each $\A$ and $\C$ of height respectively $\a$ and $\g$ there exists $\B$ of height $\b$ with $\B \rightarrow (\A)^\C_k$.

The main result of the paper is that, for $\g < \a$ with $\a$ infinite, the ordinal $\Ram(\a)^{1 + \g}_{<\om} = \sup_{k \in \N} \Ram(\a)^{1 + \g}_k$ is 
\[
\varphi_{\log \g} (\a \cdot \om),
\]
where $\varphi_{\log \g}$ denotes a composition of Veblen functions indexed by the ordinals resulting from a logarithm of $\g$ obtained from its Cantor normal form.\smallskip

We now describe the structure of the paper.
In Section \ref{sec:largeness} we introduce systems of fundamental sequences and the largeness notions induced by them.
We also introduce a way to pass from a system of fundamental sequences on some ordinal $\zeta$ to a system on $\Gamma_\zeta$ (which is the $\zeta$-th ordinal closed under the binary Veblen function). 
We also highlight some properties that the new system satisfies.
Then we give the basics of blocks and barriers and define largeness notions induced by them. 

In Section \ref{sec:bridges} we show how to produce blocks and barriers starting from a system of fundamental sequences, preserving the same largeness notions.
Then we show how to obtain what we call a pseudosystem of fundamental sequences starting from a block or a barrier.

In Section \ref{sec:ramseytheorem} we introduce the notation for Ramsey theory in the framework with blocks and barriers.
We also define the ordinal $\varphi_{\log \g} (\a \cdot \om)$ that in the following sections we prove to be $\Ram(\a)^{1 + \g}_{<\om}$.

Section \ref{sec:pigeonhole} deals with the barrier pigeonhole principle, which computes $\Ram(\a)^1_k$.
In Section \ref{sec:lowbound} we show that under our hypothesis $\varphi_{\log \g} (\a \cdot \om) \leq \Ram(\a)^{1 + \g}_{<\om}$: the proof employs the system of fundamental sequences described in Section  \ref{sec:largeness}.
Section \ref{sec:upbound} completes the proof by showing $\Ram(\a)^{1 + \g}_{<\om} \leq \varphi_{\log \g} (\a \cdot \om)$: this is obtained for arbitrary countable ordinals $\a$ and $\g$.

Finally, in Section \ref{sec:nestedness} we show that our system of fundamental sequences enjoys a crucial property used in the proofs of Section \ref{sec:lowbound}.
Since this proof is rather technical we delay it to the last section.

\section{Largeness Notions}\label{sec:largeness}

If $X \subseteq \N$ we denote by $[X]^{< \om}$ the set of finite subsets of $X$ and, for $n \in \N$, by $[X]^n$ the set of subsets of $X$ with exactly $n$ elements.
We use uppercase letters $X, Y, Z$ to denote subsets of $\N$ which may be finite or infinite.
We use lowercase letters $s, t, u$ to denote finite subsets of $\N$.

We identify a subset of $\N$ with the strictly increasing sequence (finite or infinite) which enumerates it and denote by $X(i)$ or $X_i$ the element in position $i$ in the sequence.
If $s$ is such a finite sequence then $|s|$ denotes its length (coinciding with the cardinality of the set) and we write $s = \la s_0, \ldots, s_{|s|-1} \ra$.

Given $s,t \subseteq \N$, by $s \sqsubseteq t$ we mean that $s$, as a sequence, is an initial segment (or prefix) of $t$.
This is stronger than $s \subseteq t$, which denotes set-theoretic inclusion as usual.
The irreflexive versions of the previous relations are denoted by $\sqsubset$ and $\subset$.

We let $s^* = s \setminus \{\max s\}$ and $X^- = X \setminus \{\min X\}$.
By $s < X$ we mean that each element of $s$ is strictly smaller than each element of $X$ or equivalently (when $s$ and $X$ are nonempty) $\max s < \min X$.
We write $s \conc X$ for the concatenation of the sequences $s$ and $X$.
Notice that in our setting $s \conc X$ does make sense only if $s < X$ and, set-theoretically, coincides with $s \cup X$.

\subsection{Systems of Fundamental Sequences}\label{subsec:systems}

For ordinals $\a$ and $\b$ we write $\a \ggeq \b$ to mean that the exponent of the last term in the Cantor normal form of $\a$ is greater than or equal to the exponent of the first term in the Cantor normal form of $\b$.

\begin{definition}
A \textit{fundamental sequence} for an ordinal $\a > 0$ is a non-decreasing sequence of ordinals $\a [n]$ such that $\sup \{\a [n] + 1 : n \in \om \} = \a$.
For notational convenience we stipulate that the fundamental sequence of $0$ is always $0[n] = 0$ for all $n$.

Let $\G$ be an ordinal.
If we fix a fundamental sequence for each ordinal $\a \le \G$ we speak of a \textit{system of fundamental sequences on $\G$}.
\end{definition}

\begin{definition}
A system of fundamental sequences on $\Gamma$ is \textit{regular} if for each ordinal $\b_0 + \om^{\b_1} \le \Gamma$ with $\b_0 \ggeq \om^{\b_1}$ (i.e.\ $\b_0$ are the leading terms in the Cantor normal form, which has another term, of $\b_0 + \om^{\b_1}$) and each $n$ it holds that
$$(\b_0 + \om^{\b_1})[n] = \b_0 + (\om^{\b_1}[n]).$$
\end{definition}
	
Notice that if $\a$ is an ordinal such that $\om^\b < \a < \om^{\b+1}$, regularity implies that $\a[n] \ge \om^\b$ for each $n$.

\begin{definition}\label{def:nested}
A system of fundamental sequences on $\Gamma$ is \textit{nested} if it is never the case for $\g < \b \le \G$ and $n > 1$ that 
$$\g > \b[n] > \g[n].$$
\end{definition}

Regularity is a very natural property that is satisfied by most systems of fundamental sequences that have been studied.
Nestedness is a weakening of the Bachmann property which states that for ordinals $\g < \b \le \G$ it is never the case that $\g > \b[n] > \g[1]$.
The Bachmann property was introduced in \cite{MR219424} and studied in \cite{MR1141947}.
An example of a system of fundamental sequences that satisfies the Bachmann property (but, incidentally, is not regular) is provided in \cite{MR4749344}.

For the rest of Subsection \ref{subsec:systems} we fix a system of fundamental sequences on some ordinal $\G$.
We also assume to work only with ordinals $\le \G$.

\begin{definition}
Let $s \in [\N]^{< \om}$ and $\a$ be an ordinal.
We denote by $\a[s]$ (or $\a[s_0, \ldots, s_{|s|-1}]$) the ordinal $\a[s_0] \ldots [s_{|s|-1}]$.

We say that $s$ is \textit{$\a$-large} if $\a[s] = 0$, \textit{$\a$-small} if $\a[s] > 0$ and \textit{$\a$-size} if $s$ is $\a$-large but $s^*$ is $\a$-small.
($\a$-size are often called \lq\lq exactly $\a$-large\rq\rq\ in the literature.)
\end{definition}

Notice that every $\a$-large set has a $\a$-size initial segment.
We denote by $[X]^\a$ the set of $\a$-size subsets of $X$.

We introduce a sum operation for largeness notions.
Given ordinals $\a, \b$ and a set $s$, we say that $s$ is $(\a \uplus \b)$-large if $s$ can be partitioned in two parts $s_\b < s_\a$ such that $s = s_\b \conc s_\a$, $s_\b$ is $\b$-size and $s_\a$ is $\a$-large.
We may also say that $s$ is $(\a \uplus \b)$-size if in addition $s_\a$ is $\a$-size.%, and that $s$ is $(\a \uplus \b)$-small if either $s$ is $\b$-small or $s_\a$ is $\a$-small.

We start by proving a couple of lemmas about nested systems of fundamental sequences.
	
\begin{lemma}\label{sequencelemma}
Assume the system of fundamental sequences is nested.
Let $n > 1$ and $(n_i)_{i \in \om}$ be a sequence such that $n \le n_i$ for each $i$. 
Then for each ordinal $\a$ there exists $k \in \N$ such that $\a[n] = \a[n_0, \ldots, n_k]$.
\end{lemma}
	
\begin{proof}
Fix $\a$ and notice that the sequence $(\a[n_0, \ldots, n_m])_{m \in \om}$ is strictly decreasing until it reaches $0$: let $k \in \N$ be largest such that $\a[n_0, \ldots, n_k] \ge \a[n]$ ($k$ exists because $\a[n] \le \a[n_0]$).
We claim that $\a[n_0, \ldots, n_k] = \a[n]$.
		
For the sake of contradiction suppose that $\a[n_0, \ldots, n_k] > \a[n]$.
Then $\a > \a[n_0, \ldots, n_k] > \a[n]$ and so by nestedness $\a[n_0, \ldots, n_k, n] \ge \a[n]$. 
Since $n \le n_{k+1}$ then $\a[n_0, \ldots, n_{k+1}] \ge \a[n]$, contradicting the maximality of $k$.
\end{proof}
	
\begin{lemma}\label{subsetlemma}
Assume the system is nested.
Let $s, t \in [\N]^{< \om}$ be such that $|s| \le |t|$ and $1 < t(i) \le s(i)$ for each $i < |s|$.
Then $\a[t] \le \a[s]$ for each ordinal $\a$.
\end{lemma}
	
\begin{proof}
We define a strictly increasing function $f$ with $\dom (f) \sqsubseteq \{0, \ldots, |t|\}$, $\ran (f) \subseteq \{0, \ldots, |s|\}$ and $\a[t \restriction i] = \a[s \restriction f(i)]$ whenever $f(i)$ is defined.
		
We start by setting $f(0)=0$ which satisfies $\a[t \restriction 0] = \a = \a[s \restriction f(0)]$.
Assume now that we have defined $f(i)$ and that $\a[t \restriction i] = \a[s \restriction f(i)]$.
Since $f$ is strictly increasing we have that $i \le f(i)$ and so $t(i) \le s(i) \le s(f(i)) \le \ldots \le \max s$.
By Lemma \ref{sequencelemma} we have two cases.
Either there exists $j$ with $f(i) \le j < |s|$ such that 
$$\a[t \restriction i][t(i)] = \a[t \restriction i][s(f(i)), \ldots, s(j)],$$
so that $\a[t \restriction {i+1}] = \a[s \restriction {f(i)}] [s(f(i)), \ldots, s(j)] = \a [s \restriction {j+1}]$. 
In this case we set $f(i+1)=j+1$.
Otherwise  $\a[t \restriction {i+1}] < \a [t \restriction i] [s(f(i)), \ldots, \max s] = \a[s]$: in this case we let $f(i+1)$ undefined and we get $\a[t] \le \a[t \restriction {i+1}] < \a[s]$ hence our thesis.
If $f$ is defined on the whole set $\{0, \ldots, |t|\}$ then it must be $|t|=|s|$ and $f(|t|)=|s|$.
Therefore $\a[t] = \a[s]$ and the thesis follows.
\end{proof}

These lemmas show that largeness notions produced by nested systems of fundamental sequences behave as expected.
The proof of the following corollary is immediate from Lemma \ref{subsetlemma}.

\begin{corollary}\label{intuitive syst}
Assume the system is nested and let $s \subseteq t$ with $\min t > 1$.
If $s$ is $\a$-large for some ordinal $\a$, then $t$ is $\a$-large too.
If $s \subsetneq t$ and $t$ is $\a$-size, then $s$ is $\a$-small.
\end{corollary}

Given ordinals $\b < \a$, it is not always the case that an $\a$-large set is $\b$-large.
To analyze this situation, Ketonen and Solovay developed in \cite{MR607894} the notion of $\Rightarrow_n$ and proved that this relation has nice properties on ordinals below $\epsilon_0$.
This leads to the fact, proven below, that every $\a$-large set with the minimum larger than a bound depending only on $\b$ is $\b$-large.
		
\begin{definition}
Given $\a, \b$ we write $\a \Rightarrow_n \b$ to mean that $\b = \a[n, \dots, n]$ for some number (possibly $0$) of $n$'s.
\end{definition}
				
\begin{lemma}\label{lem:canonicalnorm}
Assume the system is nested.
If $\b \le \a$ there exists $n>1$ such that $\a \Rightarrow_n \b$.
\end{lemma}

\begin{proof}
Let $n_0, \dots, n_k$ be a sequence such that $\b = \a [n_0, \ldots, n_k]$: such a sequence exists by letting $n_0 = \min \{n > 1 : \b \le \a[n]\}$ and $n_{i+1} = \min \{n > 1 : \b \le \a[n_0, \dots, n_i, n]\}$ as long as $\b < \a[n_0, \dots, n_i]$. 
Since the sequence $\a [n_0, n_1, \ldots]$ is strictly decreasing, at some finite stage it must be $\b = \a [n_0, \ldots, n_k]$.
Let $n = \max \{n_0, \dots, n_k\}$.
By Lemma \ref{sequencelemma}, $\a [n_0] = \a [n, \dots, n]$ for a number of $n$'s.
Iterating, we get $\b = \a [n_0, \dots, n_k] = \a [n, \dots, n]$, i.e.\ $\a \Rightarrow_n \b$.
\end{proof}
		
\begin{definition}\label{normdefinition}
Assume the system is nested.
To each ordinal $\a$ we associate a natural number $|\a|$ that we call the norm of $\a$:
$$|\a| = \min \{n > 1 : \G \Rightarrow_n \a \}.$$
\end{definition}

The norm function was introduced in \cite{MR607894} as well.
		
\begin{proposition}\label{prop: geq norm}
Assume the system is nested.
If $n \ge |\a|$ then $\G \Rightarrow_n \a$.
\end{proposition}

\begin{proof}
Immediate from Lemma \ref{sequencelemma}.
\end{proof}

The next property of the norm function is crucial: we say that the norm is good because it satisfies this property.

\begin{proposition}\label{prop: goodnorm}
Assume the system is nested.
If $\b < \d$ then $\b \le \d[|\b|]$.
\end{proposition}
		
\begin{proof}
Let $n = |\b|$.
For some number of $n$'s we have 
$$\G [n, \ldots ,n] \ge \d > \G [n, \ldots ,n][n] \ge \G [n,\ldots,n][n,\ldots,n] = \b.$$
If $\G [n,\ldots,n] = \d$ then we immediately get $\d[n] \ge \b$.
On the other hand, if $\G [n,\ldots,n] > \d$, since we have also $\d > \G [n,\ldots,n] [n]$, then nestedness implies	$\d[n] \ge \G [n,\ldots,n][n] \ge \b$.
\end{proof}
		
Notice that the norm function depends on the system of fundamental sequences.
For this reason, the goodness is really a property of the system itself.
The following important lemmas hold for any good norm.
		
\begin{lemma}\label{lem: rm above norm} 
Assume the system is nested.
For any $\a$ and $\b$ and $n \ge |\b|$,
$$\a \ge \b \quad \iff \quad \a \Rightarrow_n \b.$$
\end{lemma}
		
\begin{proof}
The right-to-left direction is obvious. 
			
The left-to-right direction is proved by transfinite induction.
We have from Proposition \ref{prop: goodnorm} that if $\a > \b$, then $\a[n] \ge \a[|\b|] \ge \b$.
Then, by the induction hypothesis, we get $\a[n] \Rightarrow_n \b$.
\end{proof}

The following lemma was proved by Bigorajska and Kotlarski in \cite{MR2231881} for their specific  system of fundamental sequences on $\e_0$.
We show that any nested system satisfies it.

\begin{lemma}[Estimation Lemma]\label{lem: estimation lemma}
Assume the system is nested.
If $s$ is such that $s^*$ is $\a$-large (equivalently, $s$ is $(1 \uplus \a)$-large) then there is no strictly decreasing function $h \colon s \rightarrow \a$ satisfying $|h(s_i)| \le s_i$ for all $i < |s|$.
\end{lemma}
		
\begin{proof}
Let $\la s_0, \dots, s_n \ra$ be the $\a$-size proper prefix of $s$.
Suppose $h \colon s \rightarrow \a$ is strictly decreasing and satisfies $|h(s_i)| \le s_i$ for all $i < |s|$.

We prove by induction on $i \le n$ that $h(s_i) \le \a[s_0, \ldots, s_i]$.
When $i = 0$, $h(s_0) < \a$ and $|h(s_0)| \le s_0$ imply, by Proposition \ref{prop: goodnorm}, $h(s_0) \le \a[|h(s_0)|] \le \a[s_0]$.
For the successor step, we have $h(s_i) \le \a[s_0, \ldots, s_i]$ by induction hypothesis, and hence $h(s_{i + 1}) < \a[s_0, \dots, s_i]$ because $h$ is strictly decreasing.
Since we are assuming $|h(s_{i+1})| \le s_{i+1}$, using Proposition \ref{prop: goodnorm} we have $h(s_{i+1}) \le \a[s_0, \ldots, s_i][|h(s_{i+1})|] \le \a[s_0, \ldots, s_{i+1}]$.
		
We thus obtain $h(s_n) \le \a[s_0, \dots, s_n] = 0$ and we cannot have that $h(s_{n+1}) < h(s_n)$.
\end{proof}

Various systems of fundamental sequences have been developed, and they all look more or less the same. 
They are usually developed on an ordinal for which we have some ordinal notation.
Here instead, we define a way of going from a system of fundamental sequences on an ordinal $\zeta$ to a system of fundamental sequences on the ordinal $\Gamma_\zeta$ -- the $\zeta$-th fixed point for the binary Veblen function. 
Recall that the unary Veblen functions are a hierarchy of normal functions introduced by Veblen in \cite{MR1500814}.
The function $\varphi_0$ is ordinal exponentiation in base $\om$ and for each ordinal $\a > 0$, $\varphi_\a$ enumerates the common fixed points of $\varphi_\b$ for $\b < \a$.
Since the Veblen functions are normal, each of them has infinitely many fixed points.
Then one can define the binary Veblen function $\varphi(\a,\b) = \varphi_\a (\b)$: its fixed points are the ordinal $\a$ such that $\a = \varphi(\a,0)$.
For each $\zeta$, $\G_\zeta$ is the $\zeta$-th ordinal $\a$ such that $\varphi_\a (0) = \a$ (i.e.\ the $\zeta$-th fixed point of the binary Veblen function).

\begin{definition}\label{inducedsystem}
Given a system of fundamental sequences on some ordinal $\zeta$ (which we denote as $\xi \lceil n \rceil$ for $\xi \le \zeta$), we define a system of fundamental sequences on $\G_\zeta$ as follows.
\begin{enumerate}[(1)]
	\item If $\a = \a_0 + \om^{\a_1}$ with $\a_0 \ggeq \om^{\a_1}$ then $\a [n] = \a_0 + (\om^{\a_1} [n])$.
	\item $0[n] = 0$ and $1[n] =0$.
	\item If $0 < \a < \om^\a$ then $\om^\a [n] = \om^{\a [n]} \cdot n$.
	\item If $0 < \d < \varphi_\d (0)$ then $\varphi_{\d} (0) [n] = \varphi_{\d [n]}^{n + 1} (0)$.
	\item If $0 < \a < \varphi_\d (\a)$ and $0 < \d$ then $\varphi_{\d} (\a) [n] = \varphi_{\d [n]}^{n+1}(\varphi_{\d}(\a [n]) + 1)$.
	\item $\G_0 [n] = \varphi_{\varphi_{\ddots \varphi_0 (0) \iddots} (0)} (0)$, where $\varphi$ is iterated $n+1$ times.
    \item \label{7} If $0 < \xi \le \zeta$ then $\G_\xi [n] = \varphi_{\varphi_{\ddots \varphi_{\G_{\xi \lceil n \rceil} + 1} (0) \iddots} (0)} (0)$, with $\varphi$ iterated $n + 1$ times.
\end{enumerate}
\end{definition}
		
Notice that \ref{7} above is where we are explicitly using the system of fundamental sequences $\xi \lceil n \rceil$ on $\zeta$.
It is not hard to see that Definition \ref{inducedsystem} really yields a system of fundamental sequences.
Moreover, this system is regular.

\begin{theorem}\label{nested->nested}
If the system of fundamental sequences on $\zeta$ is nested then the system of fundamental sequences on $\G_\zeta$ introduced in Definition \ref{inducedsystem} is nested.
\end{theorem}

We postpone the quite technical proof of Theorem \ref{nested->nested} to Section \ref{sec:nestedness}.
We may assume from now on that the system of Definition \ref{inducedsystem} is nested.

\begin{remark}
Notice that our system fails the Bachmann property just by the definition of the fundamental sequences for successors and powers of $\om$.
Let $\b_0$ be a limit ordinal which is not an $\e$-ordinal (namely, not a fixed point of the exponential function $\varphi_0$) and $n > 1$.
Then set $\g_0 = \b_0[n] + 1$, so that in particular $\g_0 < \b_0$.
Let $\b = \om^{\b_0}$ and $\g = \om^{\g_0}$: then $\b[n] = \om^{\b_0[n]} \cdot n$ and $\g[1] = \om^{\b_0[n]}$, so that $\b > \g > \b[n] > \g[1]$.

Moreover, notice that without the requirement $n > 1$ in Definition \ref{def:nested} our system cannot be nested.
In fact, for $\g = \om^{\om^{\om^{\epsilon_0 + 1}}}$, $\b = \epsilon_1$ and $n = 1$ we have $\g[1] = \om^{\om^{\om^{\epsilon_0}}} = \epsilon_0 < \b[1] = \om^{\om^{\epsilon_0+1}} < \g < \b$.
\end{remark}

\subsection{Fronts, Blocks and Barriers}\label{subsec:barriers}

Blocks and barriers are combinatorial objects introduced in \cite{MR173640} to define better quasi orderings, a refinement of the notion of well quasi orderings.
Laver later used better quasi orderings in \cite{MR279005} to prove Fraïssé's order type conjecture.
Here we give the definition of a front \cite{MR4291876}, which is a slightly more general object than a block but it is more useful for our purposes.
We consider also smooth barriers, which were introduced in \cite{MR1219735}.

\begin{definition}\label{blockdef}
A set $\B \subseteq [\N]^{< \om}$ is a \textit{front} if either $\B = \{ \la \ra \}$ or it satisfies the following properties:
\begin{enumerate}[(1)]
	\item $\base(\B) = \{n \in \N: \exists s \in \B \, \exists i < |s| \, (s(i) = n)\}$ is infinite;
	\item for each infinite $X \subseteq \base(\B)$ there is some $s \in \B$ such that $s \sqsubset X$;
	\item for each $s,t \in \B$ $s \not\sqsubset t$.
\end{enumerate}

We say that $\B = \{ \la \ra\}$ is the \textit{degenerate front} and by convention we say that $\base(\{ \la \ra\}) = \N$.
A \textit{block} is a non degenerate front.

A \textit{barrier} is a front $\B$ such that for each $s,t \in \B$ $s \not\subset t$.

A front $\B$ is \textit{smooth} if for each $s,t \in \B$ with $|s| < |t|$ there exists $i < |s|$ such that $s(i) < t(i)$.
\end{definition}

For every $n$, $\{s \in [\N]^{< \om} : |s| = n\}$ is a smooth barrier.
For convenience, we denote by $\one$ and $\two$ the barriers of singletons and pairs.

The set $\{s \in [\N]^{< \om} : |s| = s(0) + 1\}$ is also a smooth barrier, often called the Schreier barrier.

A subfront (respectively, a subblock, a subbarrier) is a subset of a front which is still a front (respectively, a block, a barrier).
If $\B$ is a block and $\B' \subseteq \B$ is a subblock of $\B$ then it must be $\B' = \{s \in \B : s \subseteq X\}$ for some infinite $X \subseteq \base(\B)$. 
Vice versa, for each infinite $X \subseteq \base(\B)$ we obtain a subblock of $\B$ by restricting to the elements of $\B$ which are subsets of $X$: we denote it by $\B \restriction X$.
It is straightforward to see that a smooth block is automatically a barrier. 
It is also immediate that any subblock of a (smooth) barrier is a (smooth) barrier.
	
Notice that up to isomorphism we may assume that the base of a front is $\N$.
We often tacitly do that in proofs where only one block is involved and no restrictions on its base are required.
	
Every front is well ordered by lexicographic order and we denote by $\ot (\B)$ its \textit{order type}.
The degenerate front is the only front of order type $1$.
While the order type of blocks can be any countable limit ordinal, it can be proved that the order type of barriers are exactly the ordinals of the form $\om^\b \cdot n$ for $0 \le \b < \om_1$, $0 < n < \om$ and if $\b < \om$ then $n = 1$ \cite{MR366758}.
Any barrier of order type $\om^\b \cdot n$ contains a subbarrier of order type $\om^\b$. 
The order type of each smooth barrier is indecomposable (i.e.\ of the form $\om^\b$ for some $0 \le \b < \om_1$) and for each $\b$ with $0 \le \b < \om_1$ there exists a smooth barrier with order type $\om^\b$ \cite{MR1219735}. 
However there exist barriers with indecomposable order type which are not smooth.
	
Another measure of the complexity of a front is its \textit{height}.
Let $\B$ be a front and let 
$$T(\B) = \{s \in [\base(\B)]^{< \om}: \forall t \, (t \sqsubset s \rightarrow t \notin \B)\}.$$
In other words $T(\B)$ is the closure under initial segments of $\B$.
Then $T(\B)$ is a well founded tree and the elements of $\B$ are its leaves.
The \textit{height} $\height(\B)$ is the height of the tree $T(\B)$, defined as usual.
The degenerate front is the only front with height $0$.
For $s \in T(\B)$ we denote its height in the tree $T(\B)$ by $\height_\B(s)$ or simply by $\height(s)$ if the front $\B$ is clear from the context.
Marcone proved the following in \cite{MR1219735}.

\begin{lemma}\label{comb}
If $\B$ is a block, then there exists a smooth barrier $\B'$ such that $\base(\B) = \base(\B')$, $\height(\B) = \height(\B')$ and $\B \subset T(\B')$.
\end{lemma}

\begin{definition}
If $\A$ and $\B$ are fronts such that $\base(\B) = \{n \in \base(\A) : n \ge m\}$ for some $m$ (we say that $\base(\B)$ is a \textit{final segment} of $\base(\A)$) then the \textit{sum} $\A \oplus \B$ is the front $\{s \conc t: s \in \B, t \in \A, \max s < \min t\}$.

If the base of one of the fronts $\A$ and $\B$ is a final segment of the other we define the \textit{union} $\A \sqcup \B$ to be the front consisting of the leaves of $T(\A) \cup T(\B)$.

If we apply either operation to smooth barriers we obtain a smooth barrier.  
\end{definition}
    
\begin{remark}\label{heightunionsum}
Notice that $\base(\A \oplus \B) = \base(\B)$ and $\base(\A \sqcup \B) = \base(\A) \cup \base(\B)$.

Arguing by induction and assuming that $\A$ and $\B$ are smooth barriers\footnote{this might fail when $\A$ is not a smooth barrier.}, for each $s \in T(\B)$ $\height_{\A \oplus \B} (s) = \height(\A) + \height_\B(s)$ and hence $\height(\A \oplus \B) = \height(\A) + \height(\B)$ (this explains the order of the summands in our notation).

Again an easy induction shows that if $\A$ and $\B$ are fronts then for each $s \in T(\A \sqcup \B)$ it holds $\height_{\A \sqcup \B} (s) = \max (\height_\A(s), \height_\B(s))$, where we stipulate that $\height_\A(s) = -1$ if $s \notin T(\A)$ (and the same for $\B$).
It follows that $\height(\A \sqcup \B) = \max (\height(\A), \height(\B))$.
\end{remark}
	
If $\B$ is a block and $s \in T(\B)$ then we define the tree $T_s(\B) = \{t : s \conc t \in T(\B)\}$ and we call $\B_s$ the set of its leaves, so that $\B_s = \{t \in [\base(\B)]^{< \om} : s \conc t \in \B\}$.
One can check that $\B_s$ is a front, $\height_\B (s) = \height (\B_s)$ and if $\B_s$ is non degenerate then $\base(\B_s) = \base(\B) \setminus \{0, \ldots, \max s\}$.
Moreover, if $\B$ is a smooth barrier then $\B_s$ is a smooth barrier too.
If $s = \la n \ra$ then we simply write $T_n(\B)$ and $\B_n$ instead of $T_{\la n \ra} (\B)$ and $\B_{\la n \ra}$. 
	
We now establish the relationship between the order type and the height of a smooth barrier.
	
\begin{proposition}\label{smoothhtot}	
If $\B$ is a smooth barrier then $\ot (\B) = \om^\a$ if and only if $\height(\B) = \a$.
\end{proposition}
	
\begin{proof}
If $\B$ is the degenerate front then the equivalence is immediate with $\a = 0$.

Now assume $\ot (\B) = \om^\a$ with $\a>0$ and suppose the result holds for all smooth barriers of order type less than $\om^\a$.
Notice that $\la n \ra \in T(\B)$ for every $n$. 
Since $\B$ is a smooth barrier, each $\B_n$ is a smooth barrier (possibly degenerate) and so $\ot(\B_n) = \om^{\gamma_n}$ for some $\gamma_n$.
Since $\om^\a = \sum_{n \in \om} \om^{\gamma_n}$ we have $\sup \{\gamma_n +1 : n \in \om\} = \a$.
By inductive hypothesis $\height (\la n \ra) = \height (\B_n) = \gamma_n$ and therefore $\height(\B) = \sup \{\gamma_n +1 : n \in \om\} = \a$.
		
If instead $\height(\B) = \a$ then, since $\B$ is a smooth barrier, $\ot(\B) = \om^\g$ for some $\g$.
By the direction already proved, it must be $\g = \a$.
\end{proof}

We noticed before that for each countable ordinal $\b$ there exists a smooth barrier of order type $\om^\b$.
By Proposition \ref{smoothhtot}, for each countable ordinal $\b$ there exists a smooth barrier of height $\b$.

A nice property satisfied by each smooth barrier $\B$ (but not by every barrier) is that the height is preserved by restricting to final segments of $\base (\B)$.

\begin{lemma}\label{sb tail height}
If $\B$ is a smooth barrier and $X$ is a final segment of $\base(\B)$ then $\height (\B) = \height (\B \restriction X)$. 
\end{lemma}

\begin{proof}
It is clear that $\height(\B) \ge \height (\B \restriction X)$.
Let $\a = \height (\B)$, so that by Proposition \ref{smoothhtot} $\ot (\B) = \om^\a$.
Since $\B \restriction X$ is a final segment of $\B$ with respect to the lexicographic order we have $\ot (\B \restriction X) = \om^\a$ and hence $\height (\B \restriction X) = \a$.
\end{proof}

The first part of the following lemma is \cite[Lemma 3.2]{MR1219735}, while the second follows from the first and Proposition \ref{smoothhtot}.
	
\begin{lemma}\label{smallerht}		
If $\A$ and $\B$ are fronts with the same base such that $\ot(\A) < \ot(\B)$ then there exist $s \in \A$ and $t \in \B$ such that $s \sqsubset t$.
		
If $\A$ and $\B$ are smooth barriers with the same base such that $\height(\A) < \height(\B)$ then there exist $s \in \A$ and $t \in \B$ such that $s \sqsubset t$.
\end{lemma}
    
The next lemma gives an alternative characterization of smoothness.
	
\begin{lemma}\label{alternativesmooth}
Let $\B$ be a block.
The following are equivalent:
\begin{enumerate}[(1)]
	\item \label{21} $\B$ is smooth;
	\item \label{22} for all $s,t \in T(\B)$ if $|s| = |t|$ and $\height(s) < \height(t)$ then there exists $i < |s|$ such that $s(i) < t(i)$.
\end{enumerate}
\end{lemma}
	
\begin{proof}
First assume \ref{22}.
Let $s,t \in \B$ be such that $|s| < |t|$.
Let $u \sqsubset t$ be its initial segment of length $|s|$. 
Then $\height(u) > 0$ and so by \ref{22} there exists $i < |s|$ such that $s(i) < u(i) = t(i)$ meaning that $\B$ is smooth.
		
Now assume that $\B$ is smooth and fix $s,t \in T(\B)$ of the same length and such that $\height(s) < \height(t)$.
We proceed by induction on $\height(s)$. 
If $\height(s) = 0$ then $\height(t) > 0$ and so $s \in \B$ and $t \in T(\B) \setminus \B$. 
Let $u \in \B$ be an extension of $t$. 
Then $|s| = |t| < |u|$ and so by \ref{21} there exists $i < |s|$ such that $s(i) < u(i) = t(i)$.
Suppose now that $\height(s) = \a > 0$. 
By Lemma \ref{sb tail height} $\height(\B_t) = \height (\B_t \restriction X)$ where $X = \base(\B) \setminus \{0, \ldots, \max(s \cup t)\}$, hence there exists $n > \max (s \cup t)$ such that $\height(t \conc \la n \ra) \ge \a$.
Then $\height(s \conc \la n \ra) < \a \le \height(t \conc \la n \ra)$.
By inductive hypothesis there exists $i < |s \conc \la n \ra|$ such that $s \conc \la n \ra (i) < t \conc \la n \ra (i)$.
Notice that it must be $i < |s|$ because $s \conc \la n \ra (|s|) = n = t \conc \la n \ra (|s|)$.
Therefore $s(i) < t(i)$.    
\end{proof}

We now define largeness notions starting from fronts.
In Section \ref{sec:bridges} we show that these are reasonable generalizations of the classical largeness notions based on systems of fundamental sequences.

\begin{definition}
If $\B$ is a front, we say that $t \in [\base(\B)]^{< \om}$ is \textit{$\B$-large} if there exists $s \in \B$ such that $s \sqsubseteq t$, \textit{$\B$-small} if there exists $s \in \B$ such that $t \sqsubset s$, and \textit{$\B$-size} if $t \in \B$.
\end{definition}

If $\B$ is the degenerate front then all finite sets are $\B$-large and the only $\B$-size set is $\la \ra$.
If $\B = [\N]^k$ then the $\B$-large sets are exactly the finite sets of cardinality at least $k$.
If $\B$ is the Schreier barrier then $\B$-large means $\om$-large in the classical (Paris-Harrington) sense.

Notice that in the context of largeness notions with smooth barriers, the $\oplus$ operation plays essentially the same role as the $\uplus$ operation we defined for largeness notions associated to systems of fundamental sequences.

When we consider barriers, we immediately obtain the analogue of Corollary \ref{intuitive syst}.

\begin{corollary}\label{intuitive bar}
If $\B$ is a barrier, $s \subseteq t$ and $s$ is $\B$-large, then $t$ is $\B$-large too.
If $s \subsetneq t$ and $t$ is $\B$-size, then $s$ is $\B$-small.
\end{corollary}

Given a front $\B$ we can write $\height(\B)$ in Cantor normal form as $\om^{\b_0} + \ldots + \om^{\b_{n-1}}$.
It is however not always the case, even assuming smoothness of $\B$, that we can decompose $\B$ as the sum of smooth barriers $\B_i$ of height $\om^{\b_i}$ for $i < n$. 
It is therefore natural to give the following definition.
	
\begin{definition}
Let $\B$ be a smooth barrier of height (written in Cantor normal form) $\om^{\b_0} + \ldots + \om^{\b_{n-1}}$.
We say that $\B$ is \textit{decomposable} if there exist smooth barriers $\B_0, \dots, \B_{n-1}$ with $\height(\B_i) =\om^{\b_i}$ such that $\B = \B_0 \oplus \ldots \oplus \B_{n-1}$.
The ordered sequence of barriers $\B_0, \ldots, \B_{n-1}$ is called a \textit{decomposition} of $\B$.
We abbreviate smooth decomposable barrier with $\SD$-barrier.
\end{definition}
	
Notice that a smooth barrier with indecomposable height is already a decomposition of itself.

%%%%%%%%%%%%%%%%%%%%%%%%%%%%%
%%%%%%%%%%%%%%%%%%%%%%%%%%%%%
%%DO WE NEED THIS?%%%%%%%%%%%
%%%%%%%%%%%%%%%%%%%%%%%%%%%%%
%%%%%%%%%%%%%%%%%%%%%%%%%%%%%
%\begin{definition}
%A smooth barrier $\B$ is \textit{hereditarily decomposable} if for each $s \in T(\B)$, the smooth barrier $\B_s$ is decomposable.
%\end{definition}

We now describe how to extend the elements of a smooth barrier to obtain a $\SD$-barrier with the same base and height.
This requires working with the terms of the Cantor normal form of the height one by one.
Recall that if $\om^{\b_1}$ is the last term in the Cantor normal form of $\height(\B)$ (which has more than one term) then we can write $\height(\B) = \b_0 + \om^{\b_1}$ with $\b_0 \ggeq \om^{\b_1}$.
	
\begin{definition}
Let $\B$ be a non degenerate front with $\height(\B) = \b_0 + \om^{\b_1}$ where $\b_0 \ggeq \om^{\b_1}$.
The \textit{Lower Part of $\B$} is 
\[
\overline{\B}_1 = \{t \in T(\B) : \height_\B(t) \le \b_0 \wedge \height_\B(t^*) > \b_0\}.
\]
		
Let 
$$T_0 = \bigcup_{t \in \overline{\B}_1} T(\B_t) \cup \{\la n \ra : n \in \base(\B) \setminus \bigcup_{t \in \overline{\B}_1} \base(\B_t)\}.$$
The \textit{Upper Part of $\B$} is the set $\overline{\B}_0$ of leaves of $T_0$.
\end{definition}
	
We remark that $\{\la n \ra : n \in \base(\B) \setminus \bigcup_{t \in \overline{\B}_1} \base(\B_t)\}$ is included in the definition of $T_0$ just to ensure $\base(\overline{\B}_0) = \base(\B)$.
	
\begin{lemma}\label{lowerpart}
Let $\B$ be a smooth barrier with $\height(\B)= \b_0 + \om^{\b_1}$ where $\b_0 \ggeq \om^{\b_1}$.
Then $\overline{\B}_0$ and $\overline{\B}_1$ are smooth barriers with the same base as $\B$, $\height(\overline{\B}_0) = \b_0$, $\height(\overline{\B}_1) = \om^{\b_1}$ and each element of $\overline{\B}_0 \oplus \overline{\B}_1$ extends some element of $\B$.
\end{lemma}
	
\begin{proof}
It is clear that $\overline{\B}_1$ is prefix free.
For each $n \in \base(\B)$, $\la n \ra \in T(\overline{\B}_1)$ since $\la n \ra^* = \la \ra$ and $\height_\B(\la \ra) = \b_0 + \om^{\b_1} > \b_0$.
Thus $\base(\overline{\B}_1) = \base(\B)$.
		
Let $X \subseteq \base(\B)$ be infinite and let $s \in \B$ be such that $s \sqsubset X$. 
Let $u \sqsubseteq s$ be shortest such that $\height_\B(u) \le \b_0$.
Notice that $u \ne \la \ra$ and hence $u \in \overline{\B}_1$.
Thus $\overline{\B}_1$ is a block.
		
Before proving the smoothness of $\overline{\B}_1$ we show that for each $s \in T(\overline{\B}_1)$
$$\height_{\overline{\B}_1}(s) = \height_\B(s) \dotminus \b_0 =
\begin{cases}
	0 & \text{ if } \height_\B(s) < \b_0 \\
	\height_\B(s) - \b_0 & \text{ if }  \height_\B(s) \ge \b_0
\end{cases}$$
Notice that in the second case the Cantor normal form of $\height_\B(s)$ starts with $\b_0$ and so we are simply removing those terms.
We proceed by induction on $\height_{\overline{\B}_1} (s)$.
If $\height_{\overline{\B}_1}(s) = 0$ then $s \in \overline{\B}_1$ and $\height_\B(s) \le \b_0$ so that $\height_\B (s) \dotminus \b_0 = 0$. 
If $\height_{\overline{\B}_1}(s) > 0$ then $s \in T(\overline{\B}_1) \setminus \overline{\B}_1$ and
\begin{align*}
	\height_{\overline{\B}_1}(s) & = \sup_{n \in \base(\B)}(\height_{\overline{\B}_1}(s \conc \la n \ra) + 1) \\
	& = \sup_{n \in \base(\B)} (\height_\B(s \conc \la n \ra) \dotminus \b_0 + 1) \\
	& = \sup_{n \in \base(\B)} (\height_\B(s \conc \la n \ra) +1 \dotminus \b_0) \\
	& = \sup_{n \in \base(\B)} (\height_\B(s \conc \la n \ra) +1) \dotminus \b_0 \\
	& = \height_\B (s) \dotminus \b_0,
\end{align*}
where we are using the above observation about $\height_\B(s) \dotminus \b_0$.
		
In particular we have $\height(\overline{\B}_1) = \height_{\overline{\B}_1} (\la \ra) = \height_\B(\la \ra) \dotminus \b_0 = \om^{\b_1}$.
		
To prove the smoothness of $\overline{\B}_1$ we show $2$ of Lemma \ref{alternativesmooth}.
Let $s,t \in T(\overline{\B}_1)$ be such that $|s| = |t|$ and $\height_{\overline{\B}_1} (s) < \height_{\overline{\B}_1}(t)$.
We claim that $\height_\B (s) < \height_\B(t)$.
Towards a contradiction, suppose that $\height_\B (s) \ge \height_\B(t)$.
If $\height_\B(t) \ge \b_0$ then both Cantor normal forms start with $\b_0$.
When we perform ${\dotminus} \b_0$ we simply remove those terms and get $\height_{\overline{\B}_1} (s) \ge \height_{\overline{\B}_1}(t)$, a contradiction.
If $\height_\B (s) \ge \b_0 > \height_\B(t)$ or $\b_0 > \height_\B (s)$ then $\height_{\overline{\B}_1}(t) = 0$ and so $\height_{\overline{\B}_1} (s) \ge \height_{\overline{\B}_1}(t)$, a contradiction.
Therefore $\height_\B (s) < \height_\B(t)$.
By smoothness of $\B$ and Lemma \ref{alternativesmooth}, we get that there exists $i < |s|$ such that $s(i) < t(i)$ and so $\overline{\B}_1$ is smooth. 
		
We already noticed that $\base(\overline{\B}_0) = \base(\B)$.
We now show that $\overline{\B}_0$ is a block.
Recall that for each $t \in T(\B)$, $\B_t$ is a smooth barrier of height $\height_\B(t)$ and base $\base(\B) \setminus \{0, \ldots, \max t\}$.
Notice that $\bigcup_{t \in \overline{\B}_1} T(\B_t)$ is infinite since there exists $s \in \overline{\B}_1 \setminus \B$ and so for each $n \in \base(\B)$ with $n > \max s$, $\la n \ra \in T(\B_s)$.
Let $X \subseteq \base(\overline{\B}_0)$ be infinite.
If $\la \min X \ra \notin \bigcup_{t \in \overline{\B}_1} T(\B_t)$ then we are done as $\la \min X \ra \in \overline{\B}_0$.
Otherwise, for each $t \in \overline{\B}_1$ with $\max t < \min X$, $X \subseteq \base(\B_t)$.
For each such $t$, there exists a prefix $s_t \in \B_t$ of $X$.
Since there are only finitely many such $t$, the longest $s_t$ belongs to $\overline{\B}_0$.
We conclude that $\overline{\B}_0$ is a block with the same base as $\B$.
		
Next we prove that $\overline{\B}_0$ is smooth.
Let $s,t \in \overline{\B}_0$ with $|s| < |t|$. 
If $s(0) < t(0)$ we are done, so assume that $t(0) \le s(0)$.
Since $|t| > 1$ there exists $u \in \overline{\B}_1$ such that $t \in \B_u$. 
Since $t(0) \le s(0)$ we know that $s \subset \base(\B_u)$. 
Notice that $s$ must be $\B_u$-large since $s \in \overline{\B}_0$ and $\B_u \subset T(\overline{\B}_0)$.
Therefore there exists $s' \sqsubseteq s$ such that $s' \in \B_u$ and by smoothness of $\B_u$ we obtain that there exists $i < |s'|$ such that $s(i) = s'(i) < t(i)$ as required.
		
Our next goal is to show that $\height(\overline{\B}_0) = \b_0$.
Let $t \in T(\B)$ be such that $\height_\B (t) = \b_0$ and notice that $t \in \overline{\B}_1$.
Therefore $\height(\overline{\B}_0) \ge \height (\B_t) = \b_0$. 
Now let $n \in \base(\overline{\B}_0)$.
If $n \notin \bigcup_{t \in \overline{\B}_1} \base(\B_t)$ then $\height_{\overline{\B}_0} (\la n \ra) = 0 < \b_0$.
Otherwise there are finitely many $t \in \overline{\B}_1$ with $\la n \ra \in T(\B_t)$.
Call $\A$ the smooth barrier consisting of the $\sqcup$-union of $\B_t$ for these $t$'s.
Then $T(\A) \subseteq T(\overline{\B}_0)$ and $\height_{\overline{\B}_0} (\la n \ra) = \height_\A (\la n \ra) = \max \{ \height_{\B_t} (\la n \ra) : t \in \overline{\B}_1 \wedge \la n \ra \in T(\B_t) \} < \b_0$ by Remark \ref{heightunionsum} and  definition of $\overline{\B}_1$.
This implies that $\height(\overline{\B}_0) \le \b_0$.
		
We are left to prove that each element of $\overline{\B}_0 \oplus \overline{\B}_1$ extends some element of $\B$.
Let $t \conc s \in \overline{\B}_0 \oplus \overline{\B}_1$ with $t \in \overline{\B}_1$ and $s \in \overline{\B}_0$.
Towards a contradiction, suppose that $t \conc s \in T(\B) \setminus \B$ and let $s'$ be such that $s \sqsubset s'$ and $t \conc s' \in \B$.
Then $s' \in \B_t$ and so some extension of $s'$ belongs to $\overline{\B}_0$.
Since we are assuming $s \in \overline{\B}_0$, this contradicts that $\B_0$ is a block.
\end{proof}
	
Notice that in general it is not the case that $\overline{\B}_0 \oplus \overline{\B}_1 = \B$ as there can be $t_1, t_2 \in \overline{\B}_1$ and $s_1 \sqsubset s_2$ such that $s_1 \in \B_{t_1}$ and $s_2 \in \B_{t_2}$. 
	
\begin{example}
Let $\B$ be a smooth barrier with base $\N$ of height $\om+1$ such that $\la 0 \ra \in \B$ and for each $s \in \B$ if $\min s > 0$ then $|s| > 1$.
Then by definition the Lower Part $\overline{\B}_1$ has height $1$ and must be the smooth barrier $\one$ of singletons on base $\N$.
In particular $\la 0 \ra \in \overline{\B}_1$.
The Upper Part $\overline{\B}_0$ clearly contains sequences with positive minimum.
Then $\la 0 \ra \notin \overline{\B}_0 \oplus \overline{\B}_1$.
\end{example}
	
\begin{corollary}\label{dec}
If $\B$ is a smooth barrier, then there exists a $\SD$-barrier $\B'$ such that $\base(\B) = \base(\B')$, $\height(\B) = \height(\B')$ and $\B \subset T(\B')$.
\end{corollary}

\begin{proof}
We proceed by induction on the number of terms in the Cantor normal form of $\height (\B)$.
If there is only one term the result is trivial.
		
Suppose that the Cantor normal form of $\height (\B)$ has at least two terms. 
By Lemma \ref{lowerpart}, the Cantor normal form of $\height (\overline{\B}_0)$ has one less term than the one of $\height (\B)$.
Therefore, by inductive hypothesis, there exists a $\SD$-barrier $\overline{\B}_0'$ with the same base and same height as $\overline{\B}_0$ and such that $\overline{\B}_0 \subset T(\overline{\B}_0')$.
It is clear that $\B \subset T(\overline{\B}_0 \oplus \overline{\B}_1) \subseteq T(\overline{\B}_0' \oplus \overline{\B}_1)$ and that $\overline{\B}_0' \oplus \overline{\B}_1$ is a $\SD$-barrier as required.
\end{proof}

Notice that by Lemma \ref{comb} and by Corollary \ref{dec}, starting from a front $\B$ we can always produce a $\SD$-barrier $\B'$ such that $\base(\B) = \base(\B')$, $\height(\B) = \height(\B')$ and $\B \subset T(\B')$.

Notice that the only $\SD$-barriers of height $n$ are of the form $[\base (\B)]^n$.
Hence, up to isomorphism, there is only one such a $\SD$-barrier.
	
Recall that there exist smooth barriers of arbitrary countable height.
Let $\b > 0$ be a countable ordinal with Cantor normal form $\om^{\b_0} + \ldots + \om^{\b_n}$.
For each $i \le n$ fix a smooth barrier $\B_i$ of height $\om^{\b_i}$ and assume that for each $i < n$ $\base(\B_{i+1})$ is a final segment of $\base(\B_i)$.
Hence by Remark \ref{heightunionsum} $\B_0 \oplus \ldots \oplus \B_n$ is a $\SD$-barrier of height $\b$.
Therefore there exist $\SD$-barriers of arbitrary countable height.
	
Let $\a,\b$ be ordinals such that $\a$ is strictly less than the first term in the Cantor normal form of $\b$.
Fix smooth barriers $\A$ of height $\a$ and $\B$ of height $\b$.
Then $\A \oplus \B$ is a smooth barrier of height $\a + \b = \b$ which is not isomorphic to $\B$.
It follows that for each ordinal $\b \ge \om$ there are infinitely many non isomorphic smooth barriers of height $\b$.
Therefore there are infinitely many non isomorphic $\SD$-barriers of any infinite height.	

\begin{remark}\label{sd restriction}
If we restrict a $\SD$-barrier $\B$ to a final segment $X$ of its base, we still get a $\SD$-barrier.
This is because smoothness is clearly preserved since we have a subbarrier, height is the same by Lemma \ref{sb tail height} and any decomposition $\B_0, \ldots, \B_n$ of $\B$ yields a decomposition $\B_0 \restriction X, \ldots, \B_n \restriction X$ of $\B \restriction X$.
\end{remark}

\section{Bridges Between Systems and Barriers}\label{sec:bridges}

In this section we show that given a system of fundamental sequences, we can build fronts which yield the same largeness notions.
Moreover, if the system is nested and regular, we get $\SD$-barriers.
Next we show the converse, namely that given a front we can produce something very similar to a system of fundamental sequences.
Again, if the front is a $\SD$-barrier, the induced system satisfies weaker versions of nestedness and regularity.

For the rest of this section, fix a system of fundamental sequences on some ordinal $\G$ and an infinite set $M \subseteq \N$ with $\min M >1$\footnote{the restriction to natural numbers larger than $1$ mirrors our definition of nestedness.}.
If $\a \le \G$, let $\B^\a = [M]^\a$ be the set of $\a$-size subsets of $M$.

\begin{theorem}\label{alphasizebarrier}
For each $\a \le \G$, $\B^\a$ is a front.
Moreover, if the system is nested every $\B^\a$ is a smooth barrier.
\end{theorem}
	
\begin{proof}
We first need to check the three properties of Definition \ref{blockdef}.
If $\a = 0$ then $\B^\a$ is easily seen to be the degenerate front, so suppose $\a > 0$.
\begin{enumerate}[(1)]
	\item Fixed $n \in M$, let $n_0, \ldots, n_{k-1} \in M$ be such that $\a[n, n_0, \ldots, n_{k-1}] = 0$ but $\a[n, n_0, \ldots, n_{k-2}] > 0$. 
	Then $\la n, n_0, \ldots, n_{k-1} \ra$ is $\a$-size.
	Hence $\la n, n_0, \ldots, n_{k-1} \ra \in \B^\a$ and $n \in \base(\B^\a)$, so $\base(\B^\a) = M$.
	\item Let $X \subseteq M$ be infinite and let $j$ be least such that $\a[X_0, \ldots, X_j]=0$.
	Then $\la X_0, \ldots, X_j \ra \in \B^\a$ and $\la X_0, \ldots, X_j \ra \sqsubset X$.
	\item Let $s,t \in \B^\a$ and suppose for sake of contradiction $s \sqsubset t$. 
	Then $s \sqsubseteq t^*$ and so it follows that $\a[t^*] \le \a [s] = 0$, which contradicts the fact that $t$ is $\a$-size.
\end{enumerate}
		
Assume the system of fundamental sequences is nested. 
Since a smooth block is automatically a barrier it suffices to prove smoothness.
Let $s,t \in \B^\a$ be such that $|s| < |t|$. 
Towards a contradiction assume that for each $i < |s|$, $t(i) \le s(i)$.
Since the system of fundamental sequences is nested and $\min t > 1$, $s$ and $t^*$ satisfy the hypothesis of Lemma \ref{subsetlemma}.
Therefore we have that $\a[t^*] \le \a[s] = 0$, which contradicts the fact that $t$ is $\a$-size.
We conclude that $\B^\a$ is a smooth barrier.
\end{proof}

The next lemmas provide properties that $\B^\a$ inherits from the fixed system of fundamental sequences.
%For the rest of the section, Assume the system of fundamental sequences is nested.
	
\begin{lemma}\label{first}
If $t \in T(\B^\a)$ then $\height_{\B^\a}(t) = \a[t]$.
\end{lemma}
	
\begin{proof}
We proceed by induction on $\height_{\B^\a} (t)$.
If $\height_{\B^\a}(t)=0$ then $t \in \B^\a$, $t$ is $\a$-size by definition and so $\a[t]=0$.
If $\height_{\B^\a} (t) > 0$ then $t \in T(\B^\a) \setminus \B^\a$ and
\begin{multline*}
   \height_{\B^\a}(t) = \sup_{n \in M} (\height_{\B^\a}(t \conc \la n \ra) +1) \\
   = \sup_{n \in M} (\a[t \conc \la n \ra] +1) = \sup_{n \in M} (\a[t][n] + 1) = \a[t].\qedhere
\end{multline*}
\end{proof}
	
\begin{corollary}\label{alphaheight}
$\height(\B^\a)=\a$.
\end{corollary}
	
\begin{lemma}\label{second}
For each $t \in T(\B^\a)$, $s \in \B^\a_t$ if and only if $\max t < \min s$ and $s \in \B^{\height_{\B^\a}(t)}$.
\end{lemma}
	
\begin{proof}
If $s \in \B^\a_t$ then $\max t < \min s$ and $t \conc s \in \B^\a$ by definition.
It follows that $0 = \a[t \conc s] = \a[t][s]$ and by Lemma \ref{first} $\height_{\B^\a}(t)[s] = 0$.
In other words, $s$ is $\height_{\B^\a}(t)$-large.
Since $t \conc s$ is $\a$-size, no prefix of $s$ can be $\height_{\B^\a}(t)$-large which means that $s \in \B^{\height_{\B^\a}(t)}$.
		
Conversely, let $s \in \B^{\height_{\B^\a}(t)}$ with $\max t < \min s$ (so that $t \conc s$ does make sense).
Then, again by Lemma \ref{first}, $\a[t \conc s] = \a[t][s] = \height_{\B^\a}(t) [s] = 0$ and so $s \in \B^\a_t$.
\end{proof}
    
\begin{lemma}\label{third}
Assume the system of fundamental sequences is nested and regular.
Let $\a = \b_0 + \om^{\b_1}$ where $\b_0 \ggeq \om^{\b_1}$.
If $t \conc s \in \B^\a$ then the following are equivalent:
\begin{enumerate}[(1)]
    \item \label{31} $\height_{\B^\a}(t) = \b_0$;
    \item \label{32} $t$ is $\om^{\b_1}$-size;
    \item \label{33} $s$ is $\b_0$-size.
\end{enumerate}
\end{lemma}
	
\begin{proof}
Assuming \ref{31}, by Lemma \ref{second} $s \in \B^{\height_{\B^\a}(t)}$ which means that $s$ is $\b_0$-size, so that we have \ref{33}.

If \ref{32} fails then either a proper subset or a proper superset of $s$ is $\b_0$-size, contradicting \ref{33}.

%Moreover, notice that each proper prefix of $t$ is $\om^{\b_1}$-small (otherwise if $t'$ is the $\om^{\b_1}$-size proper prefix of $t$, then $(t \setminus t') \conc s$ is $\b_0$-size, contradicting the fact that $s$ is $\b_0$-size by Corollary \ref{intuitive syst}).
%Thus, by Lemma \ref{first} and regularity of the system of fundamental sequences
%$$\b_0 = \height_{\B^\a}(t) = \a[t] = (\b_0 + \om^{\b_1}) [t] = \b_0 + (\om^{\b_1}[t]).$$
%Therefore $\om^{\b_1}[t]=0$ and $t$ is $\om^{\b_1}$-size and \ref{32} is proved.

Finally, assuming \ref{32}, by Lemma \ref{first} and regularity of the system of fundamental sequences 
$$\height_{\B^\a} (t) = \a[t] = (\b_0 + \om^{\b_1})[t] = \b_0 + (\om^{\b_1}[t]) = \b_0,$$
that is \ref{31}.
\end{proof}
	
%\begin{lemma}\label{fourth}
%Let $\A,\B$ be smooth barriers.
%Suppose that there exists $m$ such that $s \in \B$ if and only if $\min s > m$ and $s \in \A$.
%If $\A$ is decomposable then $\B$ is decomposable.
%\end{lemma}
	
%\begin{proof}
%Since $\A,\B$ are smooth barriers they have the same height, say $\a$.
%If $\a$ is an indecomposable ordinal then the thesis is trivial.

%Suppose $\a$ is not an indecomposable ordinal.
%Let $\A_0, \ldots, \A_{n-1}$ be a decomposition of $\A$. 
%Consider the restriction of each $\A_i$ to $\base(\A_i) \setminus \{0 \ldots m\}$.
%Call $\B_i$ each restricted $\A_i$.
%We claim that $\B_0, \ldots, \B_{n-1}$ is a decomposition of $\B$.
%Clearly for each $i < n$ $\height(\B_i) = \height(\A_i)$ so the $\B_i$'s have the correct heights.
%If $s \in \B$ then $s \in \A$ and so for $i < n$ there exist $t_i \in \A_i$ such that $s = t_{n-1} \conc \ldots \conc t_0$.
%Since $\min s > m$, the same holds for each $t_i$ and so $t_i \in \B_i$ proving that $\B \subseteq \B_0 \oplus \ldots \oplus \B_{n-1}$.		
%If $s \in \B_0 \oplus \ldots \oplus \B_{n-1}$ then it can be written as $t_{n-1} \conc \ldots \conc t_0$ for some $t_i \in \B_i$. 
%By definition, $t_{n-1} \conc \ldots \conc t_0 \in \A$ and since each $\min t_i > m$ the same holds for $\min s$ too.
%Therefore $s \in \B$ and we conclude that $\B$ is decomposable.
%\end{proof}
	
\begin{theorem}\label{regularbarrier}
Assume the system of fundamental sequences is nested and regular.
Then $\B^\a$ is a $\SD$-barrier.
\end{theorem}
	
\begin{proof}
Theorem \ref{alphasizebarrier} yields that $\B^\a$ is a smooth barrier, so we only need to prove that regularity of the system implies decomposability of the barrier.
Let $\om^{\b_0} + \ldots + \om^{\b_n}$ be the Cantor normal form of $\a$.
Then it is easy to check that regularity of the system of fundamental sequences implies $\B^\a = \B^{\om^{\b_0}} \oplus \ldots \oplus \B^{\om^{\b_n}}$.
\end{proof}	
	
Next we show how we can obtain fundamental sequences starting from a barrier.
We recall that if $\B$ is smooth then each $\B_n$ is smooth.
	
\begin{lemma}\label{barrierfundseq}
Let $\B$ be a smooth barrier and $\height(\B) = \a$.
Then $\a[n] = \height(\B_n)$ is a fundamental sequence for $\a$.
\end{lemma}

\begin{proof}
If the sequence is not non decreasing there exist $n < m$ with $\a[m] < \a[n]$, i.e.\ $\height(\B_m) < \height(\B_n)$.
Let $\B_n'$ be the restriction of $\B_n$ to $\base (\B_m)$, which is a smooth barrier with the same height as $\B_n$ by Lemma \ref{sb tail height}.
By Lemma \ref{smallerht} there exist $s \in \B_m$ and $t \in \B_n'$ such that $s \sqsubset t$.
The sequences $\la m \ra \conc s$ and $\la n \ra \conc t$ belong to $\B$ and contradict its smoothness.

Moreover
$$\sup \{\a[n] + 1 : n \in \om\} = \sup \{\height(\B_n) +1 : n \in \om\} = \height(\B) = \a.$$
Therefore $\a[n]$ is a fundamental sequence for $\a$.
\end{proof}
	
Recall that if $T$ is a well-founded tree then for each $\b \le \height(T)$ there exists $s \in T$ such that $\height_T (s) = \b$.
Using this and the fact that if $\B$ is a smooth barrier then $\B_s$ is a smooth barrier for each $s \in T(\B)$, we obtain fundamental sequences for every $\b \le \height(\B)$.
More precisely, for each $s \in T(\B) \setminus \B$ let $\b^\B_s = \height_\B(s) = \height(\B_s)$: then we define\footnote{the definition is a bit involved because we need to define $\b^\B_s[n]$ even when $n \notin \base (\B)$.}
$$\b^\B_s[n] = \b^\B_{s \conc \la m \ra} \text{ where } m = \min \{k \in \base(\B) : k \ge \max (\max s +1, n)\}.$$
Analogously to what we do with systems of fundamental sequences, if $s \in \B$ then we stipulate that $\b^\B_s [n] = 0$ for each $n$.
We will often just write $\b_s$ instead of $\b^\B_s$ when the barrier $\B$ is clear from the context.
Notice that the fundamental sequence for $\b_s$ depends not only on the ordinal $\b_s$ but also on the specific $s \in T(\B)$: in fact it is easy to provide examples of barriers $\B$ such that for some $s \ne t \in T(\B)$ we have $\b_s = \b_t$ but $\b_s[n] \ne \b_t[n]$ for some $n$. 
Therefore we do not have a system of fundamental sequences in the sense of Subsection \ref{subsec:systems}.
		
It is clear that we cannot expect nestedness in the strong sense i.e.\ that for every $s,t \in T(\B)$ and $n > 1$ it does not hold
$$\b_s > \b_t > \b_s[n] > \b_t[n]$$
(it is easy to build a counterexample).
However, it is immediate to check that the following pseudonestedness holds: for each $s, t \in T(\B)$ with $s \sqsubset t$ and each $n \ge t(|s|)$ we have that $\b_s[n] \ge \b_t$.

$\SD$-barriers enjoy the following version of regularity.

\begin{lemma}\label{pseudoregularity}
Let $\B$ be a $\SD$-barrier of height $\om^{\b_0} + \ldots + \om^{\b_n}$.
Let $\B = \B_0 \oplus \ldots \oplus \B_n$ be a decomposition and let $\C = \B_0 \oplus \ldots \oplus \B_{n-1}$ so that $\B = \C \oplus \B_n$ where $\height(\B_n) = \om^{\b_n}$ and $\g = \height(\C) \ggeq \om^{\b_n}$.
Then for each $m$ $\b^\B_{\la \ra} [m] = \g + \b^{\B_n}_{\la \ra} [m]$.
\end{lemma}
		
\begin{proof}
Notice that we just need to prove the statement when $m \in \base (\B) = \base(\B_n)$. 	
For each $s \in \B_n$ we have $\height_\B(s) = \g$.
Therefore for each $t \in T(\B_n)$ it holds that $\height_\B (t) = \g + \height_{\B_n} (t)$.
It follows that $\b^\B_{\la \ra} [m] = \height_\B (\la m \ra) = \g + \height_{\B_n} (\la m \ra) = \g + \b^{\B_n}_{\la \ra} [m]$.
\end{proof}

\section{Ramsey Theorem for Barriers}\label{sec:ramseytheorem}

By a \textit{$k$-coloring} of a set $X$ we mean a function $c \colon X \rightarrow \{0, \ldots, k {-} 1\}$.
Typically $X$ consists of subsets of some finite set $s$.
A \textit{homogeneous set} for $c$ is a set $t \subseteq s$ for which there exists $i < k$ (the \textit{color of $t$}) such that $c$ colors all subsets of $t$ which belong to $X$ with color $i$.

The arrow notation was introduced by Erd{\"o}s and Rado in \cite{MR58687}.
For each $m,n,k,N \in \N$ with $m \ge n$ we write 
$$N \rightarrow (m)^n_k$$
to mean that for each $s \in [\N]^N$ and each $k$-coloring of $[s]^n$  there is a homogeneous subset of $s$ of cardinality $m$.
		
We adapt the arrow notation to our framework. 
If $\C$ is a front let us denote by $[s]^\C$ the collection of all $\C$-size subsets of $s$.
		
\begin{definition}
Let $k \in \N$ and let $\A$, $\B$ and $\C$ be fronts where $\base(\B) = \base(\A)$ is a final segment of $\base(\C)$.
We write
$$\B \rightarrow (\A)^\C_k$$
to mean that for each $(\one \oplus \B)$-size $s$ and each $k$-coloring of $ [s]^\C$, there exists a homogeneous $(\C \oplus \A)$-size subset of $s$.
\end{definition}

We remark that we ask that the homogeneous set is $(\C \oplus \A)$-size instead of just $\A$-size to be sure that there are actually $\C$-size sets to be colored.
We want to avoid the case in which the set is trivially homogeneous just because it does not have $\C$-size subsets.
Similarly the use of $(\one \oplus \B)$-size sets makes the statements and the arguments more natural.
		
As usual in Ramsey theory, the following asymmetric definition is useful.
		
\begin{definition}
Let $k \in \N$ and let $\A_0, \ldots, \A_{k-1}$, $\B$ and $\C$ be fronts where $\base(\B) = \bigcap_{i < k} \base(\A_i)$ and each $\base(\A_i)$ is a final segment of $\base(\C)$.
We write
$$\B \rightarrow (\A_0, \ldots, \A_{k-1})^\C_k$$
to mean that for each $(\one \oplus \B)$-size $s$ and each $k$-coloring of $ [s]^\C$, there exists a homogeneous $(\C \oplus \A_i)$-size subset of $s$ of color $i$ for some $i < k$.
\end{definition}

If $\A$ and $\C$ are fronts such that $\base(\A)$ is a final segment of $\base(\C)$, we say that a $k$-coloring $c \colon [s]^\C \rightarrow k$ is \textit{bad} (relative to $\C, \A$) if it does not contain a homogeneous $(\C \oplus \A)$-size set.
More generally, if $\A_0, \ldots, \A_{k-1}$ are fronts such that each of $\base(\A_i)$ is a final segment of $\base(\C)$, we say that a $k$-coloring $c \colon [s]^\C \rightarrow k$ is \textit{bad} (relative to $\C, \A_0, \ldots, \A_{k-1}$) if it does not contain a homogeneous $(\C \oplus \A_i)$-size set of color $i$ for any $i < k$.

\begin{lemma}\label{inframbarr}
Let $k \in \N$ and let $\A_0, \ldots, \A_{k-1}, \C$ be fronts such that each of $\base(\A_i)$ is a final segment of $\base(\C)$.
Then each infinite $Y \subseteq \bigcap_{i < k} \base(\A_i)$ has an initial segment $s_Y$ such that for each coloring $c \colon [s_Y]^\C \rightarrow k$ there exists a homogeneous $(\C \oplus \A_i)$-size subset of $s_Y$ of color $i$ for some $i < k$.
\end{lemma}
		
\begin{proof}
Nash-Williams proved in \cite{MR173640} that each $k$-coloring $c$ of $[Y]^\C$ has an infinite homogeneous set.
This is essentially the clopen Ramsey theorem \cite{MR337630}.
A standard compactness argument now completes the proof.
%Now, towards a contradiction, assume that each prefix $Y \restriction n$ of $Y$ has a bad $k$-coloring.
%Let $T$ be the set of such colorings ordered with set-theoretic inclusion.
%Then $T$ is an infinite finitely branching tree.
%By K{\"o}nig's Lemma $T$ has a path $\{c_j : j \in \om\}$ and $c = \bigcup_j c_j$ is a $k$-coloring of $[Y]^\C$.
%Let $H$ be an infinite homogeneous set for $c$ say of color $i$ and let $t$ be the $(\C \oplus \A_i)$-size prefix of $H$: then the restriction of $c$ to $Y \restriction {\max t +1}$ is an element of $T$ whose domain contains a $(\C \oplus \A_i)$-size homogeneous set of color $i$, a contradiction.
\end{proof}

Notice that in the previous lemma, if at least one between $\A_i$ and $\C$ is a block, then $s_Y$ cannot be $\la \ra$.

We can now obtain what we call the barrier Ramsey theorem\footnote{beware that the terminology barrier Ramsey theorem is used in \cite{CGLHP} for a different statement.}, which allows us to introduce the Ramsey ordinals in our framework.

\begin{theorem}[Barrier Ramsey Theorem]\label{transramth}
Let $k \in \N$ and let $\A$ and $\C$ be $\SD$-barriers such that $\base(\A)$ is a final segment of $\base(\C)$.
Then there exists a $\SD$-barrier $\B$ with $\base(\B) = \base (\A)$ such that $\B \rightarrow (\A)^\C_k$.
\end{theorem}

\begin{proof}
If $\A$ is the degenerate front then $\B = \C$ works.
If $\C$ is the degenerate front then $\B = \A$ works.
Assume that neither $\A$ nor $\C$ is the degenerate front.
In the notation of Lemma \ref{inframbarr}, let $\A_i = \A$ for each $i < k$ and consider 
$$\mathcal{Y} = \{s \in [\base(\A)]^{< \om} : s=s_Y \text{ for some infinite } Y \subseteq \base(\A) \}.$$
Then $\B' = \{s \in \mathcal{Y} : \forall t \sqsubset s (t \notin \mathcal{Y})\}$ is a block with the same base as $\A$ and satisfies $\B' \rightarrow (\A)^\C_k$.
By Lemma \ref{comb} and Corollary \ref{dec} starting from the block $\B'$ we can build a $\SD$-barrier $\B$ with $\base(\B) = \base(\B')$ and $\height(\B) = \height(\B')$ such that each element of $\B$ is $\B'$-large.
Therefore $\B$ is a $\SD$-barrier with the same base as $\A$ that satisfies $\B \rightarrow (\A)^\C_k$.
\end{proof}

We aim to relate $\height(\B)$ for $\B$ satisfying Theorem \ref{transramth} to $\height(\A)$ and $\height(\C)$.

\begin{definition}\label{ramord}
For $k \in \N$ and countable ordinals $\a$ and $\g$, we define $\Ram(\a)^\gamma_k$ as the least ordinal $\b$ such that for all $\SD$-barriers $\A$ and $\C$ with $\height(\A) = \a$, $\height(\C) = \g$ and $\base(\A)$ a final segment of $\base(\C)$, there exists a $\SD$-barrier $\B$ with $\height(\B) = \b$ and $\base(\B) = \base(\A)$ such that $\B \rightarrow (\A)^\C_k$.
Moreover let \[
\Ram(\a)^\gamma_{< \om} = \sup_{k \in \N} \Ram(\a)^\gamma_k.
\]
\end{definition}
		
Again, we extend our definition to the asymmetric case.
		
\begin{definition}\label{generalramord}
For $k \in \N$ and countable ordinals $\a_0, \ldots, \a_{k-1}, \g$, we define $\Ram(\a_0, \ldots,\a_{k-1} )^\gamma_k$ as the least ordinal $\b$ such that for all $\SD$-barriers $\A_0, \ldots, \A_{k-1}, \C$ with $\height(\A_i)=\a_i$, $\height(\C)= \g$ and each $\base(\A_i)$ a final segment of $\base(\C)$, there exists a $\SD$-barrier $\B$ with $\height(\B) = \b$ and $\base(\B)= \bigcap_{i < k} \base(\A_i)$ such that $\B \rightarrow (\A_0, \ldots, \A_{k-1})^\C_k$.
Moreover let 
$$\Ram(\a_0, \ldots, \a_{k-1})^\gamma_{< \om} = \sup_{k \in \N} \Ram(\a_0, \ldots, \a_{k-1})^\gamma_k.$$
\end{definition}
		
Before stating our main theorem, we need to introduce a logarithm function related to the Veblen hierarchy.

\begin{definition}\label{philog}
Let $\g = \om^{\d_0} + \ldots + \om^{\d_n}$ be an ordinal written in Cantor normal form.
We define the function $\varphi_{\log \g}$ as the composition $\varphi_{\d_0} \circ \ldots \circ \varphi_{\d_n}$.
Moreover let $\varphi_{\log 0}$ be the identity function.
\end{definition}

Fernández-Duque and Joosten introduced and studied the concept of hyperations of normal functions in \cite{MR3037552}.
Our logarithm operator is the particular case of the hyperation of ordinal exponentiation in base $\om$, as noticed in \cite[Corollary 4.10]{MR3037552}\footnote{we thank Fedor Pakhomov for pointing us to \cite{MR3037552}.}.

We are now ready to state the main result of the paper.

\begin{theorem}[Main Theorem]\label{maintheorem}
For any countable ordinals $\a$ and $\g$ such that $\g < \a < \G_\z$ and $\a \ge \om$ we have $\Ram(\a)^{1 + \g}_{< \om} = \varphi_{\log \g} (\a \cdot \om)$.
\end{theorem}

The appearance of $1 + \g$ in the above statement allows us to avoid using different formulas in the finite and infinite case.
%Moreover we remark that we ask $\a \ge \g$ in accordance with Erd{\"o}s and Rado arrow notation where they ask $m \ge n$.
The hypothesis that $\a$ is infinite is necessary;
indeed, when $\a$ (and hence $\g$) is finite, the finite Ramsey theorem always yields $\Ram(\a)^{1 + \g}_{< \om} = \om$.
Finally, notice that we do not get sharp Ramsey ordinals for a fixed number $k$ of colors; 
this is not surprising, as sharp Ramsey ordinals are known only in very few finite cases.

\section{The Barrier Pigeonhole Principle}\label{sec:pigeonhole}

We start by studying the Barrier Pigeonhole Principle and proving the case $\g = 0$ of Theorem \ref{maintheorem}.
We use $\cplus$ and $\ctimes$ to indicate respectively the natural (or Hessenberg) ordinal sum and product.
The following observation about $\cplus$ is useful.

\begin{proposition}\label{hessenbergprop}
Let $\a$ and $\b$ be ordinals such that $\a = \sup_{n \in \N} \a_n$ and $\b = \sup_{n \in \N} \b_n$. 
Then $\a \cplus \b = \max \big(\sup_{n \in \N} (\a_n \cplus \b), \sup_{n \in \N} (\a \cplus \b_n)\big)$.
\end{proposition}

\begin{definition}\label{pigeonfront}
Let $\A_0, \ldots, \A_{k-1}$ be fronts and $Y$ a set such that each $\base(\A_i)$ is a final segment of $Y$ and let $X = \bigcap_{i < k} \base (\A_i)$.
Then we define
$$T(\B) = \{s \in [X]^{< \om} : \text{$s$ has a bad $k$-coloring relative to $\one, \A_0, \ldots, \A_{k-1}$}\}.$$
Let $\B$ be the set of leaves of $T(\B)$ (justifying the name of the initial set).

We sometimes write $T(\B(\A_0, \ldots \A_{k-1}))$ and $\B(\A_0, \ldots \A_{k-1})$ to highlight the dependence on the fronts $\A_0, \ldots, \A_{k-1}$.
\end{definition}

\begin{lemma}\label{pigeonbarrier}
Let $\A_0, \ldots, \A_{k-1}$ be smooth barriers and $Y$ a set such that each $\base(\A_i)$ is a final segment of $Y$.
Then $\B = \B(\A_0, \ldots \A_{k-1})$ of Definition \ref{pigeonfront} is a smooth barrier.
\end{lemma}

\begin{proof}
By Lemma \ref{inframbarr} $T(\B)$ is a well founded tree and hence $\B$ is a front with base $X = \bigcap_{i < k} \base(\A_i)$.
We only need to prove that $\B$ is smooth.
Towards a contradiction, let $s,t \in \B$ with $|s| < |t|$ be such that $t(i) \le s(i)$ for all $i < |s|$ and let $m > \max (s \cup t)$.
Each $k$-coloring of either $s \conc \la m \ra$ or $t \conc \la m \ra$ has a homogeneous $(\one \oplus \A_i)$-size subset of color $i$ for some $i<k$.
Let $c \colon t \rightarrow k$ be a bad $k$-coloring and let $\overline{c} \colon s \conc \la m \ra \rightarrow k$ be defined as $\overline{c}(s \conc \la m \ra (j)) = c(t(j))$ for all $j < |s| +1 \le |t|$.
There exists a homogeneous $(\one \oplus \A_{i_0})$-size set $v_s$ for $\overline{c}$ of color $i_0 < k$.
Extend $c$ to $t \conc \la m \ra$ by setting $c(m) = i_0$.
Then the set $v_t = \{n \in t \conc \la m \ra: c(n) = i_0\}$ is $(\one \oplus \A_{i_0})$-size and $|v_t| \ge |v_s|+1$.
For each $\ell < |v_s|$, either $v_s(\ell) = s(j)$ for some $j < |s|$ and $v_t(\ell) = t(j)$ or $v_s(\ell) = m$ and $v_t(\ell) = t(|s|)$.
Therefore for each $\ell < |v_s|$, $v_t(\ell) \le v_s(\ell)$, contradicting the smoothness of the barrier $\one \oplus \A_{i_0}$.
\end{proof}

\begin{lemma}\label{pigeonheight}
Let $\A_0, \ldots, \A_{k-1}$ be smooth barriers and $Y$ a set such that each $\base(\A_i)$ is a final segment of $Y$. 
Then
\[
\height (\B(\A_0, \ldots, \A_{k-1})) = \height (\A_0) \cplus \ldots \cplus \height (\A_{k-1}).
\]
\end{lemma}

\begin{proof}
Denote $\height (\A_i)$ by $\a_i$, $\B(\A_0, \ldots, \A_{k-1})$ by $\B$ and $\base(\B)$ by $X$. 
We proceed by induction on the finite tuples of ordinals $(\a_0 , \ldots , \a_{k-1})$ using the well founded partial order defined by 
$$(\g_0 , \ldots , \g_{\ell-1}) \le (\d_0 , \ldots , \d_{\ell'-1}) \text{ iff } \ell < \ell' \lor (\ell = \ell' \land \forall i < \ell\, \g_i \le \d_i).$$ 

If $(\a_0, \ldots, \a_{k-1}) = (0, \ldots, 0)$ then every $\A_i$ is the degenerate front and so $\B(\A_0, \ldots, \A_{k-1})$ is the degenerate front.
Hence $\height(\B(\A_0, \ldots, \A_{k-1})) = 0$ as required.

If for some $i<k$ we have $\a_i = 0$ (that is, $\A_i$ is the degenerate front) then each singleton is $(\one \oplus \A_i)$-size and this means that a bad $k$-coloring does not use color $i$.
Therefore $\B(\A_0, \ldots, \A_{k-1}) = \B(\A_0, \ldots, \A_{i-1}, \A_{i+1}, \ldots, \A_{k-1})$ and we can use the induction hypothesis.

If for every $i<k$ we have $\a_i > 0$ then we first show that for each $n \in X$ 
$$T(\B_n) = \bigcup_{j<k} T(\B(\A_0, \ldots, (\A_j)_n, \ldots, \A_{k-1})).$$
Let $s \in T(\B_n)$ namely $\la n \ra \conc s \in T(\B(\A_0, \ldots, \A_{k-1}))$.
There exists a bad $k$-coloring $c$ of $\la n \ra \conc s$.
If $c(n) = j$ then $\la n \ra \conc \{m \in s : c(m) = j\}$ belongs to $T(\A_j)$ and so we get that $s \in T(\B(\A_0, \ldots, (\A_j)_n, \ldots, \A_{k-1}))$.
Conversely, let $s \in T(\B(\A_0, \ldots, (\A_j)_n, \ldots, \A_{k-1}))$ for some $j < k$.
There exists a bad $k$-coloring $c$ of $s$ and in particular the set $\la n \ra \conc \{m \in s : c(m) = j\}$ belongs to $T(\A_j)$.
We know that $\base(\B(\A_0, \ldots, (\A_j)_n, \ldots, \A_{k-1})) = X \setminus \{0, \ldots, n\}$.
Therefore $\la n \ra \conc s \in [X]^{< \om}$ and $\la n \ra \conc s \in T(\B)$ which means that $s \in T(\B_n)$.

It follows that for each $n \in X$
$$\B_n = \bigsqcup_{j<k} \B(\A_0, \ldots, (\A_j)_n, \ldots, \A_{k-1}).$$

We are now ready to compute $\height(\B)$.
For each $n \in X$
\begin{align*}
    \height_{\B} (\la n \ra) & = \height \left(\bigsqcup_{j < k} \B(\A_0, \ldots, (\A_j)_n, \ldots, \A_{k-1})\right) \\
	& = \max_{j < k} (\height (\B(\A_0, \ldots, (\A_j)_n, \ldots, \A_{k-1}))
\end{align*}
where the last equality follows by Remark \ref{heightunionsum}.
By inductive hypothesis $\height_{\B} (\la n \ra) = \max_{j < k} (\a_0 \cplus \ldots \cplus \height((\A_j)_n) \cplus \ldots \cplus \a_{k-1})$.
Hence, using commutativity of the natural sum and Proposition \ref{hessenbergprop}, we obtain
\begin{align*}
	\height(\B) & = \sup_{n \in X} (\max_{j < k} (\a_0 \cplus \ldots \cplus \height((\A_j)_n) \cplus \ldots \cplus \a_{k-1}) +1) \\
	& = \a_0 \cplus \ldots \cplus \a_{k-1}. \qedhere
\end{align*}
\end{proof}

\begin{theorem}\label{generalpigeon}
For all countable ordinals $\a_0, \ldots, \a_{k-1}$
$$\Ram(\a_0, \ldots, \a_{k-1})^1_k = \a_0 \cplus \ldots \cplus \a_{k-1}.$$  
\end{theorem}
		
\begin{proof}
The unique $\SD$-barrier of height $1$ is the singleton barrier $\one$ and we may assume its base is $\N$.
Let $\A_0, \ldots, \A_{k-1}$ be $\SD$-barriers such that each $\base(\A_i)$ is a final segment of $\N$ and with height respectively $\a_0, \ldots, \a_{k-1}$.
Let $\B$ be the set $\B(\A_0, \ldots, \A_{k-1})$ of Definition \ref{pigeonfront}.
By Lemma \ref{pigeonbarrier} $\B$ is a smooth barrier with base $X = \bigcap_{i < k} \base(\A_i)$ and by definition $\B \rightarrow (\A_0, \ldots, \A_{k-1})^{\one}_k$.
Moreover, by Lemma \ref{pigeonheight}, $\height(\B) = \a_0 \cplus \ldots \cplus \a_{k-1}$.
By Corollary \ref{dec} there exists a $\SD$-barrier with same base and same height as $\B$ and such that all of its elements are $\B$-large.
Such $\SD$-barrier shows that $\Ram(\a_0, \ldots, \a_{k-1})^1_k \le \a_0 \cplus \ldots \cplus \a_{k-1}$.

For the other inequality fix $\SD$-barriers $\A_0, \ldots, \A_{k-1}$ with $\base(\A_i) = \N$ and $\height (\A_i) = \a_i$.
We need to prove that if $\D$ is a $\SD$-barrier such that $\base(\D) = \N$ and $\height(\D) < \a_0 \cplus \ldots \cplus \a_{k-1}$, then $\D \nrightarrow (\A_0, \ldots, \A_{k-1})^{\one}_k$. 
Thus we need to find $s \in \D$ and $n \in \N$ with $n > \max s$ such that $s \conc \la n \ra$ has a bad $k$-coloring. 
By Lemma \ref{smallerht} there are $s \in \D$ and $t \in \B (\A_0, \ldots, \A_{k-1})$ such that $s \sqsubset t$. 
By definition of $\B (\A_0, \ldots, \A_{k-1})$, $t$ has a bad $k$-coloring and its restriction to $s \conc \la t(|s|) \ra$ is a bad $k$-coloring.
\end{proof}
		
\begin{corollary}[Barrier Pigeonhole Principle]\label{pigeon}
$\Ram(\a)^1_k = \a \ctimes k$ and therefore $\Ram(\a)^1_{< \om} = \a \cdot \om$.
\end{corollary}
		
\begin{proof}
For the first part apply Theorem \ref{generalpigeon} to the case $\a_0 = \ldots = \a_{k-1} = \a$.
For the second part it suffices to notice that $\sup_{k \in \N} (\a \ctimes k) = \a \cdot \om$.
\end{proof}

\section{The Lower Bound}\label{sec:lowbound}

In this section we establish the lower bound for Theorem \ref{maintheorem}, i.e.\
$$\Ram(\a)^{1 + \g}_{< \om} \ge \varphi_{\log \g}(\a \cdot \om).$$
To prove this we use the nested and regular system of fundamental sequences on $\G_\zeta$ introduced in Definition \ref{inducedsystem}, and the largeness notions associated to this system.

The next theorem is the main result of the section.

\begin{theorem}\label{lowbound}
Let $k > 0$, $\g < \a < \G_\zeta$, $\a \ge \om$ and $\mu < \varphi_{\log \g} (\a \ctimes k)$.
There exists an infinite set $M$ such that for each $s \in [M]^{1 \uplus \mu}$ there is a coloring $c \colon [s]^{1 + \g} \rightarrow k + 3$ with no homogeneous $((1 + \g) \uplus \a)$-size subsets.
\end{theorem}
		
We now show how Theorem \ref{lowbound} yields our lower bound.
Fix $\g < \a < \G_\zeta$ and $\nu < \varphi_{\log \g} (\a \ctimes \om) < \G_\zeta$ ordinals.
Pick $k \in \N$ and $\mu$ an ordinal such that $\nu < \mu < \varphi_{\log \g} (\a \ctimes k)$.
Apply Theorem \ref{lowbound} to $k, \g, \a, \mu$ to get an infinite set $M$ (we may assume $\min M > 1$).
Now let $\A = [M]^\a$, $\C = [M]^{1 + \g}$ and $\D = [M]^{1 \uplus \mu}$.
Notice that $\A$ and $\C$ do not depend on $\nu$.
By Theorem \ref{regularbarrier} $\A$, $\C$ and $\D$ are $\SD$-barrier.
Let $\B$ be any $\SD$-barrier of height $\nu$ with $\base(\B) = M$. 
Since $\height(\B) = \nu < 1 + \mu = \height(\D)$, by Lemma \ref{smallerht} there are $s \in \B$ and $t \in \D$ such that $s \sqsubset t$.
Then $s \conc \la t(|s|) \ra$ inherits a bad coloring of its $(1 + \g)$-size sets with $k + 3$ colors from $t$ and from the fact that $t \in \D$.
This means that $\B \nrightarrow (\A)^\C_{k+3}$ and, since $\B$ is a generic $\SD$-barrier of height $\nu$ with $\base(\B) = \base(\A)$, we conclude.
		
Therefore we only need to prove Theorem \ref{lowbound}.
For the rest of the section, fix $\a, \g, \mu$ and $k$ as in the statement of the theorem.
We need to introduce some machinery.
We start defining a coloring on tuples of ordinals and the first step is to define the \textit{overline function}.
		
\begin{definition}\label{overline def}
Given ordinals $\b > \d$, let $\overline{\b,\d}$ be the largest exponent in the Cantor normal form of $\b$ which differs from the corresponding term in the Cantor normal form of $\d$.
That is, if the Cantor normal forms are
\begin{align*}
	\b = & \,\, \om^{\b_0} + \ldots + \om^{\b_l} + \om^{\b_{l+1}} + \ldots + \om^{\b_n} \\
	\d = & \,\, \om^{\b_0} + \ldots + \om^{\b_l} + \om^{\d_{l+1}} + \ldots + \om^{\d_m}
\end{align*}
and $\b_{l+1} > \d_{l+1}$, then $\overline{\b,\d}=\b_{l+1}$.
If $\b \le \d$, we let $\overline{\b,\d} = 0$.
\end{definition}
		
We now need to climb the Veblen hierarchy, and for that we introduce the so called \textit{peeling functions}, which are length preserving functions on finite sequences of ordinals.
We denote finite sequences of ordinals with upper case letters from the beginning of the alphabet.
We reserve as usual lower case letters from the second half of the alphabet for finite sequences of numbers.

We also introduce the concept of the collection of subterms of an ordinal.
To do that we use the notion of multiset, i.e.\ a set in which elements can occur multiple times.
Recall that if $A$ and $B$ are multisets then $A+B$ is the multiset where each element has multiplicity the sum of the its multiplicities in $A$ and $B$.

\begin{definition}
We recursively define the multiset $\Sub (\b)$ of subterms of an ordinal $\b$ as follows.
\begin{itemize}
    \item If $\b = 0$, let $\Sub(\b) = \{0\}$.
    \item If $\b = \varphi_{\b_0} (\b_1)$ with $\b_1 < \b$, let $\Sub(\b) = \{\b\} + \Sub(\b_1)$.
    \item If $\b = \om^{\b_0} + \ldots + \om^{\b_n}$ with $n > 0$ and $\b_0 \ge \ldots \ge \b_n$, let $\Sub(\b) = \{\b\} + \Sub(\om^{\b_0}) + \ldots + \Sub(\om^{\b_n})$.
\end{itemize}
\end{definition}

Notice that if $\a \in \Sub(\b)$ and $\b \in \Sub(\g)$ then $\a \in \Sub(\g)$.
Moreover, the multiplicity of $\a$ in the multiset $\Sub(\g)$ is at least that of $\a$ in $\Sub(\b)$.
		
\begin{definition}\label{peeling def}
We define functions $\pbar_\d : \ON^{<\om} \rightarrow \ON^{<\om}$ by induction on $\d$ as follows.
Let $\pbar_0$ be the identity function, and let
$$\pbar_1 (\b_0,\b_1,\dots,\b_l) = (\overline{\b_0,\b_1}, \overline{\b_1,\b_2}, \ldots, \overline{\b_l,0}).$$
For ordinals of the form $\rho + \om^\d$ with $\rho \ggeq \om^\d$, we let
$$\pbar_{\rho + \om^\d} = \pbar_{\om^\d} \circ \pbar_\rho.$$
On infinite ordinals of the form $\om^\d$, we first define
$$\pbar_{< \om^\d}(A) = \lim_{\rho \rightarrow \om^\d} \pbar_\rho (A).$$
We prove in Lemma \ref{basic peeling} below that for each ordinal $\nu \le \rho$ and each $i < |A|$, $\pbar_\rho (A) \in \Sub (\pbar_\nu (A))$.
Therefore, $\pbar_\rho(A)$ is coordinate wise non increasing as a function of $\rho$, and the limit above exists.
Furthermore, we see in Lemma \ref{basic peeling} that each ordinal in the tuple $\pbar_{< \om^\d} (A)$ is either $0$ or a fixed point of $\varphi_{\d'}$ for every $\d' < \d$ and thus belongs to the image of $\varphi_\d$.
Then we define $\pbar_{\om^\d} (A)$ by peeling off one application of $\varphi_\d$ from each nonzero entry in $\pbar_{<\om^\d} (A)$.
In other words
$$\pbar_{\om^\d} (A) = \varphi_\d^{-1} \circ \pbar_{<\om^\d} (A)$$
where we are applying $\varphi_\d^{-1}$ to each nonzero entry of $\pbar_{<\om^\d}(A)$.

We write $p_\d (A)$ (without the bar) for the first element of $\pbar_\d (A)$.
\end{definition}
		
Notice that to be consistent with the terminology we may actually say that $\pbar_1$ peels off one application of $\varphi_0$.
This is appropriate since $\pbar_1 = \pbar_{\om^0}$.
		
\begin{lemma}\label{basic peeling}
The following properties of the peeling functions hold.
\begin{enumerate}[(1)]
	\item \label{61} For every $\rho$ and $\nu < \rho$ we have that $\pbar_\rho (A) (i) \in \Sub(\pbar_\nu (A) (i))$ for each $i < |A|$.
	\item \label{62} Each entry of $\pbar_{<\om^\d}(A)$ is either $0$ or a fixed point of $\varphi_{\d'}$ for every $\d' < \d$.
	\item \label{63} If $A (i) < \varphi_\d(\b)$ for all $i < |A|$, then for all $i < |A|$ such that $\pbar_{\om^\d} (A) (i) > 0$ we have $\pbar_{\om^\d} (A) (i) < \b$.
	\item \label{64} If $A(i) < \varphi_{\log \d}(\b)$ for all $i < |A|$, then for all $i < |A|$ such that $\pbar_\d(A)(i) > 0$ we have $\pbar_\d(A)(i) < \b$.
\end{enumerate}
\end{lemma}

The conclusions of \ref{63} and \ref{64} are stated as implications to include the case $\b = 0$.

\begin{proof}
We prove \ref{61} by induction on $\rho$.
For $\rho=0$ the statement is trivial, while for $\rho=1$ it suffices to notice that $\pbar_1$ outputs for each entry either $0$ or the exponent of one of its Cantor normal form terms, which is a subterm of the corresponding entry of $A = \pbar_0 (A)$.

When $\rho = \om^\d$, we know by inductive hypothesis that for each $\nu < \mu < \om^\d$, each entry of $\pbar_\mu (A)$ is a subterm of the corresponding entry of $\pbar_\nu (A)$.
Hence the limit $\pbar_{< \om^\d} (A)$ is well defined, and each of its entries is a subterm of the corresponding entry of $\pbar_\nu (A)$ for each $\nu < \om^\d$.
By definition of $\pbar_{\om^\d}$ we peel off one application of $\varphi_\d$ from each nonzero entry of $\pbar_{< \om^\d} (A)$.
It follows that each entry of $\pbar_{\om^\d} (A)$ is a subterm of the corresponding entry of $\pbar_{< \om^\d} (A)$ and so it is also a subterm of the corresponding entry of $\pbar_\nu (A)$ for each $\nu < \om^\d$ as required.

Finally, consider an ordinal of the form $\rho = \mu +\om^\d$ for $\mu \ggeq \om^\d > 1$ and let $\nu < \mu + \om^\d$.
In this case $\pbar_{\mu + \om^\d} (A) = \pbar_{\om^\d} (\pbar_\mu (A))$ by definition.
If $\nu \le \mu$ then by inductive hypothesis each entry of $\pbar_\mu (A)$ is a subterm of the corresponding entry of $\pbar_\nu (A)$.
Then since $\om^\d < \mu + \om^\d$ again by inductive hypothesis we know that each entry of $\pbar_{\om^\d} (\pbar_\mu (A))$ is a subterm of the corresponding entry of $\pbar_\mu (A)$.
We are left with the case $\mu < \nu < \mu + \om^\d$ which means that $\nu = \mu + \sigma$ for some $\sigma$ such that $\mu \ggeq \sigma$ and $0 < \sigma < \om^\d$.
Then $\pbar_\nu (A) = \pbar_\sigma (\pbar_\mu (A))$ by definition.
Since $\sigma < \om^\d < \mu + \om^\d$ by inductive hypothesis we know that each entry of $\pbar_{\om^\d} (\pbar_\mu (A))$ is a subterm of the corresponding entry of $\pbar_\sigma (\pbar_\mu (A))$.\smallskip
			
For \ref{62} we first claim that $\pbar_{<\om^\d} (A)$ is a fixed point of $\pbar_{\om^{\d'}}$ for all $\d'<\d$.
Let $\rho <\om^\d$ such that each entry in $\pbar_{<\om^\d}(A)$ has stabilized at $\pbar_\rho (A)$ ($\rho$ exists by $1$) and pick $\rho'$ with $\rho \le \rho' < \om^\d$ such that $\rho' \ggeq \om^{\d'}$.
Then, we have that
$$\pbar_{\om^{\d'}} (\pbar_{<\om^\d} (A)) = \pbar_{\om^{\d'}} (\pbar_{\rho'} (A)) =
\pbar_{\rho' + \om^{\d'}} (A) =  \pbar_{<\om^\d} (A).$$
This completes the proof of the claim.

Considering the case $\d'=0$ we see that $\pbar_{<\om^\d} (A)$ is a fixed
point of $\pbar_1$, and this can happen only if the Cantor normal form of each entry of the ordinals in $\pbar_{< \om^\d} (A)$ contains only one term and that term is either $0$ or a fixed point of the exponential function $\varphi_0$.
Considering larger $\d'$, we see that no entry of $\pbar_{<\om^\d} (A)$ can be of the form $\varphi_{\d'} (\b)$ with $\b < \varphi_{\d'}(\b)$, as otherwise that last application of $\varphi_{\d'}$ would have been peeled off at some previous stage.
To see this notice that for each $n$, $\om^{\d'} \cdot n < \om^\d$ and each time you reach $\pbar_{\om^{\d'} \cdot n}$ in the recursive definition of the peeling functions, you peel off at least $n$ applications of $\varphi_{\d'}$.
Therefore, before stage $\om^\d$ there will be no more applications of $\varphi_{\d'}$.
All the entries that did not go all the way down to $0$ must then be fixed points of $\varphi_{\d'}$ for all $\d' < \d$ and part \ref{62} is proved.\smallskip
			
Part \ref{63} follows from the second, as the only ordinals smaller than $\varphi_\d (\b)$ that are invariant under the operations $\varphi_{\d'}$ for	all $\d' < \d$ are those of the form $\varphi_\d (\b')$ for some $\b' < \b$.
Therefore, if $\pbar_{<\om^\d} (A) (i) = \varphi_\d (\b')$, then $\pbar_{\om^\d} (A) (i) = \b' < \b$.\smallskip
			
Part \ref{64} is obtained by iterating the third one.
\end{proof}
		
An important observation about the peeling functions is that each entry of $\pbar_\d(A)$ does not depend on the previous entries of the finite sequence $A$.
In other words $\pbar_\d (A^-) = (\pbar_\d (A))^-$ and $\pbar_\d (A^{-n}) = (\pbar_\d (A))^{-n}$ where $A^-$ denotes the finite sequence $A$ without its first element, and $A^{-(n+1)} = (A^{-n})^-$.

Before defining the coloring promised before Definition \ref{overline def}, we still need two ingredients.
First fix a coloring $d : \a \ctimes k \rightarrow k$ such that for each $i < k$ the preimage $d^{-1} (i)$ has order type $\a$.
For each $i < k$ let $\pi_i : d^{-1}(i) \rightarrow \a$ be the order isomorphism.
Second, given a tuple $A$ of ordinals below $\varphi_{\log \g} (\a \ctimes k)$, if $p_\g(A) \le p_\g(A^-)$, let $\z_A \le \g$ be the least ordinal $\z$ such that $p_\z(A) \le p_\z(A^-)$.
As $p_\z(A^-)$ is the second entry of $\pbar_\z(A)$, we have $p_{\z_A + 1} (A) = 0$.
If $p_\g (A) > p_\g (A^-)$ then $\z_A$ does not exist and in this case we have $0 < p_\g (A) < \a \ctimes k$ by Lemma \ref{basic peeling}.
		
\begin{definition}\label{def: coloring on ord}
Let $c$ be the following $(k+3)$-coloring of the finite tuples of ordinals below $\varphi_{\log \g} (\a \ctimes k)$.
\begin{itemize}
	\item If $\z_A$ does not exist, let $c(A) = d(p_\g (A)) < k$.
	\item If $\z_A$ and $\z_{A^-}$ both exist and $\z_A > \z_{A^-}$, let $c(A) = k$.
	\item If $\z_A$ and $\z_{A^-}$ both exist and $\z_A=\z_{A^-}$, let $c(A) = k+1$.
	\item If $\z_A$ exists and either $\z_{A^-}$ does not or $\z_A < \z_{A^-}$, let $c(A) = k+2$.
\end{itemize}
\end{definition}

The strategy to prove Theorem \ref{lowbound} consists in defining a coloring $\overline{c}$ on numbers starting from the coloring $c$ on ordinals and show that homogeneous sets for $\overline{c}$ induce a descending sequence in some ordinal, which depends on the color of the homogeneous set.
A homogeneous set of color $i < k$ induces a descending sequence in $\a$.
A homogeneous set of color $k$ induces a descending sequence in $\g + 1$.
A homogeneous set of color either $k + 1$ or $k + 2$ induces a descending sequence of subterms of the first element of the sequence.
This imposes a limit on how large the homogeneous sets can be.

Recall that we fixed an ordinal $\mu < \varphi_{\log \g} (\a \ctimes k)$ and that our final goal is to define an infinite set $M$ such that each $s \in [M]^{1 \uplus \mu}$ has a bad coloring on its $(1 + \g)$-size sets with $k + 3$ colors. 
The set $M$ is a rapidly increasing sequence of natural numbers where the norm function behaves in a particular way.
For that we need to define a new norm.

Before defining the new norm, we define a function that assigns to each ordinal $\tau < \varphi_{\log \g}(\a \ctimes k)$, a finite set of ordinals $S(\tau)$ that contains all the ordinals which may be equal to $\z_A$ for some tuple of ordinals $A$ with $A(0) = \tau$.
		
\begin{definition}
We recursively define the set $S(\tau)$ as follows.
\begin{itemize}
	\item If $\tau=0$, let $S(\tau) = \{0\}$.
    \item If $\tau = \varphi_\d (0)$, let $S(\tau) = \{0, 1, \om^\d\}$.
	\item If $\tau = \varphi_\d (\sigma)$ with $0 < \sigma < \tau$, let $S(\tau) = \{0, 1\} \cup	\{\om^\d + \xi: \xi \in S(\sigma) \wedge \xi > 0\}$.
	\item If $\tau = \om^{\sigma_0} + \ldots + \om^{\sigma_n}$ with $n > 0$ and $\sigma_0 \ge \ldots \ge \sigma_n$, let $S(\tau) = \{0,1\} \cup S(\om^{\sigma_0}) \cup \ldots \cup S(\om^{\sigma_n})$.
\end{itemize}
\end{definition}

\begin{lemma}\label{lemma: Stau}
Let $\tau < \varphi_{\log \g} (\a \ctimes k)$.
Then for each tuple of ordinals $A \subseteq \varphi_{\log \g} (\a \ctimes k)$ with $A(0) = \tau$, if $\z_A$ exists then $\z_A \in S(\tau)$.
\end{lemma}
		
\begin{proof}
Recall that $\z_A$ is the least $\z \le \g$ such that $p_\z(A^-) \ge p_\z(A)$.
Because of the minimality of $\z_A$, we have that if $\z_A > 0$ then $p_{< \z_A} (A) > p_{< \z_A} (A^-)$, where in the case when $\nu$ is a successor, we use $p_{< \nu} (A)$ to mean $p_{\nu - 1} (A)$.
Since $p_\nu (A^-)$ is non increasing with $\nu$, we must have $p_{< \z_A} (A^-) \ge p_{\z_A} (A^-)$.
Putting these three inequalities together, we get
$$p_{< \z_A} (A) > p_{\z_A} (A).$$
There are two ways by which one could have $p_{< \nu} (A) > p_\nu (A) $ when $\nu > 0$: 
\begin{enumerate}[(a)]
    \item a Veblen function was peeled off from $p_{< \nu} (A)$ (possibly $\varphi_0$, when $\nu$ is successor and we apply $\pbar_1$);
    \item $\nu$ is successor and $p_{\nu - 1} (A) \le p_{\nu-1} (A^-)$.
\end{enumerate}
Notice the reason for having $p_{< \z_A} (A) > p_{\z_A} (A)$ cannot be (b), because in this case $\z_A$ would have been smaller.
It follows that	$\z_A$ must be an ordinal at which some peeling was done to $p_{< \z_A} (A)$: either an $\om$ was removed (which corresponds to peeling off $\varphi_0$), or a $\varphi_{\delta}$ was removed.

The proof is now by induction on $\tau$.
Notice that $0 \in S(\tau)$ for every $\tau$, which takes care of the case when $\tau = A(0) \le A(1)$.
So we can assume $A(0) > A(1)$.
			
%Suppose first that $\tau = 1 = \varphi_0 (0)$ and $A(0) > A(1)$.
%Then $A(1) = 0$, in which case $p_1 (A) = 0 = p_1 (A^-)$.
%So $\z_A = 1 \in S(1)$.

Suppose $\tau = \varphi_0 (\sigma)$ with $\sigma < \tau$.
Then $p_1 (A) = \sigma$ and $\z_A = 1 + \z_B$ where $B = \pbar_1 (A)$.
By inductive hypothesis $\z_A \in S(\tau)$.

Suppose now that $\tau = \varphi_\d (\sigma)$ with $\d>0$ and $\sigma < \tau$.
Then, for all $\nu < \om^\d$, $p_\nu (A) = \tau$ and $p_{\om^\d} (A) = \sigma < \tau$.
Therefore, either $\z_A = \om^\d$ or $\z_A = \om^\d + \z_B$ with $\z_B > 0$ where $B = \pbar_{\om^\d} (A)$.
If $\sigma = 0$ then $\z_A = \om^\d \in S(\varphi_\d(0))$ and moreover $p_{< \om^\d} (A^-)=0$ must hold otherwise it belongs to $\ran \varphi_\d$ contradicting the minimality of $\z_A$.
If instead $\sigma > 0$ notice that $p_{\om^\d} (A) = \varphi_\d^{-1} (\tau) > \varphi_\d^{-1} (p_{< \om^\d} (A^-)) = p_{\om^\d} (A^-)$ because either $p_{< \om^\d} (A^-)$ belongs to the range of $\varphi_\d$ (which is strictly increasing) or it is $0$.
It follows that by inductive hypothesis $\z_A = \om^\d + \z_B$ with $\z_B \in S(\sigma)$ and $\z_B>0$; hence $\z_A \in S(\tau)$.
			
Finally, suppose that $A = (\tau, \tau_1, \ldots, \tau_m)$ and that $\tau = \om^{\sigma_0} + \ldots + \om^{\sigma_n}$ with $n > 0$.
Then $p_1(A)$ is one of the $\sigma_i$'s, say $\sigma_{i_0}$.
If $p_1(A) \le p_1(A^-)$, then $\z_A = 1 \in S(\tau)$.
So suppose that $p_1(A) > p_1(A^-)$ and $\z_A \ge 2$.
The goal now is to define a tuple $A'$ of ordinals such that $A'(0) = \om^{\sigma_{i_0}}$ and $\z_{A'} = \z_A$. Then, by the inductive hypothesis we would have $\z_{A'} \in S(\om^{\sigma_{i_0}})$, and so by definition of $S(\tau)$, $\z_A \in S(\tau)$.	
Let $\d_0 = \sigma_{i_0}$ so that, for some $\d_1, \dots, \d_m$, we have
$$\pbar_1(A) = (\overline{\tau,\tau_1}, \overline{\tau_1,\tau_2}, \ldots, \overline{\tau_m,0}) = (\d_0, \d_1, \ldots, \d_m).$$
Suppose first that $\d_0 > \d_1 > \ldots > \d_m$.
Let $A'$ be defined by applying $\varphi_0$, which is the exponential function, to all the entries of $\pbar_1(A)$.
That is, $A' = (\om^{\d_0}, \om^{\d_1}, \ldots, \om^{\d_m})$. Notice that, for all $i < m$,
$$\pbar_1 (A')(i) = \overline{\om^{\d_i}, \om^{\d_{i+1}}} = \d_i = \pbar_1 (A)(i).$$
Therefore $\pbar_1 (A') = \pbar_1 (A)$ and hence $\z_{A'} = \z_A$.
Suppose now that $0 < \ell < m$ is least such that $\d_{\ell} \le \d_{\ell + 1}$.
We have that		
$$
\begin{array}{rcccccccl}
A &=& (\tau, & \ldots, & \tau_{\ell-1}, & \tau_\ell, & \tau_{\ell + 1}, & \ldots, & \tau_m) \\
\pbar_1 (A) & = & (\d_0, & \ldots, & \d_{\ell - 1}, & \d_\ell, & \d_{\ell + 1}, & \ldots, & \d_m) \\
\pbar_2 (A) & = & (\overline{\d_0, \d_1}, & \ldots, & \overline{\d_{\ell - 1}, \d_\ell}, & 0, & \overline{\d_{\ell + 1}, \d_{\ell + 2}}, & \ldots, & \overline{\d_m, 0})
\end{array}
$$
Define $A'$ by letting $A'(i) = \om^{\d_i}$ for all $i < \ell$,
$A'(\ell) = \om^{\d_\ell} + \om^{\d_\ell}$, $A'(\ell + 1) = \om^{\d_\ell}$ and $A'(i) = 0$ for all $i > \ell + 1$.
We then have that
$$
\begin{array}{rccccccccl}
A' & = & (\om^{\d_0}, & \ldots, & \om^{\d_{\ell - 1}}, & \om^{\d_\ell} + \om^{\d_\ell}, & \om^{\d_\ell}, & 0, & \ldots, & 0) \\
\pbar_1(A') & = & (\d_0, & \ldots, & \d_{\ell - 1}, & \d_\ell, & \d_\ell, & 0, & \ldots, & 0) \\
\pbar_2 (A') & = & (\overline{\d_0, \d_1}, & \ldots, & \overline{\d_{\ell - 1}, \d_\ell}, & 0, & \overline{\d_{\ell}, 0}, & 0, & \ldots, & 0)
\end{array}
$$	
Observe that $\pbar_2 (A')(\ell) = 0 = \pbar_2 (A)(\ell)$ and hence that $\pbar_\nu (A')(\ell) = 0 = \pbar_\nu (A)(\ell)$ for all $\nu \ge 2$.
Observe also that $\pbar_2 (A')(i) = \pbar_2(A) (i)$ for all $i \le \ell$.
It is not hard to see that each application of $\pbar_1$ keeps equal the two sequences up to the $\ell$-th entry while each application of $\pbar_{\om^\nu}$ for some $\nu$ just peels off (when possible) an application of $\varphi_\nu$ from each entry independently on the other entries of the sequence.
Therefore we get that $\pbar_\nu (A')(i) = \pbar_\nu (A)(i)$ for all $i \le \ell$ and all $\nu \ge 2$.
It follows that $\z_{A'} = \z_A$ as needed.
\end{proof}

\begin{remark}\label{rem:limitcase}
Notice that if $\z_A = \om^{\d_0}+\ldots+\om^{\d_n}$ is a limit ordinal, then $p_{< \z_A} (A) = \varphi_{\d_n} (0)$ and $p_{< \z_A} (A^-) = 0$.
To see this, let $B = \pbar_{< \z_A} (A) = \pbar_{< \om^{\d_n}} \circ \pbar_{\om^{\d_{n-1}}} \circ \ldots \circ \pbar_{\om^{\d_0}} (A)$ so that $\pbar_{\z_A} (A) = \varphi_{\d_n}^{-1} (B)$.
By definition of $\z_A$ it must be $B(0) > B(1)$ and by definition of $\pbar_{< \om^{\d_n}}$ both $B(0)$ and $B(1)$ must be either $0$ or in the range of $\varphi_{\d_n}$.
If $B(1) > 0$ then since $\varphi_{\d_n}$ is strictly increasing it must be $\varphi_{\d_n}^{-1} (B(0)) > \varphi_{\d_n}^{-1} (B(1))$, contradicting the definition of $\z_A$.
Then $B(1) = 0$ and $\varphi_{\d_n}^{-1} (B(0)) \le \varphi_{\d_n}^{-1} (B(1)) = 0$ which implies $B(0) = \varphi_{\d_n} (0)$.
\end{remark}
		
We defined the norm $|\d|$ of an ordinal $\d$ in Definition \ref{normdefinition}. 
We now use that norm to define a new function $||\cdot|| : \varphi_{\log \g} (\a \ctimes k) \rightarrow \om$ as follows.
		
\begin{definition}\label{doublenorm}
For each ordinal $\b < \varphi_{\log \g} (\a \ctimes k)$, let $||\b||$ be $1$ plus the maximum of:
\begin{itemize}
	\item the norm of all the subterms of $\b$;
	\item the cardinality of the multiset $\Sub (\b)$;
	\item the norm of all the ordinals in $S(\b)$;
    \item for every $\d \in \Sub(\b)$ such that $\d < \a \ctimes k$, the norm of $\pi_{d(\d)} (\d)$.
\end{itemize}
\end{definition}

We now show that we just need to know the first entry of a finite sequence of ordinals $A$ to compute a level such that $p_{< \om^\sigma} (A)$ stabilizes.
		
\begin{lemma}\label{convergence lemma}
Let $\sigma > 0$ be an ordinal and let $A = (\a_0, \ldots, \a_m)$ be a sequence of ordinals in $\varphi_{\log \g} (\a \ctimes k)$.
Then, 
$$p_{< \om^\sigma} (A) = p_{\om^{\sigma[||\a_0||]} \cdot ||\a_0||}(A).$$ 
\end{lemma}
		
\begin{proof}
First we claim that for some $\xi < \om^{\sigma[||\a_0||]} \cdot ||\a_0||$ we have that $p_\xi (A)$ is either $0$ or a fixed point of $\varphi_{\sigma[||\a_0||]}$.
To prove this, notice that by Lemma \ref{basic peeling} we know that each application of $p_{\om^{\sigma[||\a_0||]}}$ peels off at least one application of $\varphi_{\sigma[||\a_0||]}$ to each entry of its input, unless we have already got to $0$ or to a fixed point of $\varphi_{\sigma[||\a_0||]}$.
The function $p_{\om^{\sigma [||\a_0||]} \cdot ||\a_0||}$ peels off at least $||\a_0||$ applications of $\varphi_{\sigma[||\a_0||]}$.
However, by definition of $||\a_0||$ the number of operations $\varphi_{\sigma[||\a_0||]}$ occurring in $\a_0$ is strictly less than $||\a_0||$.
For this reason $p_\xi (A)$ must be $0$ or a fixed point of $\varphi_{\sigma [||\a_0||]}$ for some $\xi < \om^{\sigma[||\a_0||]} \cdot ||\a_0||$ as claimed.

If $p_\xi (A) = 0$ then the sequence has already stabilized.
Otherwise $p_\xi (A)$ is a fixed point of $\varphi_{\sigma [||\a_0||]}$ and can be uniquely written as $\varphi_\d (\b)$ for some $\d > \sigma[||\a_0||]$ and some $\b < \varphi_\d (\b)$.
Since $p_\xi (A)$ is a subterm of $\a_0$ (by Lemma \ref{basic peeling}) we get that $| \d | \le || \a_0 ||$.
Then, if $\d < \sigma$, the goodness of the norm $| \cdot |$ implies that $\sigma \Rightarrow_{||\a_0||} \d$ by Lemma \ref{lem: rm above norm}.
It follows that $\sigma[||\a_0||] \ge \d$ which is a contradiction.
We conclude that $\d \ge \sigma$, and hence $p_\xi (A) \in \ran (\varphi_\sigma)$.

If $p_\xi (A) \le p_\xi (A^-)$ then $p_{\xi + 1} (A) =0$ and hence the sequence stabilizes at $\xi+1$.
If instead $p_\xi (A) > p_\xi (A^-)$ then by induction we can prove that $p_\rho (A) = p_\xi (A) > p_\rho (A^-)$ for any $\xi \le \rho < \om^\sigma$: this is because $p_\rho (A^-) \le p_\xi (A^-)$ and at stage $\rho$ we attempt to peel off $\varphi_{\sigma'}$ for some $\sigma' < \sigma$, which has no effect on $p_\xi (A)$.

Since $\xi+1 \le \om^{\sigma[||\a_0||]} \cdot ||\a_0||$, we have $p_{< \om^\sigma} (A) = p_{\om^{\sigma[||\a_0||]} \cdot ||\a_0||}(A)$.
\end{proof}

We are ready to construct the set $M$ of Theorem \ref{lowbound}.
		
\begin{lemma}\label{infinite set M}
There exists an infinite set $M$ such that, for all $s \in [M]^{< \om}$, we have that $\max (||\mu[s \restriction i]||, ||\om||, ||\g + 1||) + 2 < s(i)$ for all $i < |s|$.
\end{lemma}

\begin{proof}
We define $M(n)$ by recursion on $n$.
Let $M(0)$ be any number strictly above $\max (||\mu||, ||\om||, ||\g + 1||) + 2$.
For each $n>0$ let $M(n) > M(n-1)$ be a number strictly larger than $||\mu[s]|| + 2$ for all $s \subseteq \{M(0), \ldots, M(n-1)\}$.

Checking that $M$ is the required set is straightforward: if $s \in [M]^{< \om}$ and $i < |s|$, then $s(i) = M(n)$ for some $n$ and $M(n) > ||\mu[s \restriction i]|| + 2$ by definition of $M(n)$.
Moreover, $s(i)$ is greater than $|| \om || + 2$ and $|| \g + 1 || + 2$ because $M(0)$ is.
\end{proof}

For the rest of this section fix a set $s \in [M]^{1 \uplus \mu}$, where $M$ is as in the previous lemma.
Define
$$T \colon s \rightarrow \varphi_{\log \g} (\a \ctimes k) \quad \text{ by } \quad T(s(i)) = \mu [s \restriction i].$$
Notice that $T(\min s) = \mu$, $T(\max s) = 0$ (as $s^*$ is $\mu$-size), and that $T$ is a strictly decreasing function.
Moreover, by definition of $M$,
$$||T(s(i))|| + 2 < s(i) \quad \quad \text{ for all $i < |s|$}.$$		
		
Now we have all the required machinery.
We define a coloring on $[s]^{\ge 1 + \g}$ (which denotes the set of $(1 + \g)$-large subsets of $s$)
\begin{align*}
    \overline{c} \colon & \,\, [s]^{\ge 1 + \g} \rightarrow k + 3	\\
    & \,\, u \mapsto c(T(u))
\end{align*}
where $c$ is the coloring from Definition \ref{def: coloring on ord} and where by $T(u_0, \ldots, u_m)$ we mean $(T(u_0), \ldots, T(u_m))$.		
Similarly, we define $\pbar_\d$ and $\z$ on $u \subseteq s$ by  $\pbar_\d (u) = \pbar_\d(T(u))$ and $\z_u = \z_{T(u)}$.
		
We now show that in the definition of $\overline{c}$, only the $(1 + \g)$-size prefix matters: i.e.\ if $t \subseteq s$ and $v \sqsubseteq t$ is the $(1 + \g)$-size initial segment of $t$, then $\overline{c}(t) = \overline{c}(v)$.
		
\begin{lemma}\label{lem: prefix coloring}
For each $\nu$ and sets $u \sqsubseteq t \subseteq s$, if $u$ is $(1 + \nu)$-large then
$$p_\nu (u) = p_\nu (t).$$
\end{lemma}
		
When $\nu = \g$ this shows that the coloring $\overline{c}$ defined above can actually be regarded as a coloring of the $(1 + \g)$-size subsets of $s$.
This is because in the coloring $\overline{c}$ of a set $u$ we only care about the peeling functions of index $\g$ or index $\nu \in S(T(\min u))$.
By definition of the set $M$ (which contains $u$) each element of $u$ is larger than $||T(\min u)||$ and by definition of $||\cdot||$, it is also larger than $|\nu|$ for each $\nu \in S(T(\min u))$.
This means that if $u$ is $\g$-large, since $\g \ge \nu$, $u$ must also be $\nu$-large and so Lemma \ref{lem: prefix coloring} yields that also $p_\nu$ only depends on the $(1+\nu)$-size initial segment of its input.
		
\begin{proof}
We proceed by induction on $\nu$.
The case $\nu = 0$ is immediate since $\pbar_0$ is the identity function.			
%If $\nu = 1$ then $u$ is $2$-large hence its cardinality is at least $2$.
%Then 
%$$p_1(u) =\overline{u(0),u(1)} = \overline{t(0),t(1)} = p_1(t).$$
If $\nu = \rho + 1$ and $u$ is $(1 + \nu)$-large then $u$ is $(1 + \rho + 1)$-large and $u^-$ is $(1 + \rho)$-large.
By Corollary \ref{intuitive syst} $u$ is $(1 + \rho)$-large too.
Hence we have both $p_\rho(u) = p_\rho(t)$ and $p_\rho(u^-) = p_\rho(t^-)$ by inductive hypothesis.
Then 
$$p_\nu(u) = \overline{p_\rho(u), p_\rho(u^-)} = \overline{p_\rho(t), p_\rho(t^-)} = p_\nu(t).$$
			
Suppose now $\nu = \om^\sigma > 1$.
Recall that we denote $\min u$ by $u_0$.
Suppose that $u_0$ is the $i$-th element of $s$.
First notice that the first ordinal of the sequence to which we are applying $\pbar_\nu$ is $T(u_0) = \mu[s \restriction i]$ and so by Lemma \ref{infinite set M} we have that $u_0 = s(i) > ||T(u_0) || + 2$.
Since $u$ is $(1 + \om^\sigma)$-large then it is $(1 + \om^\sigma [u_0])$-large too (this is because it contains $u^-$ which is $(1 + \om^\sigma [u_0])$-large).
We claim that $\om^\sigma [u_0] > \om^{\sigma[||T(u_0)||]} \cdot ||T(u_0) ||$.
We proceed by cases.
\begin{enumerate}
	\item If $\sigma < \om^\sigma$ then 
    \begin{align*}
        \om^\sigma[u_0] = \om^{\sigma[u_0]} \cdot u_0 > \om^{\sigma[||T(u_0)||]} \cdot ||T(u_0) ||.
    \end{align*}
	\item If $\sigma = \varphi_1(0)$ then
    \begin{align*}
        \om^\sigma[u_0] & = \varphi_1(0)[u_0] = \varphi_0^{u_0+1} (0) > \varphi_0^{||T(u_0)||+3} (0) \\
        & > \varphi_0^{||T(u_0)||+2} (0) \cdot ||T(u_0)|| = \om^{\sigma[||T(u_0)||]} \cdot ||T(u_0)||.
    \end{align*}
	\item If $\sigma = \varphi_1(\b) > \b > 0$ then
    \begin{align*}
        \om^\sigma[u_0] & = \varphi_1(\b)[u_0] = \varphi_0^{u_0 + 1} (\varphi_1 (\b[u_0]) + 1) \\
        & > \varphi_0^{||T(u_0)|| + 3} (\varphi_1 (\b[||T(u_0)||+2]) +1) \\
        & > \varphi_0^{||T(u_0)||+2} (\varphi_1 (\b[||T(u_0)||]) +1) \cdot ||T(u_0)|| \\
        & = \om^{\sigma[||T(u_0)||]} \cdot ||T(u_0)||.
    \end{align*}
	\item If $\sigma = \varphi_\d(0)$ with $\d > 1$ then
    \begin{align*}
        \om^\sigma[u_0] & = \varphi_\d(0)[u_0] = \varphi_{\d[u_0]}^{u_0 + 1} (0) > \varphi_{\d[||T(u_0)|| + 2]}^{||T(u_0)|| + 3} (0) \\
        & > \varphi_{\d[||T(u_0)||]}^{||T(u_0)||+1} (0) \cdot ||T(u_0)|| = \om^{\sigma[||T(u_0)||]} \cdot ||T(u_0)||.
    \end{align*}
	\item If $\sigma = \varphi_\d(\b) > \b > 0$ with $\d > 1$ then
    \begin{align*}
        \om^\sigma[u_0] & = \varphi_\d(\b)[u_0] = \varphi_{\d[u_0]}^{u_0 + 1} (\varphi_\d (\b[u_0]) + 1) \\
        & > \varphi_{\d[||T(u_0)|| + 2]}^{||T(u_0)|| + 3} (\varphi_\d (\b[||T(u_0)|| + 2]) +1) \\
        & > \varphi_{\d[||T(u_0)||]}^{||T(u_0)|| + 1} (\varphi_\d (\b[||T(u_0)||]) + 1) \cdot ||T(u_0)|| \\
        & = \om^{\sigma[||T(u_0)||]} \cdot ||T(u_0)||.
    \end{align*}
    \item If $\sigma = \G_0$ then
    \begin{align*}
        \om^\sigma[u_0] & = \G_0[u_0] = \varphi_{\G_0[u_0-1]} (0) > \varphi_{\G_0[||T(u_0)||+1]} (0) \\
        & > \varphi_{\G_0[||T(u_0)||]} (0) \cdot ||T(u_0)|| = \G_0[||T(u_0)|| + 1] \cdot ||T(u_0)|| \\ 
        & > \G_0[||T(u_0)||] \cdot ||T(u_0)|| = \om^{\sigma[||T(u_0)||]} \cdot ||T(u_0)||.
    \end{align*}
    \item If $\sigma = \G_\xi$ with $0 < \xi < \zeta$ then
    \begin{alignat*}{2}
        \om^\sigma[u_0] & = \G_\xi [u_0] \\
        & =\varphi_{\ddots \varphi_{\G_{\xi \lceil u_0 \rceil} + 1} (0) \iddots} (0) & \qquad & \text{$\varphi$ occurs $u_0 + 1$ times} \\
        & > \varphi_{\ddots \varphi_{\G_{\xi \lceil ||T(u_0)|| \rceil} + 1} (0) \iddots} (0) & \qquad & \text{$\varphi$ occurs $||T(u_0)|| + 2$ times} \\
        & > \varphi_{\ddots \varphi_{\G_{\xi \lceil ||T(u_0)|| \rceil} + 1} (0) \iddots} (0) \cdot ||T(u_0)|| & \qquad & \text{$\varphi$ occurs $||T(u_0)|| + 1$ times} \\
        & = \G_\xi [||T(u_0)||] \cdot ||T(u_0)|| \\
        & =\om^{\sigma [||T(u_0)||]} \cdot ||T(u_0)||.
    \end{alignat*}
\end{enumerate}

We know by Lemma \ref{convergence lemma} that $p_{< \om^\sigma} (u) = p_{\om^{\sigma[||T(u_0)||]} \cdot ||T(u_0)||} (u)$ and that $p_{< \om^\sigma} (t) = p_{\om^{\sigma[||T(u_0)||]} \cdot ||T(u_0)||} (t)$.
Combining the claim and Lemma \ref{convergence lemma}, we know that $p_{\om^{\sigma[||T(u_0)||]} \cdot ||T(u_0)||} (u) = p_{\om^\sigma[u_0]} (u)$ and that $p_{\om^{\sigma[||T(u_0)||]} \cdot ||T(u_0)||} (t) = p_{\om^\sigma[u_0]} (t)$.
Finally, since $u$ is $(1 + \om^\sigma[u_0])$-large, by inductive hypothesis we also know that $p_{\om^\sigma[u_0]} (u) = p_{\om^\sigma[u_0]} (t)$ and so we get $p_{< \om^\sigma} (u) = p_{< \om^\sigma} (t)$.
We then peel off one application of $\varphi_\sigma$ in each side and obtain $p_{\om^\sigma} (u) = p_{\om^\sigma} (t)$ as required.

We are left with the case $\nu = \rho + \om^\sigma$ with $\rho \ggeq \om^\sigma$ and $\sigma >0$.
Recall that $p_{< \om^\sigma} (\pbar_\rho (u)) = \lim_{\tau \rightarrow \om^\sigma} p_\tau (\pbar_\rho (u))$ and by Lemma \ref{convergence lemma} the limit converges after the ordinal $\om^{\sigma[||p_\rho (u)||]} \cdot ||p_\rho (u) ||$.
Analogously $p_{< \om^\sigma} (\pbar_\rho (t)) = \lim_{\tau \rightarrow \om^\sigma} p_\tau (\pbar_\rho (t))$ and the limit converges after $\om^{\sigma[||p_\rho (t)||]} \cdot ||p_\rho (t) ||$.
Since $p_\rho (u)$ and $p_\rho (t)$ are subterms of $T(u_0)$ (by Lemma \ref{basic peeling}) then $||T(u_0)|| \ge ||p_\rho(u)||$ and $||T(u_0)|| \ge ||p_\rho(t)||$.
Therefore, by the claim within the proof of case $\nu = \om^\sigma$ above, we get that $\om^\sigma [u_0] > \om^{\sigma[||p_\rho (u)||]} \cdot ||p_\rho(u)||$ and $\om^\sigma [u_0] > \om^{\sigma[||p_\rho (t)||]} \cdot ||p_\rho(t)||$.
Now $u$ is $(1 + \rho + \om^\sigma [u_0])$-large because $u^-$ is.
Hence, by inductive hypothesis, $p_{\rho + \om^\sigma [u_0]} (u) = p_{\rho +\om^\sigma [u_0]} (t)$.
We get that $p_{< \om^\sigma} (\pbar_\rho (u)) = p_{< \om^\sigma} (\pbar_\rho (t))$ and peeling off one application of $\varphi_\sigma$ in each side we obtain $p_{\rho + \om^\sigma} (u) = p_{\om^\sigma} (\pbar_\rho (u)) = p_{\om^\sigma} (\pbar_\rho (t)) = p_{\rho + \om^\sigma} (t)$ as required.
\end{proof}
		
We can now go back to the proof of Theorem \ref{lowbound}.
		
\begin{proof}[Proof of Theorem \ref{lowbound}]
Towards a contradiction suppose that $t$ is a $((1 + \g) \uplus \a)$-large homogeneous subset of $s$ for the coloring $\overline{c}$ we defined above.
Write $t$ as $t_\a \conc t_\g$ with $t_\a < t_\g$, where $t_\a$ is $\a$-size and $t_\g$ is $(1 + \g)$-large.
For each $i \le |t_\a|$, let $t^{-i}$ be the set obtained by removing the first $i$ elements from $t$.
Since $t_\g \subseteq t^{-i}$, we have that $t^{-i}$ is $(1 + \g)$-large.
By homogeneity, we have that $\overline{c} (t^{-i})$ has the same color for all $i \le |t_\a|$.
We distinguish four cases based on the color of the homogeneous set $t$.\medskip
			
{\bf Case 1:} Suppose that $t$ is homogeneous of color $j < k$.
Then, for every $i \le |t_\a|$, we have that $p_\g (t^{-i}) \in d^{-1}(j)$.
Recall that $d^{-1}(j)$ is a subset of $\a \ctimes k$ isomorphic to $\a$ and that we denoted by $\pi_j : d^{-1}(j) \rightarrow \a$ the order isomorphism.
Define a map
\begin{align*}
    f \colon & \,\, t_\a \cup \{t_\g (0)\} \rightarrow \a \\
    & \,\, t(i) \mapsto \pi_j(p_\g(t^{-i})).  
\end{align*}
Recall that since $\overline{c} (t^{-i}) = j < k$, we have that $\z_{t^{-i}}$ does not exists, and hence that $p_\g (t^{-i}) > p_\g(t^{-(i+1)})$.
Therefore $f$ is strictly decreasing.
Furthermore, the ordinal $p_\g (t^{-i})$ is a subterm of $T(t(i))$ which belongs to $d^{-1}(j)$ and hence $|\pi_j(p_\g (t^{-i}))| < ||T(t(i))|| < t(i)$ by Definition \ref{doublenorm} and by definition of the set $M$ in Lemma \ref{infinite set M}.
In other words, for all $m \in t_\a \cup \{t_\g (0)\}$, $|f(m)| < m$.	
The Estimation Lemma \ref{lem: estimation lemma} states that no such $f$ can exist because $t_\a \cup \{t_\g (0)\}$ is $(1 \uplus \a)$-size.\medskip
			
{\bf Case 2:} Suppose that $t$ is homogeneous of color $k$.
In this case we have that
$$\z_{t^{-0}} > \z_{t^{-1}} > \cdots > \z_{t^{-|t_\a|}}$$
and all these ordinals belong to $\g + 1$.
We now define a map
\begin{align*}
    f \colon & \,\, t_\a \cup \{t_\g (0)\} \rightarrow \g + 1 \\
    & \,\, t(i) \mapsto \z_{t^{-i}}.
\end{align*}
			
To apply the Estimation Lemma, we need to show that $t_\a \cup \{t_\g (0)\}$ is $(1 \uplus (\g + 1))$-large which is equivalent to $t_\a$ being $(\g+1)$-large.
First we notice that it is $(1 \uplus \a)$-large since $t_\a$ is $\a$-size. 
Then recall that we are assuming $\a > \g$ which is equivalent to $\a \ge \g+1$.
If $\a = \g + 1$ we are done.
If $\a > \g + 1$ then by Lemma \ref{prop: goodnorm} and because $t \subset M$ (and $M(0) > |\g + 1|$ by Lemma \ref{infinite set M}) we know that for each $i$ if $\a[t(0), \dots, t(i)] > \g + 1$ then $\a[t(0), \dots, t(i+1)] \ge \g + 1$.
Since $\a[t_\a]=0$ there exists $i < |t_\a|$ such that $\a [t(0), \ldots, t(i)] = \g + 1$.
Therefore $t_\a$ is $(\g + 1)$-large.
			
Notice that, by Lemma \ref{lemma: Stau}, $\z_{t^{-i}} \in S(T(t(i)))$.
It follows that for $m = t(i)$,
$$|f(m)| = |\z_{t^{-i}}| \leq ||T(t(i))|| < t(i) = m.$$
The Estimation Lemma states that no such $f$ exists. \medskip
			
{\bf Case 3:} Suppose that $t$ is homogeneous of color $k+1$.
In this case we have
$$\z_{t^{-0}} = \z_{t^{-1}} = \cdots = \z_{t^{-|t_\a|}}$$
and let $\z$ denote this ordinal.
Notice that $\z \ne 0$ because $T(t)$ is a strictly decreasing sequence of ordinals.

We claim that the ordinals $p_\z (t^{-i})$ originate from distinct occurrences of the elements of the multiset $\Sub (T(t(0)))$, and hence there should be no more than $||T(t(0))||$ many of them.
Since $||T(t(0))|| < t(0)$ by Lemma \ref{infinite set M},
this implies that $t_\a$ has less than $t(0)$ elements.
However, since $\a \ge \om$ and $M(0) > | \om |$, we get that $t_\a$ is $\om$-large (the argument is the same used to show that $t_\a$ is $(\g + 1)$-large in Case $2$ above), which contradicts the fact that $t_\a$ has less than $t(0)$ elements. 
			
Let us now prove the claim.
We have
\begin{gather*}
   p_{< \z}(t^{-0}) > p_{< \z}(t^{-1}) > \cdots > p_{< \z}(t^{-|t_\a|}) > p_{< \z}(t^{-|t_\a| - 1})\\ 
   \text{and} \quad p_\z (t^{-0}) \leq p_\z (t^{-1}) \leq \cdots \le p_\z(t^{-|t_\a|}) \le p_\z(t^{-|t_\a| - 1}).
\end{gather*}
First we show that $\z$ must be a successor ordinal so that $p_{< \z}$ is $p_{\z - 1}$.
Towards a contradiction suppose that $\z$ is limit.
By Remark \ref{rem:limitcase}, it must be $p_{< \z}(t^{-0}) = \varphi_\d (0)$ for some $\d > 0$ and $p_{< \z}(t^{-1}) = 0$.
In particular, the strictly decreasing chain above must have length $2$ and so $t_\a=\emptyset$, a contradiction.

%So $p_\z (t^{-i}) < p_{\z-1}(t^{-i})$ for every $i \le |t_\a|$.
Going back to the proof of the claim, recall that, for every $i \le |t_\a|$, $p_\z (t^{-i})$ is an exponent of a term in the Cantor normal form of $p_{\z-1}(t^{-i})$, we claim that for this to be the case, we must have the following situation:
\begin{align*}
    p_{\z-1} (t^{-0}) & = \om^{\a_0} + \cdots + \om^{\a_{n_1}} + \om^{p_\z (t^{-1})} + \cdots + \om^{\a_{n_0}} + \om^{p_\z(t^{-0})} + \tau_0 \\
    p_{\z-1} (t^{-1}) & = \om^{\a_0} + \cdots + \om^{\a_{n_1}} + \om^{p_\z(t^{-1})} + \cdots + \om^{\a_{n_0}} + \tau_1 \\
    p_{\z-1} (t^{-2}) & = \om^{\a_0} + \cdots + \om^{\a_{n_1}} + \tau_2 \\
    \vdots \qquad  & \vdots \qquad \vdots
\end{align*}
By definition of $p_\z (t^{-0})$, we know that the Cantor normal forms of the ordinals $p_{\z-1} (t^{-0})$ and $p_{\z-1} (t^{-1})$ are equal up to a certain term $\om^{\a_{n_0}}$, and then in $p_{\z-1} (t^{-0})$ we have the term $\om^{p_\z (t^{-0})}$, while the tail $\tau_1$ of $p_{\z-1}(t^{-1})$ is strictly smaller than $\om^{p_\z (t^{-0})}$.
Since $p_\z (t^{-0}) \leq p_\z (t^{-1})$, we must	have that $p_\z (t^{-1})$ is an exponent of the Cantor normal form of $p_{\z-1} (t^{-1})$ that shows up before $n_0$.
This implies that $p_\z (t^{-1})$ is one of the exponents of $p_{\z-1} (t^{-0})$ too.
Since $\om^{p_\z (t^{-0})} > \tau_1$, we must have that $\om^{p_\z (t^{-1})}$ comes strictly before $\om^{p_\z (t^{-0})}$ in the Cantor normal form of $p_{\z-1} (t^{-0})$.
That is, even if they are equal	as ordinals, they are different subterms of $^{p_\z (t^{-0})}$.
Continuing like this we get that each $p_{\z} (t^{-i})$ is a different exponent in the Cantor normal form of ${p_{\z-1} (t^{-0})}$.
Finally, recall that ${p_{\z-1} (t^{-0})}$ is already a subterm of $T(t(0))$ by Lemma \ref{basic peeling}.
Therefore, we get that that each $p_\z (t^{-i})$ for $i \le |t_\a|$ originates from a distinct occurrence of an element of $\Sub(T(t(0)))$, as required. \medskip
			
{\bf Case 4:} Suppose that $t$ is homogeneous of color $k+2$.
In this case we have
$$\z_{t^{-0}} < \z_{t^{-1}} < \z_{t^{-2}} < \cdots < \z_{t^{-|t_\a|}}.$$
Notice that, by definition of the color $k+2$, they must all be defined while $\z_{t^{-|t_\a|-1}}$ is either defined and larger than $\z_{t^{-|t_\a|}}$ or undefined.
For each $i$, let us denote $\z_{t^{-i}}$ by $\z_i$ for simplicity.
Again we first show that each $\z_i$ is a successor ordinal.
Each of them must be non zero because $T(t)$ is strictly decreasing.
Towards a contradiction suppose that $\z_i$ is limit.
Then by Remark \ref{rem:limitcase}, it must be $p_{< \z_i}(t^{-i}) = \varphi_\d (0)$ for some $\d > 0$ and $p_{< \z_i}(t^{-i-1}) = 0$.
But in this case $\z_{i+1}$ must be defined and strictly smaller than $\z_i$, a contradiction.

We claim that for $i>0$, $p_{\z_i-1} (t^{-i})$ is a proper subterm of  $p_{{\z_{i-1}}-1} (t^{-i+1})$.
It would then again follow that the ordinals $p_{\z_i} (t^{-i})$ are all different subterms of $T(t(0))$ and we would get the same contradiction as in the previous case.
For ease the notation we only consider the case $i = 1$.
Since $p_{\z_0-1} (t^{-0}) > p_{\z_0-1}(t^{-1})$ and $p_{\z_0} (t^{-0}) \leq p_{\z_0} (t^{-1})$, as in Case 3, we have that $p_{\z_0} (t^{-1})$ is an exponent of	the Cantor normal form of $p_{\z_0-1} (t^{-0})$ that shows up before $p_{\z_0} (t^{-0})$.
Since $\z_1 > \z_0$, we have that $p_{\z_1-1} (t^{-1})$ is a subterm of $p_{\z_0} (t^{-1})$.
We thus have that $p_{\z_1-1} (t^{-1})$ is a proper subterm of $p_{\z_0-1} (t^{-0})$.
\end{proof}

\section{The Upper Bound}\label{sec:upbound}

In this section we show that for $k \in \N$ and countable ordinals $\a$ and $\g$, $\Ram(\a)^{1 + \g}_{< \om} \le \varphi_{\log \g}(\a \cdot \om)$.
This is achieved by proving the following theorem, which is the main result of this section.
		
\begin{theorem}\label{upbound}
Let $k \in \N$ and $\a$ and $\g$ be countable ordinals.	
Then 
$$\Ram(\a)^{1 + \g}_{k} \le \varphi_{\log \g}(\a \ctimes k).$$
\end{theorem}

Recall that this means that for all $\SD$-barriers $\A$ and $\C$ with $\height(\A)=\a$, $\height(\C) = 1 + \g$ and $\base(\A)$ a final segment of $\base(\C)$, there exists a $\SD$-barrier $\B$ such that $\base(\B) = \base(\A)$, $\height(\B) \le \varphi_{\log \g} (\a \ctimes k)$ and $\B \rightarrow (\A)^\C_k$.

Since $\sup_{k \in \N} \a \ctimes k = \a \cdot \om$ and the Veblen functions are normal, Theorem \ref{upbound} implies that $\Ram(\a)^{1 + \g}_{< \om} \le \varphi_{\log \g}(\a \cdot \om)$.
		
Notice that if $\g = 0$, by Corollary \ref{pigeon} we already know that $\Ram(\a)^1_k = \a \ctimes k$.
Since by Definition \ref{philog} $\varphi_{\log 0}$ is the identity function we obtain Theorem \ref{upbound} in this case and from now on we can assume that $\g \ge 1$.
Moreover, if both $\a$ and $\g$ are below $\om$ then $\Ram(\a)^{1 + \g}_{k} < \om \le \varphi_{\log \g}(\a \ctimes k)$.
Hence we can restrict to the case $\max \{\a, \g\} \ge \om$.
		
First we show that we can reduce to a simpler case, useful for the rest of the section.
		
\begin{lemma}\label{smallerexp}
Let $k \in \N$ and $\A$, $\B$ and $\C$ be fronts where $\base(\B) = \base(\A)$ is a final segment of $\base(\C)$.
If $\B \rightarrow (\A)^{\one \oplus \C}_k$ then $\B \rightarrow (\A)^\C_k$  
\end{lemma}
		
\begin{proof}
Let $s$ be a $(\one \oplus \B)$-size set and let $c \colon [s]^\C \rightarrow k$ be a coloring.
Define the coloring $c' \colon [s]^{\one \oplus \C} \rightarrow k$ as $u \mapsto c(u^*)$. 
Then by assumption there exists a set $t \in [s]^{\one \oplus \C \oplus \A}$ $c'$-homogeneous, say of color $i < k$.
By definition $t^*$ is $(\C \oplus \A)$-size: we claim that $t^*$ is $c$-homogeneous of color $i$.
Let $u \in [t^*]^\C$.
Then $\max t > \max u$ and so $u \conc \la \max t \ra \in [t]^{\one \oplus \C}$. 
It follows that $c(u) = c'(u \conc \la \max t \ra) = i$.
\end{proof}
		
We need the following extension of the arrow notation.
		
\begin{definition}
Let $k \in \N$ and $\A$, $\B$, $\C$ and $\D$ be fronts such that $\base(\B) = \base(\A)$ and $\base(\D)$ are final segments of $\base(\C)$.  
We write
$$\B \rightarrow (\A)^{\one \oplus \C \oplus \D \rightarrow \one \oplus \D}_k$$
to mean that for each $(\one \oplus \B)$-size $s$ and each $k$-coloring $c$ of $ [s]^{\one \oplus \C \oplus \D}$, there exists a $(\one \oplus \C \oplus \A)$-size $t$ subset of $s$ such that the coloring of an element of $[t]^{\one \oplus \C \oplus \D}$ depends only on its $(\one \oplus \D)$-size initial segment.
We say that $t$ is a $(\one \oplus \D)$-prehomogeneous set for $c$.
\end{definition}
		
Next we state a key intermediate result to prove Theorem \ref{upbound}.
		
\begin{theorem}\label{keyth}
Let $k \in \N$ and $\A$ and $\C$ be $\SD$-barriers where $\height(\A) = \a$, $\height(\C) = \g$ and $\base(\A)$ is a final segment of $\base(\C)$.
There exists a $\SD$-barrier $\B$ with $\base(\B) = \base(\A)$ and $\height(\B) \le \varphi_{\log \g} (\a)$ such that $\B \rightarrow (\A)^{\one \oplus \C \rightarrow \one}_k$.
\end{theorem}
		
We delay the proof of Theorem \ref{keyth} to the end of the section.
We show first how to derive Theorem \ref{upbound}.
		
\begin{proof}[Proof of Theorem \ref{upbound}]
Recall that $\A$ and $\C$ are $\SD$-barriers with $\height(\A)=\a$, $\height(\C) = 1 + \g$ and $\base(\A)$ a final segment of $\base(\C)$.
Recall also that we assume $\g \ge 1$.
If $\g$ is finite then by decomposability $\C = [\base(\C)]^{1 + \g}$. 
Therefore $\C = \one \oplus \C'$ where $\C' = [\base(\C)]^\g$.
If $\g$ is infinite then $\height(\C) = 1 + \g = \g$ and in this case let $\C' = \C$.
		
Let $\B(\A) = \B(\A, \ldots, \A)$ be the $\SD$-barrier of Definition \ref{pigeonfront}.
By Lemma \ref{pigeonheight} $\height (\B (\A)) = \a \ctimes k$ and by Theorem \ref{generalpigeon} $\B(\A) \rightarrow (\A)^{\one}_k$.
By Theorem \ref{keyth} there exists a $\SD$-barrier $\B$ such that $\height(\B) \le \varphi_{\log \g}(\a \ctimes k)$ and $\B \rightarrow (\B(\A))^{\one \oplus \C' \rightarrow \one}_k$.
We claim that it suffices to prove that $\B \rightarrow (\A)^{\one \oplus \C'}_k$.
Indeed, if $\g$ is finite then $\one \oplus \C' = \C$, while if $\g$ is infinite then $\one \oplus \C' = \one \oplus \C$ so that Lemma \ref{smallerexp} yields $\B \rightarrow (\A)^\C_k$ for the same $\SD$-barrier $\B$.

We thus need to prove $\B \rightarrow (\A)^{\one \oplus \C'}_k$.
Let $s$ be $(1 \oplus \B)$-size and let $c \colon [s]^{\one \oplus \C'} \rightarrow k$ be any coloring. 
We need to prove that there exists a homogeneous $(\one \oplus \C' \oplus \A)$-size set for $c$.
Let $t \subseteq s$ be the $(\one \oplus \C' \oplus \B(\A))$-size set which is $\one$-prehomogeneous for $c$.
By definition $t = t_0' \conc t_1'$ for some $\B(\A)$-size $t_0'$ and some $(\one \oplus \C')$-size $t_1'$ and so
we may also write $t= t_0' \conc \la \min t_1' \ra \conc t_1'^-$.
We define $t_0 = t_0' \conc \la \min t_1' \ra$ and $t_1 = t_1'^-$: in this way $t= t_0 \conc t_1$ where $t_0$ is $(\one \oplus \B(\A))$-size and $t_1$ is $(\one \oplus \C'_{\max t_0})$-size.

We claim that for each $n \in t_0$, $\la n \ra \conc t_1$ is $(\one \oplus \C')$-large and so for each $n \in t_0$ there exists $r \subseteq t$ such that $\la n \ra \conc r$ is $(\one \oplus \C')$-size.
Suppose that for some $n$, $\la n \ra \conc t_1$ is $(\one \oplus \C')$-small and let $v$ be a $(\one \oplus \C')$-size set such that $\la n \ra \conc t_1 \sqsubset v$.
Then since $n \le \max t_0$ it follows that $\la \max t_0 \ra \conc t_1$ and $v$ contradict the smoothness of $\one \oplus \C'$.

Let $\overline{c} \colon t_0 \rightarrow k$ be the coloring defined as $\overline{c}(n) = c (\la n \ra \conc r)$ where $r$ is any $(\one \oplus \C'_n)$-size subset of $t$.
By the claim above and since $t$ is $\one$-prehomogeneous for $c$, $\overline{c}$ is well defined.
We obtain that $\overline{c}$ is a $k$-coloring of a $(\one \oplus \B(\A))$-size set and by Theorem \ref{generalpigeon} there exists $u \in [t_0]^{\one \oplus \A}$ $\overline{c}$-homogeneous, say of color $i < k$.

We claim that $u \conc t_1$ is a $(\one \oplus \C' \oplus \A)$-large homogeneous set for $c$.			
We start by showing the largeness.
Since $u^*$ is $\A$-size by definition, we only need to prove that $\la \max u \ra \conc t_1$ is $(\one \oplus \C')$-large.
Towards a contradiction assume that $\la \max u \ra \conc t_1$ is $(\one \oplus \C')$-small and let $v$ be a $(\one \oplus \C')$-size set such that $\la \max u \ra \conc t_1 \sqsubset v$.
Then since $\la \max t_0 \ra \conc t_1 = t_1'$ is $(\one \oplus \C')$-size and $\max u \le \max t_0$, we get that $v$ and $t_1'$ contradict the smoothness of $\one \oplus \C'$.
Therefore $\la \max u \ra \conc t_1$ is $(\one \oplus \C')$-large.
We are left to prove that $u \conc t_1$ is $c$-homogeneous.
Since $\la \max t_0 \ra \conc t_1$ is $(\one \oplus \C')$-size then $t_1$ must be $(\one \oplus \C')$-small.
For this reason a $(\one \oplus \C')$-size subset $w$ of $u \conc t_1$ cannot be entirely contained in $t_1$ and this means that $\min w \in u$.
By definition $c(w) = \overline{c} (\min w) = i$.

Therefore $u \conc t_1$ is $c$-homogeneous and so $s$ has a homogeneous $(\one \oplus \C' \oplus \A)$-size set included in $u \conc t_1$.
\end{proof}	

Notice that the above proof is modular: if Theorem  \ref{keyth} holds when $\g'<\g$ then Theorem \ref{upbound} holds when $\g'<\g$ as well. 
Thus we can prove Theorem \ref{keyth} by transfinite induction on $\g$ using as hypothesis also  Theorem \ref{upbound} for smaller ordinals.

The next lemma deals with the case $\g = 1$, so that $\varphi_{\log \g} = \varphi_0$ is ordinal exponentiation with base $\om$.
		
\begin{lemma}\label{caseg=1}
Let $k \in \N$ and $\A$ be a $\SD$-barrier with $\height(\A) = \a \ge \om$.
There exists a $\SD$-barrier $\B$ with $\base(\B) = \base(\A)$, $\height(\B) \le \om^\a$ and such that $\B \rightarrow (\A)^{\two \rightarrow \one}_k$.
\end{lemma}
		
\begin{proof}
We proceed by induction on $\a$.
For each $n$ let $\B'_n$ be a $\SD$-barrier with $\base(\B'_n) = \base(\A_n)$ and $\height(\B'_n) \le \om^{\height(\A_n)}$ such that $\B'_n \rightarrow (\A_n)^{\two \rightarrow \one}_k$.
If $\height(\A_n)$ is infinite then $\B'_n$ exists by induction hypothesis.
If $\height(\A_n) = 0$ then $\A_n$ is the degenerate front and $\B'_n = \one$ works, so in this case $\height(\B'_n) = 1 = \om^0 = \om^{\height(\A_n)}$.
If $0 < \height(\A_n) < \om$ then by classical Ramsey theory there exists $m \in \N$ such that $\B'_n$ can be taken as the $\SD$-barrier of $m$-tuples so in this case $\height(\B'_n) = m < \om \le \om^{\height(\A_n)}$.

Let $\B_n$ be a $\SD$-barrier (whose existence is proved in Section \ref{sec:pigeonhole}) of height $\height(\B'_n) \ctimes k$ and such that $\B_n \rightarrow (\B'_n)^{\one}_k$.
We define $\B = \{s: s^- \in \B_{\min s} \}$ (which justifies the name $\B_n$ for the previous $\SD$-barriers).
Such $\B$ is a block but not necessarily a $\SD$-barrier.
However by Lemma \ref{comb} and by Corollary \ref{dec} there exists a $\SD$-barrier with same height and same base of $\B$ and such that each of its elements is $\B$-large.
Thus we may assume that $\B$ is a $\SD$-barrier and we claim that $\B \rightarrow (\A)^{\two \rightarrow \one}_k$.
			
Let $s$ be $(\one \oplus \B)$-size, $n = \min s$ and $c \colon [s]^{\two} \rightarrow k$.
We show that there exist $t' \in [s^-]^{\one \oplus \B'_n}$ and $i < k$ such that $c(\la n,m \ra) = i$ for all $m \in t'$. 
In order to do that, we define a new coloring $\overline{c} \colon s^- \rightarrow k$ as $\overline{c}(m) = c(\la n,m \ra)$.
Since $s^-$ is $(\one \oplus \B_n)$-size and $\B_n \rightarrow (\B'_n)^{\one}_k$, there exists a $\overline{c}$-homogeneous $(\one \oplus \B'_n)$-size set, which is the desired $t'$.
			
By inductive hypothesis and since $\B'_n \rightarrow (\A_n)^{\two \rightarrow \one}_k$ there exists a set $t'' \in [t']^{\two \oplus \A_n}$ which is $\one$-prehomogeneous for $c$.
Therefore $\la n \ra \conc t'' \in [s]^{\two \oplus \A}$ is $\one$-prehomogeneous for $c$.
			
Notice that
\begin{multline*}
    \height(\B) = \sup_{n \in \base(\B)} (\height(\B_n) + 1) = \sup_{n \in \base(\B)} ((\height(\B'_n) \ctimes k) + 1) \\
    \le \sup_{n \in \base(\B)} ((\om^{\height(\A_n)} \ctimes k) +1) \le \om^\a
\end{multline*}
as required.
\end{proof}
		
The next Lemma deals with the case $\g = \om^\d$. 
		
\begin{lemma}\label{caseomdelta}
Assume that Theorem \ref{upbound} holds for each $\g' < \om^\d$.
Let $k \in \N$ and $\A$ and $\C$ be $\SD$-barriers where $\height(\A) = \a$, $\height(\C) = \om^\d \ge \om$ and $\base(\A)$ is a final segment of $\base(\C)$.
There exists a $\SD$-barrier $\B$ with $\base(\B) = \base(\A)$, $\height(\B) \le \varphi_\d (\a)$ and such that $\B \rightarrow (\A)^{\one \oplus \C \rightarrow \one}_k$.
\end{lemma}
		
\begin{proof}
We proceed by induction on $\a$.
If $\a=0$ then $\A$ is the degenerate front and $\one \oplus \C \oplus \A = \one \oplus \C$.
Take $\B = \C$: then $\B \rightarrow (\A)^{\one \oplus \C \rightarrow \one}_k$ and $\height(\B) = \varphi_0 (\d) \le \varphi_0(\varphi_\d (0)) = \varphi_\d(0)$ as required.
		
Suppose $\a > 0$ and, by inductive hypothesis, for each $n$ let $\B_n'$ be a $\SD$-barrier with $\base(\B'_n) = \base(\A_n)$ such that $\height(\B'_n) \le \varphi_\d (\height(\A_n))$ and $\B'_n \rightarrow (\A_n)^{\one \oplus \C \rightarrow \one}_k$.  
Let $\B_n$ be a $\SD$-barrier (whose existence is ensured by the hypothesis on Theorem \ref{upbound} and the fact that $\height(\C_n) < \om^\d$ for each $n$) such that $\height(\B_n) \le \varphi_{\log \height(\C_n)} (\height(\B'_n) \ctimes k)$ and $\B_n \rightarrow (\B'_n)^{\one \oplus \C_n}_k$.
We define $\B = \{s: s^- \in \B_{\min s} \}$ (which justifies the name $\B_n$ for the previous $\SD$-barriers).
As in the proof of Lemma \ref{caseg=1} we may assume that $\B$ is a $\SD$-barrier and we claim that $\B \rightarrow (\A)^{\one \oplus \C \rightarrow \one}_k$.
			
Let $s$ be $(\one \oplus \B)$-size, $n = \min s$ and $c \colon [s]^{\one \oplus \C} \rightarrow k$.
We show that there exist $t' \in [s^-]^{\one \oplus \B'_n}$ and $i < k$ such that $c(\la n \ra \conc r) = i$ for all $r \in [t']^{\one \oplus \C_n}$. 
We define $\overline{c}: [s^-]^{\one \oplus \C_n} \rightarrow k$ as $\overline{c}(r) = c(\la n \ra \conc r)$.
Since $s^-$ is $(\one \oplus \B_n)$-size and $\B_n \rightarrow (\B'_n)^{\one \oplus \C_n}_k$ then there exists a $\overline{c}$-homogeneous $(\one \oplus \C_n \oplus \B'_n)$-size set.
Its $(\one \oplus \B'_n)$-size prefix is the desired $t'$.
			
Since $\B'_n \rightarrow (\A_n)^{\one \oplus \C \rightarrow \one}_k$, there exists a set $t'' \in [t']^{\one \oplus \C \oplus \A_n}$ which is $\one$-prehomogeneous for $c$.
Therefore $\la n \ra \conc t'' \in [s]^{\one \oplus \C \oplus \A}$ is $\one$-prehomogeneous for $c$.
			
We are left to prove that the height of $\B$ is bounded by $\varphi_\d(\a)$.
Indeed
\begin{align*}
	\height(\B) & = \sup_{n \in \base(\B)} (\height(\B_n) + 1) \le \sup_{n \in \base(\B)} (\varphi_{\log \height(\C_n)} (\height(\B'_n) \ctimes k) +1) \\
	& \le \sup_{n \in \base(\B)} (\varphi_{\log \height(\C_n)} (\varphi_\d (\height(\A_n)) \ctimes k) +1) \\
	& \le \sup_{n \in \base(\B)} (\varphi_{\log \height(\C_n)} (\varphi_\d(\a))) = \varphi_\d(\a)
\end{align*}
where the last inequality follows because the Veblen functions are strictly increasing, while the last equality follows by the fact that $\log \height(\C_n) < \d$ for each $n$ and so $\varphi_\d (\a)$ is a fixed point of $\varphi_{\log \height(\C_n)}$.
\end{proof}

The next step is to prove a strengthening of Lemma \ref{caseomdelta} which is useful for the proof of the remaining cases of Theorem \ref{keyth}.
		
\begin{lemma}\label{casenu2}
Assume that Theorem \ref{upbound} holds for each $\g' < \om^\d$.
Let $k \in \N$ and $\A$, $\C$ and $\D$ be $\SD$-barriers where $\height(\A) = \a$, $\height(\C) = \om^\d$ and $\base(\A)$ and $\base(\D)$ are final segments of $\base(\C)$.
There exists a $\SD$-barrier $\B$ with $\height(\B) \le \varphi_\d (\a)$ and $\base(\B) = \base(\A)$ such that $\B \rightarrow (\A)^{\one \oplus \C \oplus \D \rightarrow \one \oplus \D}_k$.
\end{lemma}
		
\begin{proof}
%Case $\d = 0$ is Lemma \ref{casenu} so suppose $\d > 0$. 
Starting from $\A$ and $\C$ we define sets of $\SD$-barriers $(\F_s)_{s \in T(\A)}$ and $(\E_s)_{s \in T(\A)}$ where $\base(\E_s) = \base(\F_s) = \base(\A) \setminus \{0, \ldots, \max s\}$. 
If $s \in \A$ then we define $\F_s = \C$ restricted to $\base(\A) \setminus \{0, \ldots, \max s\}$ (the restriction is still a $\SD$-barrier by Remark \ref{sd restriction}).
If $s \in T(\A) \setminus \A$, assume that for each $n > \max s$ we have defined $\F_{s \conc \la n \ra}$. 
Let $\E_{s \conc \la n \ra}$ be a $\SD$-barrier (whose existence is ensured by the assumption on Theorem \ref{upbound} and by the fact that $1 + \height(\C_n) < \om^\d$ for each $n$) with $\base(\E_{s \conc \la n \ra}) = \base(\F_{s \conc \la n \ra})$, $\height(\E_{s \conc \la n \ra}) \le \varphi_{\log \height(\C_n)} (\height(\F_{s \conc \la n \ra}) \ctimes k^{(2^{|s|})})$ and $\E_{s \conc \la n \ra} \rightarrow (\F_{s \conc \la n \ra})^{\one \oplus \C_n}_{k^{(2^{|s|})}}$.
We define $\F_s$ as the set $\{t \subseteq \base(\A) : t^- \in \E_{s \conc \la \min t \ra}\}$. 
Then $\F_s$ is clearly a block with $\base(\F_s) = \base(\A) \setminus \{0, \ldots, \max s\}$. 
By Lemma \ref{comb} and Corollary \ref{dec} we may assume that $\F_s$ is a $\SD$-barrier with same height and same base. 

We claim that for each $s \in T(\A)$, $\height(\F_s) \le \varphi_\d (\height(\A_s))$.
We proceed by induction on $\height(\A_s)$.
If $s \in \A$ then $\F_s$ is the restriction of $\C$ to a final segment of its base and by Lemma \ref{sb tail height} $\height(\F_s) = \om^\d \le \varphi_\d(0) = \varphi_\d (\height(\A_s))$.
Now suppose that $s \in T(\A) \setminus \A$ and the claim holds for all $s \conc \la n \ra \in T(\A)$. 
By construction, $\height(\E_{s \conc \la n \ra}) \le \varphi_{\log \height(\C_n)} (\height(\F_{s \conc \la n \ra}) \ctimes k^{(2^{|s|})})$.
It follows that 
\begin{align*}
	\height(\F_s) & = \sup_{n \in \base(\F_s)} (\height(\E_{s \conc \la n \ra}) +1 ) \\
    & \le \sup_{n \in \base(\F_s)} (\varphi_{\log \height(\C_n)} (\height(\F_{s \conc \la n \ra}) \ctimes k^{(2^{|s|})}) +1) \\
	& \le \sup_{n \in \base(\F_s)} (\varphi_{\log \height(\C_n)} (\varphi_\d(\height(\A_{s \conc \la n \ra})) \ctimes k^{(2^{|s|})}) +1) \\
	& \le \sup_{n \in \base(\F_s)} \varphi_{\log (\height(\C_n))} (\varphi_\d(\height(\A_s))) = \varphi_\d (\height(\A_s)) 
\end{align*}
where the last inequality follows since the Veblen functions are strictly increasing and the last equality follows by the fact that for every $n$ $\varphi_{\log \height(\C_n)}$ is a composition of Veblen functions each of index strictly smaller than $\d$ and so $\varphi_\d (\height(\A_s))$ is a fixed point of $\varphi_{\log \height(\C_n)}$.

In particular $\height(\F_{\la \ra}) \le \varphi_\d(\a)$, so if we set $\B = \F_{\la \ra}$ it suffices to show that $\B \rightarrow (\A)^{\one \oplus \C \oplus \D \rightarrow \one \oplus \D}_k$.
			
Let $s$ be $(\one \oplus \B)$-size and let $c \colon [s]^{\one \oplus \C \oplus \D} \rightarrow k$.
Using an idea similar to Mathias forcing, we recursively construct sequences $f_0, f_1, \ldots \in s$ and $F_0, F_1, \ldots \subseteq s$ such that for each $n$ (letting $u_n = \la f_0, \ldots, f_n \ra$) the following holds:
\begin{enumerate}[(1)]
	\item $f_0 < f_1 < \ldots < f_n < F_n$;
	\item $f_n = \min F_{n-1}$;
	\item $F_n \subseteq F_{n-1}$;
	\item $F_n$ is $(1 \oplus \F_{u_n})$-large;
	\item let $t_1,t_2 \in [s]^{\one \oplus \C \oplus \D}$ and let $t_1'$ and $t_2'$ be their $(\one \oplus \D)$-size prefixes. If $t_1' = t_2' \subseteq u_n$ and $t_1 \setminus t_1', t_2 \setminus t_2' \subseteq F_n$ then $c(t_1) = c(t_2)$.
\end{enumerate}
			
We initialize the construction by setting $F_{-1} = s$ and leaving $f_{-1}$ undefined, so that $u_{-1}=\la \ra$. 
Clauses $1 \div 4$ are trivial.
Since $(\one \oplus \C \oplus \D)$-size sets have at least two elements, clause $5$ is trivial too.
			
Suppose we are at stage $n$ and that we have already defined $f_{n-1}$ and $F_{n-1}$ satisfying clauses $1 \div 5$.
Let $f_n = \min F_{n-1}$ (so that clause $2$ is satisfied) and let $E_n = F_{n-1}^-$.
Let $c'_n \colon [E_n]^{\one \oplus \C_{f_n}} \rightarrow k^{(2^n)}$ be the function that maps $e \in [E_n]^{\one \oplus \C_{f_n}}$ to (a code for) the function $[u_{n-1}]^\D \rightarrow k$ defined by $d \mapsto c(d \conc \la f_n \ra \conc e)$.
Since $u_{n-1}$ has $2^n$ subsets, then $[u_{n-1}]^\D$ has at most $2^n$ elements.
Therefore there are at most $k^{(2^n)}$ functions $[u_{n-1}]^\D \rightarrow k$ and so $c'_n$ is well defined.
We know that $F_{n-1}$ is $(1 \oplus \F_{u_{n-1}})$-large and so $F_{n-1}^*$ is $\F_{u_{n-1}}$-large. 
It follows that $E_n^{*} = F_{n-1}^{*-}$ is $\E_{u_n}$-large and hence $E_n$ is $(\one \oplus \E_{u_n})$-large. 
By construction $\E_{u_n} \rightarrow (\F_{u_n})^{\one \oplus \C_{f_n}}_{k^{(2^n)}}$ and so there is a $(\one \oplus \C_{f_n} \oplus \F_{u_n})$-large subset of $E_n$ which is $c'_n$-homogeneous.
We call such set $F_n$ and notice that $F_n$ is $(\one \oplus \F_{u_n})$-large (hence clauses $3$ and $4$ are satisfied).
Since $F_n \subseteq E_n$, $f_n < F_n$ and clause $1$ holds too.
We are left to prove that clause $5$ is satisfied.
Let $t_1,t_2 \in [s]^{\one \oplus \C \oplus \D}$ and let $t_1'$ and $t_2'$ be their $(\one \oplus \D)$-size prefixes.
Moreover assume that $t_1' = t_2' \subseteq u_n$ and $t_1 \setminus t_1', t_2 \setminus t_2' \subseteq F_n$.
If $t_1' = t_2' \subseteq u_{n-1}$ then since $F_n \subseteq F_{n-1}$, by clause $5$ of stage $n-1$ we are done. 
If not then $t_1'^* = t_2'^* \in [u_{n-1}]^\D$ and $\max t_1' = \max t_2' = f_n$. 
Since $F_n$ is $c'_n$-homogeneous, then $d \mapsto c(d \conc \la f_n \ra \conc t_1 \setminus t_1')$ and $d \mapsto c(d \conc \la f_n \ra \conc t_2 \setminus t_2')$ are the same function (notice that $t_1 \setminus t_1'$ and $t_2 \setminus t_2'$ are $(\one \oplus \C_{f_n})$-size).
Therefore $c(t_1) = c(t_1'^* \conc \la f_n \ra \conc t_1 \setminus t_1') = c(t_2'^* \conc \la f_n \ra \conc t_2 \setminus t_2') = c(t_2)$.

The construction stops at a stage $m$ such that $u_m \in \A$. 
At this stage $\F_{u_m}$ is the restriction of $\C$ to a final segment of its base by definition and so $F_m$ is $(\one \oplus \C)$-large.
Let $w \sqsubseteq F_m$ be $(\one \oplus \C)$-size and let $f_{m+1} = \min w$.
The set $v = u_m \conc w$ is $(\one \oplus \C \oplus \A)$-size. 
Let $t_1, t_2 \in [v]^{\one \oplus \C \oplus \D}$ and let $t_1'$ and $t_2'$ be their $(\one \oplus \D)$-size prefixes.
Moreover assume that $t_1'=t_2'$.
Suppose first that $\max t_1' = \max t_2' \ge f_{m+1}$ then it must be $t_1 \setminus t_1'^*, t_2 \setminus t_2'^* \in [w]^{\one \oplus \C}$ and since $\one \oplus \C$ is a barrier we get $t_1 \setminus t_1'^* = w = t_2 \setminus t_2'^*$.
Hence $t_1 = t_2$ and $c(t_1) = c(t_2)$.
Suppose now that $\max t_1' = \max t_2' = f_n$ for some $n \le m$ and so $t_1'^* = t_2'^* \in [u_{n-1}]^\D$. 
Since $F_n$ is $c'_n$-homogeneous and $\la f_{n+1}, \ldots, f_m \ra \conc w \subseteq F_n$, then $d \mapsto c(d \conc \la f_n \ra \conc t_1 \setminus t_1')$ and $d \mapsto c(d \conc \la f_n \ra \conc t_2 \setminus t_2')$ are the same function.
Therefore 
$$c(t_1) = c(t_1'^* \conc \la f_n \ra \conc t_1 \setminus t_1') = c(t_2'^* \conc \la f_n \ra \conc t_2 \setminus t_2') = c(t_2).$$
Hence $v$ is a $(\one \oplus \C \oplus \A)$-size $(1 \oplus \D)$-prehomogeneous set for $c$.
\end{proof}	
		
We are finally ready to prove Theorem \ref{keyth}, concluding Section \ref{sec:upbound}.
		
\begin{proof}[Proof of Theorem \ref{keyth}]
Recall that Theorem \ref{keyth} states the following.
Given $k \in \N$ and $\A$ and $\C$ $\SD$-barriers where $\height(\A)=\a$, $\height(\C) = \g$ and $\base(\A)$ is a final segment of $\base(\C)$, there exists a $\SD$-barrier $\B$ with $\height(\B) \le \varphi_{\log \g} (\a)$ and the same base as $\A$ such that $\B \rightarrow (\A)^{\one \oplus \C \rightarrow \one}_k$.
			
We proceed by induction on $\g$. 
If $\g = 0$ then $\B=\A$ works.

If $\g>0$ the induction hypothesis and the modularity of the proof that Theorem \ref{keyth} implies Theorem \ref{upbound} yield that Theorem \ref{upbound} holds for every $\g'<\g$.
Then case $\g = \om^\d$ is proved in Lemmas \ref{caseg=1} and \ref{caseomdelta}.

Suppose now that $\g = \om^\d + \nu$ with $\om^\d \ggeq \nu > 0$.
Since $\C$ is decomposable, there exist $\SD$-barriers $\C_0$, $\C_1$ such that $\height(\C_0) = \om^\d$, $\height(\C_1) = \nu$ and $\C = \C_0 \oplus \C_1$.
By inductive hypothesis (since $\height(\C_1) = \nu < \g$) there exists a $\SD$-barrier $\B'$ with $\height(\B') \le \varphi_{\log \nu} (\a)$ and $\base(\B') = \base(\A)$ such that $\B' \rightarrow (\A)^{\one \oplus \C_1 \rightarrow \one}_k$.
By Lemma \ref{casenu2} there exists a $\SD$-barrier $\B$ with $\height(\B) \le \varphi_\d (\varphi_{\log \nu} (\a)) = \varphi_{\log \g} (\a)$ and $\base(\B) = \base(\B')$ such that $\B \rightarrow (\B')^{\one \oplus \C_0 \oplus \C_1 \rightarrow \one \oplus \C_1}_k$.
We claim that $\B \rightarrow (\A)^{\one \oplus \C \rightarrow \one}_k$.
			
Let $s$ be $(\one \oplus \B)$-size and $c \colon [s]^{\one \oplus \C} \rightarrow k$ be a coloring.
Since $\C = \C_0 \oplus \C_1$ and $\B \rightarrow (\B')^{\one \oplus \C_0 \oplus \C_1 \rightarrow \one \oplus \C_1}_k$, there exists $t \in [s]^{\one \oplus \C_0 \oplus \B'}$ which is $(\one \oplus \C_1)$-prehomogeneous for $c$.
Let $t = t_0' \conc t_1'$ where $t_0'$ is $\B'$-size $t_0'$ and $t_1'$ is $(\one \oplus \C_0)$-size.
We may also write $t= t_0' \conc \la \min t_1' \ra \conc t_1'^-$.
We define $t_0 = t_0' \conc \la \min t_1' \ra$ and $t_1 = t_1'^-$ so that $t= t_0 \conc t_1$, $t_0$ is $(\one \oplus \B')$-size and $t_1$ is $(\one \oplus (\C_0)_{\max t_0})$-size.
			
Let $\overline{c} \colon [t_0]^{\one \oplus \C_1} \rightarrow k$ be the coloring defined as $\overline{c}(d) = c (d \conc r)$ where $r$ is any $(\one \oplus (\C_0)_{\max d})$-size subset of $t$.
We claim that for each $d \in [t_0]^{\one \oplus \C_1}$, $d \conc t_1$ is $(\one \oplus \C)$-large and so there exists $r \subseteq t$ which is $(\one \oplus (\C_0)_{\max d})$-size.
Suppose $d \conc t_1$ is $(\one \oplus \C)$-small for some $(\one \oplus \C_1)$-size $d$.
This implies that $\la \max d \ra \conc t_1$ is $(\one \oplus \C_0)$-small.
Let $v$ be $(\one \oplus \C_0)$-size such that $\la \max d \ra \conc t_1 \sqsubset v$.
From $\max d \le \max t_0$ it follows that $\la \max t_0 \ra \conc t_1$ and $v$ contradict the smoothness of $\one \oplus \C_0$.
Moreover, since $t$ is $(\one \oplus \C_1)$-prehomogeneous for $c$, $\overline{c}$ is well defined.
We obtain that $\overline{c}$ is a coloring of a $(\one \oplus \B')$-size set and since  $\B' \rightarrow (\A)^{\one \oplus \C_1 \rightarrow \one}_k$ there exists $u \in [t_0]^{\one \oplus \C_1 \oplus \A}$ which is $\one$-prehomogeneous for $\bar{c}$.

We claim that $u \conc t_1$ is a $(\one \oplus \C \oplus \A)$-large $\one$-prehomogeneous set for $c$.			
We start by showing the largeness.
Let $u'$ be the $\A$-size initial segment of $u$.
We need to prove that $u \setminus u' \conc t_1$ is $(\one \oplus \C)$-large.
Notice that $u \setminus u'$ is $(\one \oplus \C_1)$-size.
Towards a contradiction assume that $u \setminus u' \conc t_1$ is $(\one \oplus \C)$-small. 
In other words suppose that $\la \max u \ra \conc t_1$ is $(\one \oplus \C_0)$-small.
Let $v$ be a $(\one \oplus \C_0)$-size set such that $\la \max u \ra \conc t_1 \sqsubset v$.
Since $\la \max t_0 \ra \conc t_1 = t_1'$ is $(\one \oplus \C_0)$-size and $\max u \le \max t_0$, we get that $v$ and $t_1'$ contradict the smoothness of $\one \oplus \C_0$.
Therefore $u \setminus u' \conc t_1$ is $(\one \oplus \C)$-large and so $u \conc t_1$ is $(\one \oplus \C \oplus \A)$-large.
			
We are left to prove that $u \conc t_1$ is $\one$-prehomogeneous for $c$.
Let $w \in [u \conc t_1]^{\one \oplus \C}$ and let $w' \sqsubset w$ be $(\one \oplus \C_1)$-size.
We claim that $w' \subseteq u$.
Otherwise if we write $w = w'^* \conc \la \max w' \ra \conc w \setminus w'$, we get $\la \max w' \ra \conc w \setminus w' \subseteq t_1$.
Since $w'^*$ is $\C_1$-size we have that $\la \max w' \ra \conc w \setminus w'$ is $(\one \oplus \C_0)$-size which is a contradiction since $\la \max w' \ra \conc w \setminus w' \subseteq t_1$ and $t_1$ is $(\one \oplus \C_0)$-small by its definition.
Therefore $w' \subseteq u$.
By definition of $\overline{c}$ we have $\overline{c} (w') = c(w)$ and since $u$ is $\one$-prehomogeneous for $\overline{c}$ it follows that $c(w)$ depends only on $\min w$.
Hence $u \conc t_1$ is $\one$-prehomogeneous for $c$.
\end{proof}

Theorem \ref{lowbound} together with Theorem \ref{upbound} proves Theorem \ref{maintheorem}.

\section{Nestedness of the System}\label{sec:nestedness}

This section is devoted to the proof that the system of fundamental sequences of Definition \ref{inducedsystem} on $\G_\zeta$ is nested (Theorem \ref{nested->nested}).		
Recall that a system of fundamental sequences is nested if it is never the case for $n > 1$ that
$$\b >  \g > \b[n] > \g[n].$$

The first few lemmas show that, in some circumstances, the fundamental sequences for ordinals above some bound lie entirely above that bound.
Each ordinal in the rest of the section is below $\G_\zeta$.
		
\begin{lemma}\label{<epsilon0}
Let $k \in \om$ and $\g$ be an ordinal.
If $\varphi_0^k(0) < \g < \epsilon_0$ then $\varphi_0^k (0) \le \g[1]$ and hence $\varphi_0^k(0) \le \g[n]$ for all $n > 0$.
\end{lemma}

\begin{proof}
We proceed by induction on $k$.
The case $k = 0$ is obvious as $\varphi_0^0 (0) = 0$.
Assume now that $\varphi_0^{k + 1} (0) < \g  < \epsilon_0$.
If $\g = \om^\a$ then $\varphi_0^k (0) < \a < \epsilon_0$ and by induction hypothesis $\varphi_0^k (0) \le \a[1]$: then $\varphi_0^{k + 1} (0) \le \om^{\a [1]} = \g [1]$.
Otherwise, we can write $\g = \om^\a + \b$ for some $\a \ge \varphi_0^{k + 1} (0)$ and $\om^\a \ggeq \b$.
Then $\g [1] = \om^\a + \b [1] \ge \varphi_0^{k + 1} (0)$.
\end{proof}
		
\begin{lemma}\label{>epsilon0}
For every $\d > \epsilon_0$ we have $\d [1] \ge \epsilon_0$ and hence $\d [n] \ge \epsilon_0$ for all $n > 0$.
Similarly, for every $\d > \G_0$ we have $\d [1] \ge \G_0$ and hence $\d [n] \ge \G_0$ for all $n > 0$.
\end{lemma}
		
\begin{proof}
We start with the first statement.
The proof is by induction on $\d$.
If the Cantor normal form of $\d$ has more than one term, then $\d [1]$ is larger or equal than the first term, which is $\ge \epsilon_0$.
If the normal form of $\d$ consists of a single term we distinguish four cases.
\begin{itemize}
    \item If $\d = \om^\a$ with $\d > \a$ then we must have $\a > \epsilon_0$ and the induction hypothesis yields $\a [1] \ge \epsilon_0$.
    Hence $\d [1] = \om^{\a [1]} \ge \epsilon_0$.
	\item If $\d = \varphi_{\d_0} (0) > \d_0$ then $\d_0 > 1$ and we have $\d_0 [1] > 0$.
	Thus $\d [1] = \varphi_{\d_0 [1]}^2 (0) > \epsilon_0$.
	\item If $\d = \varphi_{\d_0} (\a)$ with $\d > \a > 0$ and $\d_0 > 0$ we have $\varphi_{\d_0} (\a [1]) \ge \epsilon_0$ and, a fortiori, $\d [1] =	\varphi_{\d_0 [1]}^2(\varphi_{\d_0}(\a [1]) + 1) > \epsilon_0$.
    \item If $\d = \G_\xi$ then, for some $\b$ depending on $\xi$, $\d[1] = \varphi_{\varphi_\b (0)} (0) \ge \epsilon_0$.
\end{itemize}

For the second statement we proceed similarly, treating only the cases where a different argument is needed.
\begin{itemize}
	\item If $\d = \varphi_{\d_0} (0) > \d_0$ then $\d_0 > \G_0$ and we have $\d_0 [1] \ge \G_0$ by inductive hypothesis.
    Thus $\d [1] = \varphi_{\d_0 [1]}^2 (0) > \G_0$.
	\item Suppose $\d = \varphi_{\d_0} (\a)$ with $\d > \a > 0$ and $\d_0 > 0$. 
    We consider subcases depending on the value of $\a$:
    \begin{itemize}
        \item when $0 < \a \le \G_0$ we must have $\d_0 \ge \G_0$ as otherwise $\d \le \G_0$, a contradiction; then we obtain $\d[1] = \varphi^2_{\d_0[1]} (\varphi_{\d_0} (\a[1])+1) > \G_0$;
        \item when $\a > \G_0$, by inductive hypothesis we have $\a[1] \ge \G_0$, and then $\d[1] = \varphi^2_{\d_0[1]} (\varphi_\d(\a[1])+1) > \G_0$.
    \end{itemize}
    \item If $\d = \G_\xi$ with $\xi > 0$ then $\d[1] = \varphi_{\varphi_{\G_{\xi \lceil n \rceil} + 1} (0)} (0) > \G_0$.\qedhere
\end{itemize}
\end{proof}
		
\begin{lemma}\label{notrange}
Fix $\d$ and suppose $\g < \b$ are such that $\g$ is in the range of $\varphi_\d$ but $\b$ is not.
Then $\g \le \b [1]$ and hence $\g \leq \b[n]$ for all $n > 0$.
The same holds if $\g < \b$ are such that $\g = \Gamma_{\g_1}$ but $\b$ is not of the form $\Gamma_\d$ for any $\d$.
\end{lemma}
		
\begin{proof}
Consider first the case of the Veblen functions.
Let $\g = \varphi_\d (\g_1)$ and argue by induction on $\b$.
If the Cantor normal form of $\b$ has more than one term then $\b[1]$ is larger or equal than the first term, which is $\ge \varphi_\d (\g_1) = \g$.
Since $\b$ is not in the range of $\varphi_\d$, the remaining case is $\b = \varphi_\xi (\a)$ for some $\xi < \d$ and $\a < \b$.
If $\a \le \g$ then $\b = \varphi_\xi (\a) \le \varphi_\xi (\varphi_\d (\g_1)) = \varphi_\d (\g_1) = \g$ which is a contradiction.
Thus $\g < \a$ and by the induction hypothesis ($\a$ is not a fixed point of $\varphi_\xi$ and hence is not in the range of $\varphi_\d$) we have $\g \le \a [1]$.
Since $\b [1] = \varphi^2_{\xi [1]} (\varphi_\xi (\a[1])+1) > \a[1]$, %(here we need to use $1$, because $\om^\a [0] = 0 \le \a[0]$), 
we get $\g < \b[1]$.		

In the case of the $\G$ function we again proceed by induction on $\b$.
Recall that $\G_{\g_1}$ is the $\g_1$-th ordinal such that $\varphi_\g (0) = \g$.
The case $\b$ decomposable is trivial as before.
We suppose $\b = \varphi_\xi (\a) > \a$ and we distinguish different subcases.
\begin{itemize}
    \item If $\b = \varphi_\xi(0)$ then $\xi > \g$.
    Moreover $\b > \xi$ otherwise it would be $\b = \xi = \varphi_\xi (0)$, contradicting the assumption of the Lemma that $\b$ is not a $\G$ ordinal.
    Therefore by inductive hypothesis $\xi[1] \ge \g$ and so $\b[1] = \varphi^2_{\xi[1]}(0) > \g$.
    \item If $\b = \varphi_\xi (\a)$ with $0 < \a \le \g$ then it must be $\xi \ge \g$ as otherwise $\b \le \G_{\g_1} = \g$, a contradiction.
    Hence $\b[1] = \varphi^2_{\xi [1]} (\varphi_\xi (\a[1])+1) \ge \varphi^2_{\xi [1]} (\varphi_\g (\a[1])+1) > \g$.
    \item If $\b = \varphi_\xi (\a)$ with $\a > \g$ then, since $\b > \a$, by inductive hypothesis $\a [1] \ge \g$.
    Hence $\b[1] = \varphi^2_{\xi [1]} (\varphi_\xi (\a[1])+1) > \varphi^2_{\xi [1]} (\g) \ge \g$. \qedhere
\end{itemize}
\end{proof}

\begin{lemma}\label{phidelta}
Let $\g$ and $\d$ be ordinals such that $\g$ is not of the form $\varphi_{\g_1} (0)$ for any $\g_1$.
If $\g > \varphi_\d (0)$ then $\g[1] \ge \varphi_\d (0)$ and so $\g[n] \ge \varphi_\d(0)$ for each $n > 0$.
\end{lemma}

\begin{proof}
If $\g$ is decomposable or $\d = 0$ the thesis is trivial. 
We can assume $\g = \varphi_\xi (\g_0) > \g_0 > 0$.
Case $\d = 1$ was proved in Lemma \ref{>epsilon0}.
Case $\xi < \d$ follows by Lemma \ref{notrange}: this is because $\g \notin \ran \varphi_\d$ as otherwise $\g$ would be a fixed point of $\varphi_\xi$, contradicting our assumption that $\g > \g_0$.
Finally, if $\xi \ge \d$ then $\g[1] = \varphi^2_{\xi[1]} (\varphi_\xi(\g_0[1])+1) \ge \varphi^2_{\xi[1]} (\varphi_\d(\g_0[1])+1) > \varphi_\d (0)$.
\end{proof}

Notice that the previous result fails when $\g$ of the form $\varphi_{\g_0} (0)$: e.g.\ $\varphi_\om (0) > \varphi_3 (0) > \varphi^2_1 (0) = \varphi_\om (0) [1]$.

\begin{lemma}\label{gammaseq}
Let $n > 1$ and $\G_\xi$ and $\d$ ordinals.
Then $\d = \G_\xi + 1$ if and only if $\d[n] = \G_\xi$.
\end{lemma}

\begin{proof}
The forward implication follows from the regularity of the system of fundamental sequences.

For the backward direction suppose $\d$ is decomposable.
Then, by regularity, $\d[n] = \G_\xi$ is greater or equal than the first term in the Cantor normal form of $\d$.
Hence $\d = \G_\xi + \b$ with $\G_\xi \ggeq \b > 0$.
Then $\G_\xi = \d [n] = \G_\xi + \b[n]$ implies $\b[n] = 0$.
Since $n > 0$, this can only happen if $\b = 1$.
We are left to prove that $\d$ cannot be indecomposable and we do that by cases.
\begin{itemize}
    \item If $\d = \om^{\d_0} > \d_0$ then $\d[n] = \om^{\d_0 [n]} \cdot n$.
    Since $n > 1$, $\d[n]$ is decomposable contradicting $\d[n] = \G_\xi$.
    \item If $\d = \varphi_{\d_0} (0) > \d_0$ then $\d_0 > 0$ and $\d[n] = \varphi^{n+1}_{\d_0[n]} (0)$.
    If $\d_0[n] < \G_\xi$ then $\d[n] < \G_\xi$.
    If $\d_0[n] \ge \G_\xi$ then, since $n > 0$, $\d[n] > \G_\xi$.
    \item If $\d = \varphi_{\d_0} (\d_1) > \d_1 > 0$ then $\d[n] = \varphi^{n+1}_{\d_0[n]} (\varphi_{\d_0} (\d_1[n]) + 1)$.
    If $\d_0[n] < \G_\xi$ then $\G_\xi$ is a fixed point of $\varphi_{\d_0[n]}$ and so $\G_\xi = \d[n]$ implies $\G_\xi = \varphi_{\d_0} (\d_1[n]) + 1$ which cannot be.
    If $\d_0[n] \ge \G_\xi$ then $\d[n] > \G_\xi$.
    \item If $\d = \G_\a$ then $\a > \xi$ otherwise $\G_\xi \ge \d > \d[n]$.
    In particular $\a > 0$.
    Then if $\G_\xi = \d[n]$, since $\G_\xi = \varphi_{\G_\xi} (0)$ and by the definition of $\G_\a [n]$, we get $\G_\xi = \G_{\a \lceil n \rceil} + 1$ which cannot be. \qedhere
\end{itemize}
\end{proof}

\begin{proof}[Proof of Theorem \ref{nested->nested}]
We prove that for every $\b, \g < \G_\zeta$ and $n > 1$:
\begin{enumerate}[(A)]
	\item\label{A} it is never the case that $\g[n] < \b[n] < \g < \b$;
	\item\label{B} for every $\d < \G_\zeta$, if $\g < \varphi_\d (\g)$ it is not the case that $\g [n] < \varphi_\d (\b [n]) < \g < \varphi_\d (\b)$.
\end{enumerate}

\begin{remark}
Notice that the hypothesis $\g < \varphi_\d(\g)$ in \ref{B} is necessary: let $\lambda$ be a limit ordinal, $\d = \lambda [n]$, $\b = \varphi_\lambda (0)$ and $\g = \varphi_{\d + 1} (0)$.
Then $\b [n] = \g [n] = \varphi_\d^{n + 1} (0) < \varphi_\d^{n + 2} (0) = \varphi_\d (\b [n]) < \g < \b = \varphi_\d (\b)$.
In this counterexample both $\g$ and $\b$ are fixed points of $\varphi_\d$.
\end{remark}

The proofs of \ref{A} and \ref{B} are by simultaneous induction on the pair $(\g, \b)$.
Notice that if $\g$ is not in the range of $\varphi_\d$ then \ref{B} is immediate.
In fact, either $\g \le \varphi_\d (\b [n])$ or $\varphi_\d (\b [n]) < \g$ and we apply Lemma \ref{notrange} to $\varphi_\d (\b [n])$ and $\g$: in this way we get $\varphi_\d(\b[n]) \le \g [1] \le \g [n]$ for all $n \ge 1$.
Therefore we prove \ref{B} only for $\d$ such that $\g$ is in the range, but not a fixed point, of $\varphi_\d$.
			
\begin{description}[before={\renewcommand\makelabel[1]{\bfseries ##1}},wide]
	\item[If $\g \leq 1$] both \ref{A} and \ref{B} hold because two distinct ordinals below $\g$ do not exist.
    
	\item[If $\g$ is of the form $\om^{\g_0}+ \g_1$ with $\om^{\g_0} \ggeq \g_1 > 0$] first recall that $\g[n] = \om^{\g_0}+ \g_1[n]$.
	To prove \ref{A} assume towards a contradiction that $\g[n] < \b[n] < \g < \b$ and consider different subcases depending on the form of $\b$.
	\begin{enumerate}[(i)]
		\item $\b = \om^{\b_0} + \b_1$ with $\om^{\b_0} \ggeq \b_1 > 0$.
		In this case $\g < \b$ implies $\g_0 \le \b_0$ and $\b[n] = \om^{\b_0} + \b_1 [n] < \g$ implies $\b_0 \le \g_0$, so that $\b_0 = \g_0$.
		Hence $\g_1 [n] < \b_1 [n] < \g_1 < \b_1$, contradicting the induction hypothesis of \ref{A}.
		\item $\b = \om^{\b_0} > \b_0$.
        In this case $\g < \b$ implies $\g_0 < \b_0$ and $\b [n] = \om^{\b_0 [n]} \cdot n < \g$ implies $\b_0 [n] \le \g_0$.
		\begin{itemize}
			\item If $\b_0 [n] < \g_0$ then $\b [n] < \om^{\g_0} \le \g [n]$ against $\g [n] < \b [n]$.
			\item If $\b_0 [n] = \g_0$ then $\om^{\g_0} \cdot n < \g = \om^{\g_0} + \g_1$ and hence $\om^{\g_0} \cdot (n-1) < \g_1$ so that we can write $\g_1 = \om^{\g_0} \cdot (n-1) + \d$ for some $\d$ with $\om^{\g_0} \ggeq \d > 0$.
            Then $\g_1 [n] = \om^{\g_0} \cdot (n-1) + \d [n] \ge \om^{\g_0} \cdot (n-1)$ and $\g [n] \ge \om^{\g_0} \cdot n = \b [n]$.
            This contradicts again $\g [n] < \b [n]$.
		\end{itemize}
		\item $\b = \varphi_{\b_0} (\b_1)$ with $\b_0 > 0$ and $\b_1 < \b$ or $\b = \G_\xi$.
        %$\b = \G_0$
		In this case $\b [n]$ is in the range of $\varphi_0$ while $\g$ is not: applying Lemma \ref{notrange} we obtain $\b [n] \le \g [1]$ contradicting $\g [n] < \b [n]$.
	\end{enumerate}
	There is no need to prove \ref{B} because $\g$ is in the range of no $\varphi_\d$.				

    \item[If $\g$ is of the form $\om^{\g_0}$ with $0 < \g_0 < \g$] recall that $\g [n] = \om^{\g_0 [n]} \cdot n$.
	To prove \ref{A} assume towards a contradiction that $\g [n] < \b [n] < \g < \b$ and consider different subcases depending on the form of $\b$:
	\begin{enumerate}[(i)]
		\item $\b = \om^{\b_0} + \b_1$ with $\om^{\b_0} \ggeq \b_1 > 0$.
		Then $\g < \b$ implies $\g_0 \le \b_0$ and hence $\g \le \om^{\b_0} + \b_1 [n] = \b [n]$.
		\item $\b = \om^{\b_0}$ with $\b_0 < \b$.
        Then $\g < \b$ implies $\g_0 < \b_0$, $\b [n] = \om^{\b_0 [n]} \cdot n < \g$ implies $\b_0 [n] < \g_0$ and $\g [n] < \b [n]$ implies $\g_0 [n] < \b_0 [n]$.
		We thus have $\g_0 [n] < \b_0 [n] < \g_0 < \b_0$ against the induction hypothesis of \ref{A}.
		\item $\b = \varphi_1 (0) = \epsilon_0$.
		By Lemma \ref{<epsilon0} $\b[n] = \varphi_0^{n+1} (0) \le \g [1] \le \g [n]$, which contradicts $\g [n] < \b [n]$.
		\item $\b = \varphi_1 (\b_1) = \epsilon_{\beta_1}$ with $0<\b_1 < \b$.
		To simplify the notation let $\xi = \epsilon_{\b_1 [n]} + 1$, so that $\b [n] = \varphi_0^{n + 1} (\xi)$.
		We have that $\b [n] < \g$ implies $\varphi_0^n (\xi) < \g_0$.
		If $\g_0 < \om^{\varphi_0^{n - 1} (\xi) + 1}$ then $\varphi_0^n (\xi)$ is the leading term of the Cantor normal form of $\g_0$ (which includes other terms) and $\g_0 [n] \ge \varphi_0^n (\xi)$.
		Thus $\g [n] > \om^{\g_0 [n]} \ge \varphi_0^{n + 1} (\xi) = \b [n]$, a contradiction.
		We can therefore assume $\g_0 \ge \om^{\varphi_0^{n - 1} (\xi) + 1}$.
		Let $\g_1$ be such that $\om^{\g_1} \le \g_0 < \om^{\g_1 + 1}$ which, by the above assumption, implies $\g_1 > \varphi_0^{n - 1} (\xi)$.
		Consider first the case $\om^{\g_1} < \g_0$: in this case the Cantor normal form of $\g_0$ includes other terms after the leading term $\om^{\g_1}$ and hence $\g_0 [n] \ge \om^{\g_1} > \varphi_0^n (\xi)$ so that $\g [n] > \om^{\g_0 [n]} > \varphi_0^{n + 1} (\xi) = \b [n]$, a contradiction.
		We thus have $\g_0 = \om^{\g_1}$ and $\g = \varphi_0^2 (\g_1)$.
%		If $\g_1 < \om^{\varphi_0^{n - 2} (\xi) + 1}$ then $\varphi_0^{n - 1} (\xi)$ is the leading term of the Cantor normal form of $\g_1$, which includes other terms, (recall that $\g_1 > \varphi_0^{n - 1} (\xi)$) and $\g_1 [n] \ge \varphi_0^{n - 1} (\xi)$.
%		Then $\g [n] = \om^{\om^{\g_1 [n]} \cdot n} \cdot n > \varphi_0^{n + 1} (\xi) = \b [n]$, a contradiction.
%		Therefore we assume $\g_1 \ge \om^{\varphi_0^{n - 2} (\xi) + 1}$.
%		Let $\g_2$ be such that $\om^{\g_2} \le \g_1 < \om^{\g_2 + 1}$ which, by the above assumption, implies $\g_2 > \varphi_0^{n - 2} (\xi)$.
%		Consider first the case $\om^{\g_2} < \g_1$: in this case the Cantor normal form of $\g_1$ includes other terms after the leading term $\om^{\g_2}$ and hence $\g_1 [n] \ge \om^{\g_2} > \varphi_0^{n - 1} (\xi)$ so that $\g [n] > \varphi_0^2 (\g_1 [n]) > \varphi_0^{n + 1} (\xi) = \b [n]$, a contradiction.
%		We thus have $\g_1 = \om^{\g_2}$ and $\g = \varphi_0^3 (\g_2)$.
		We can repeat this argument until we obtain that $\g = \varphi_0^{n + 1} (\g_n)$ for some $\g_n > \xi$.
		But then $\varphi_0^{n + 1} (\g_n [n]) < \g [n] < \b [n] = \varphi_0^{n + 1} (\xi) < \g = \varphi_0^{n + 1} (\g_n)$ which implies $\g_n [n] < \xi < \g_n$.
        Therefore we have $\g_n[n] \le \epsilon_{\b_1[n]}$.
        Towards a contradiction suppose $\g_n[n] = \epsilon_{\b_1[n]}$.
        If $\g_n$ is decomposable then it must be $\g_n = \epsilon_{\b_1[n]} + 1$ which cannot be.
        Otherwise, since $\g_n \notin \ran \varphi_1$, it must be $\g_n = \om^\d > \d$.
        In this case $n > 1$ implies that $\g_n [n]$ is decomposable and so it cannot be $\epsilon_{\b_1[n]}$.
        It follows that $\g_n [n] < \epsilon_{\b_1 [n]} < \g_n$.
		Moreover $\g_n < \g < \b = \epsilon_{\beta_1}$ and, since $\g_n$ is not a fixed point of $\varphi_1$, we contradict the induction hypothesis of \ref{B} with $\g_n[n] < \epsilon_{\b_1[n]} < \g_n < \epsilon_{\b_1}$.
		\item $\b = \varphi_{\b_0} (\b_1)$ with $\b_0 > 1$ and $\b_1 < \b$ or $\b = \G_\xi$.
        %$\b = \G_0$.
		In this case both $\b [n]$ and $\b$ are fixed points of $\varphi_0$ and $\om^{\g_0 [n]} \cdot n < \b [n] < \om^{\g_0} < \b$ implies $\g_0 [n] < \b [n] < \g_0 < \b$ against the induction hypothesis of \ref{A}.
	\end{enumerate}
	We need to prove \ref{B} only for $\d = 0$ since $\g \notin \ran \varphi_1$.
	If $\om^{\g_0 [n]} \cdot n < \om^{\b [n]} < \om^{\g_0} < \om^\b$ for some $\b$, we have that $\g_0 [n] < \b [n] < \g_0 < \b$ against the induction hypothesis of \ref{A}.
				
	\item[If $\g$ is $\varphi_1 (0) = \epsilon_0$] then $\b > \epsilon_0$ and Lemma \ref{>epsilon0} imply that $\b[n] \ge \epsilon_0$ for all $n \ge 1$.
	Therefore we cannot have a counterexample to \ref{A}.
	On the other hand \ref{B} is immediate because we should consider only the case $\d = 1$, and $\varphi_1 (\b [n]) < \epsilon_0$ is impossible.
				
	\item[If $\g$ is of the form $\varphi_{\g_0} (0)$ with $1 < \g_0 < \g$] recall that $\g [n] = \varphi_{\g_0 [n]}^{n+1} (0)$.
	To prove \ref{A} we assume towards a contradiction that $\g [n] < \b [n] < \g < \b$ and consider different subcases depending on the form of $\b$:
	\begin{enumerate}[(i)]
		\item $\b = \om^{\b_0} + \b_1$ with $\om^{\b_0} \ggeq \b_1 > 0$.
		Then, since $\g$ is an indecomposable ordinal, $\g < \b$ implies $\g \le \om^{\b_0}$ and hence $\g \le \om^{\b_0} + \b_1 [n] = \b [n]$.
		\item $\b = \varphi_{\b_0} (0)$ with $\b_0 < \b$.
        Case $\b_0 = 0$ is trivial.
		Otherwise $\g_0 < \b_0$ and we have $\varphi_{\g_0 [n]}^{n + 1} (0) < \varphi_{\b_0 [n]}^{n + 1} (0) < \varphi_{\g_0} (0) < \varphi_{\b_0} (0)$.
		Then $\g_0 [n] < \b_0 [n] < \g_0 < \b_0$, contradicting the induction hypothesis of \ref{A}.   
		\item $\b = \varphi_{\b_0} (\b_1)$ with $0 < \b_1 < \b$.
		Notice that $\g_0 < \b_0$ cannot hold, otherwise $\varphi_{\g_0} (\varphi_{\b_0} (\b_1 [n])) = \varphi_{\b_0} (\b_1 [n]) < \b [n] < \g$ (the first equality holds because $\b_0 > \g_0$) which is impossible as $\g$ is the least element in the range of $\varphi_{\g_0}$.
		Analogously, $\b_0 = \g_0$ yields $\varphi_{\g_0} (\b_1 [n]) = \varphi_{\b_0} (\b_1 [n]) < \b [n] < \g$ which is impossible for the same reason.
		Thus $\b_0 < \g_0$ and $\b$ does not belong to the range of $\varphi_{\g_0}$: Lemma \ref{notrange} then implies $\g \le \b [n]$, a contradiction.
		\item \label{case G0 g0} $\b = \G_0$.
        Then we are assuming $\varphi_{\g_0 [n]}^{n + 1} (0) < \varphi_{\G_0 [n - 1]} (0) < \varphi_{\g_0} (0) < \G_0$ which implies $\g_0 [n] < \G_0[n - 1] < \g_0 < \G_0$.
		If $n > 2$ (so that $n - 1 > 1$) we have $\g_0 [n - 1] < \G_0 [n - 1] < \g_0 < \G_0$ contradicting the induction hypothesis of \ref{A}.
		In the case $n = 2$, since $\G_0 [1] = \epsilon_0$, we get a contradiction with Lemma \ref{>epsilon0}.
		\item $\b = \G_\xi$.
        Let $\d_{n+1} = \G_{\xi \lceil n \rceil} + 1$ and $\d_i = \varphi_{\d_{i+1}} (0)$ for $i \le n$.
        Thus $\G_\xi [n] = \varphi_{\d_0} (0)$.
        Then our assumption $\g [n] < \b [n] < \g < \b$ yields $\g_0[n] < \d_0 < \g_0 < \b$.
        We claim that $\g_0 = \varphi_{\g_1} (0)$ for some $\g_1$.
        %Notice that it cannot be $\g_1 = \g_0$ otherwise $\g = \g_0$.
        In fact, if $\g_0$ were decomposable then, since $\d_0$ is indecomposable, $\g_0[n] \ge \d_0$, a contradiction.
            %\item If $\g_0 = \om^{\g_1} > \g_1$ then, since $\b$ and $\d_{n+1}$ are fixed points of $\varphi_0$, we get $\g_1[n] < \d_{n+1} < \g_1 < \b$.
            %Hence $\varphi^{n+1}_{\g_1[n]} (0) < \varphi_{\d_{n+1}} (0) < \varphi_{\g_1} (0) < \b$, contradicting the inductive hypothesis of \ref{A}.
        If instead $\g_0 = \varphi_{\g_1} (\g_2) > \g_2 > 0$ then, since $\d_0 = \varphi_{\d_1} (0)$, Lemma \ref{phidelta} implies $\g_0[n] \ge \d_0$.
            %\item If $\g_0 = \varphi_{\g_0} (0)$, then $\g = \g_0$, against our assumption.
        This proves the claim.        
        Notice that it must be $\g_0 > \g_1 > 0$.
        The claim and the chain of inequalities $\g_0[n] < \d_0 < \g_0 < \b$ yields $\g_1[n] < \d_1 < \g_1 < \b$.
        Iterating $n+1$ times we obtain for each $i \le n+1$ an ordinal $\g_i$ such that $\g_{i-1} = \varphi_{\g_i} (0)$ and $\g_i[n] < \d_i < \g_i < \b$.
        In particular $\g_{n+1}[n] < \d_{n+1} < \g_{n+1} < \b$.
        Since $\g_{n+1}$ is not a $\G$ ordinal and $\g_{n+1} > \d_{n+1} = \G_{\xi \lceil n \rceil} + 1$, Lemma \ref{notrange} implies $\g_{n+1} [n] \ge \G_{\xi \lceil n \rceil}$.
        Since $\G_{\xi \lceil n \rceil} + 1 > \g_{n+1} [n]$ it must be $\g_{n+1} [n] = \G_{\xi \lceil n \rceil}$.
        Lemma \ref{gammaseq} implies that $\g_{n+1} = \G_{\xi \lceil n \rceil} + 1$, contradicting $\G_{\xi \lceil n \rceil} + 1 = \d_{n+1} < \g_{n+1}$.
	\end{enumerate}
	To prove \ref{B} we should consider only the case $\d = \g_0$: $\varphi_{\g_0} (\b [n]) < \varphi_{\g_0} (0)$ is impossible.
				
	\item[If $\g$ is of the form $\varphi_{\g_0}(\g_1)$ with $0 < \g_0$ and $0<\g_1<\g$] we first recall that $\g[n] = \varphi_{\g_0[n]}^{n+1}(\varphi_{\g_0}(\g_1[n])+1)$.
    To prove \ref{A} we assume towards a contradiction that $\g[n] < \b[n] < \g < \b$ and consider different subcases depending on the form of $\b$:
	\begin{enumerate}[(i)]
		\item $\b = \om^{\b_0} + \b_1$ with $\om^{\b_0} \ggeq \b_1 > 0$.
        Then, since $\g$ is an indecomposable ordinal, $\g < \b$ implies $\g \le \om^{\b_0}$ and hence $\g \le \om^{\b_0} + \b_1 [n] = \b [n]$.
        \item $\b = \varphi_{\b_0} (0)$ with $\b_0<\b$.
        Case $\b_0 = 0$ is trivial.
        So assume $\b_0 > 0$, recall that $\b[n] = \varphi_{\b_0[n]}^{n+1}(0)$ and notice that $\g<\b$ implies $\g_0 < \b_0$.
        We consider three different subcases.
		\begin{itemize}
			\item If $\b_0 [n] < \g_0$ then we also have $\g_0[n] < \b_0[n]$, against the induction hypothesis of \ref{A}.
            In fact $\b_0[n] \leq \g_0[n]$ implies $\b[n] = \varphi_{\b_0[n]}^{n+1} (0) \leq \varphi_{\g_0[n]}^{n+1}(0) < \g[n]$, a contradiction.
			\item If $\b_0 [n] = \g_0$ then, peeling off an application of $\varphi_{\g_0}$ we obtain $\varphi_{\b_0[n]}^n(0) < \g_1$.
            If $\g_1$ does not belong to the range of $\varphi_{\b_0 [n]} = \varphi_{\g_0}$ then Lemma \ref{notrange} implies $\varphi_{\b_0 [n]}^n (0) \le \g_1[n]$ and hence $\b[n] \le \varphi_{\g_0} (\g_1 [n]) < \g [n]$ which is a contradiction.
            Therefore $\g_1 = \varphi_{\g_0} (\g_2)$ for some $\g_2 < \g_1$ and we have $\varphi_{\b_0[n]}^{n-1}(0) < \g_2$.
            If $\g_2$ does not belong to the range of $\varphi_{\b_0 [n]}$ then Lemma \ref{notrange} implies $\varphi_{\b_0 [n]}^{n-1} (0) \le \g_2[n]$ and hence $\b[n] \le \varphi_{\g_0}^2(\g_2[n]) < \varphi_{\g_0} (\varphi_{\g_0}(\g_2[n]) +1) < \varphi_{\g_0} (\g_1[n]) < \g[n]$.
            Therefore $\g_2 = \varphi_{\g_0} (\g_3)$ for some $\g_3 < \g_2$ and we iterate this procedure finding for each $i \leq n+1$ some $\g_i$ such that $\varphi_{\g_0}(\g_i) = \g_{i-1}$ and $\varphi_{\b_0[n]}^{n-i+1}(0) < \g_i$.
            Then, starting from $0 \le \g_{n+1} [n]$ we obtain $\b[n] \leq \varphi_{\g_0}^{n+1}(\g_{n+1}[n]) < \dots < \varphi_{\g_0}^i (\g_i[n]) <  \dots < \varphi_{\g_0} (\g_1[n]) < \g[n]$, which is again a contradiction.
			\item If $\g_0 < \b_0 [n]$ then $\g$ is not in the range of $\varphi_{\b_0[n]}$ and Lemma \ref{notrange} implies that $\b [n] \le \g [n]$.
		\end{itemize}
		\item $\b = \varphi_{\b_0} (\b_1)$ with $0<\b_1<\b$.
        As before, we consider different subcases.
		\begin{itemize}
			\item If $\b_0 < \g_0$ then $\b$ is not in the range of $\varphi_{\g_0}$ and Lemma \ref{notrange} implies that $\g \le \b[n]$.
			\item If $\b_0 = \g_0$ then we have $\g_1<\b_1$ and, since $\b_0[n]=\g_0[n]$, also $\g_1 [n] < \b_1 [n]$.
            The induction hypothesis of \ref{A} then yields that $\g_1 \le \b_1 [n]$ and hence $\g \leq \varphi_{\g_0}(\b_1[n]) < \varphi_{\b_0}(\b_1[n]) + 1 < \b[n]$ a contradiction.
			\item If $\b_0 [n] < \g_0 < \b_0$ then by the induction hypothesis of \ref{A} we have $\b_0[n] \le \g_0[n]$.
            Since $\varphi_{\g_0} (\varphi_{\b_0} (\b_1[n])) = \varphi_{\b_0} (\b_1[n]) < \b[n] < \g$ (where the equality holds because $\g_0 < \b_0$) we have, peeling off an application of $\varphi_{\g_0}$, that $\varphi_{\b_0} (\b_1[n]) < \g_1$.
            Notice also that $\g_1 < \g < \b$ so that we get $\varphi_{\b_0} (\b_1[n]) < \g_1 < \varphi_{\b_0}(\b_1)$.
            We know that $\g_1$ is not a fixed point of $\varphi_{\g_0}$ and, a fortiori, of $\varphi_{\b_0}$.
            Hence, by the induction hypothesis of \ref{B}, $\varphi_{\b_0} (\b_1[n]) \leq \g_1[n]$.
            Applying $\varphi_{\g_0}$ to both members of the previous inequality, we get $\varphi_{\b_0} (\b_1[n]) + 1 \le \varphi_{\g_0}(\g_1[n]) + 1$.
            Finally, applying $n+1$ times $\varphi_{\b_0[n]}$ to the left hand side and $\varphi_{\g_0[n]}$ to the right hand side of the previous inequality (in addition to the fact that $\b_0[n] \le \g_0[n]$), we get precisely $\b[n] \le \g[n]$.
			\item If $\b_0[n] = \g_0$ then to ease the notation let $\xi = \varphi_{\b_0}(\b_1[n])+1$ so that $\b[n] = \varphi_{\b_0 [n]}^{n+1} (\xi)$. 
            Then we have $\varphi_{\b_0 [n]}^n (\xi) < \g_1$ by peeling off one application of $\varphi_{\g_0}$ from both sides.
            If $\g_1$ does not belong to the range of $\varphi_{\b_0 [n]} = \varphi_{\g_0}$ then Lemma \ref{notrange} implies $\varphi_{\b_0[n]}^n(\xi) \le \g_1[n]$ and hence $\b[n] \le \varphi_{\g_0} (\g_1 [n]) < \g [n]$ which is a contradiction.
            Therefore $\g_1 = \varphi_{\g_0} (\g_2)$ for some $\g_2 < \g_1$ and we have $\varphi_{\b_0[n]}^{n-1}(\xi) < \g_2$ again by peeling off one application of $\varphi_{\g_0}$ from both sides. 
            If $\g_2$ does not belong to the range of $\varphi_{\b_0 [n]}$ then Lemma \ref{notrange} implies $\varphi_{\b_0 [n]}^{n-1} (\xi) \le \g_2 [n]$ and hence $\b [n] \le \varphi_{\g_0}^2(\g_2 [n]) < \varphi_{\g_0} (\g_1 [n]) < \g[n]$.
            Therefore $\g_2 = \varphi_{\g_0} (\g_3)$ for some $\g_3 < \g_2$ and we iterate this procedure finding for each $i \le n+1$ some $\g_i$ such that $\varphi_{\g_0} (\g_i) = \g_{i-1}$ and $\varphi_{\b_0 [n]}^{n-i+1} (\xi) < \g_i$.
            Then we have $\varphi_{\b_0} (\b_1 [n]) < \xi < \g_{n+1}$ and $\g_{n+1}$ is not a fixed point for $\varphi_{\g_0}$ and hence not in the image of $\varphi_{\b_0}$ because $\b_0 > \b_0[n] = \g_0$.
            By Lemma \ref{notrange} $\varphi_{\b_0}(\b_1[n]) \le \g_{n+1}[1]$.
            Notice that if $\g_{n+1}$ is a successor then $\varphi_{\b_0}(\b_1[n]) < \xi \le \g_{n+1}[1]$ and if $\g_{n+1}$ is limit then $\g_{n+1}[1] < \g_{n+1}[n]$.
            Therefore $\varphi_{\b_0}(\b_1[n]) < \g_{n+1}[n]$ and hence $\xi \le \g_{n+1}[n]$.
            Thus, applying $n+1$ times $\varphi_{\g_0}$ to both sides of this inequality, we get $\b [n] \le \varphi_{\g_0}^{n+1} (\g_{n+1} [n]) < \ldots < \varphi_{\g_0}^i (\g_i [n]) < \ldots < \varphi_{\g_0} (\g_1 [n]) < \g [n]$.
			\item If $\g_0 < \b_0 [n]$ then $\g$ is not in the range of $\varphi_{\b_0[n]}$ and Lemma \ref{notrange} implies that $\b [n] \le \g [n]$.
		\end{itemize}
    	\item $\b = \Gamma_0$. We are assuming $\varphi_{\g_0[n]}^{n+1}(\varphi_{\g_0}(\g_1[n])+1) < \varphi_{\Gamma_0[n-1]}(0) < \varphi_{\g_0}(\g_1) < \Gamma_0$ and consider three subcases.
        \begin{itemize}
            \item If $\G_0[n-1] > \g_0$ then $\g_1 > \G_0[n]$.
            Since $\g_1 [n] < \g[n]$ we would get $\g_1[n] < \G_0[n] < \g_1 < \G_0$, contradicting the inductive hypothesis of \ref{A}.
            \item If $\G_0[n-1] = \g_0$ then $\G_0 [n] < \varphi_{\g_0} (\g_1[n]) + 1 < \g[n]$, a contradiction.
            \item Otherwise $\g_0[n] < \Gamma_0[n-1] < \g_0 < \Gamma_0$ and we can argue as in \ref{case G0 g0} of case $\g = \varphi_{\g_0} (0)$.
	    \end{itemize}
       \item $\b = \Gamma_\xi$.
        Let $\d$ be such that $\b[n] = \varphi_{\varphi_\d (0)} (0)$ and notice that $\varphi_\d (0) > \g_0 [n]$.
        We are assuming $\varphi^{n+1}_{\g_0[n]} (\varphi_{\g_0} (\g_1[n]) + 1) < \varphi_{\varphi_\d (0)} (0) < \varphi_{\g_0} (\g_1) < \G_\xi$ and consider different subcases.
        \begin{itemize}
            \item If $\g_0 > \varphi_\d (0)$.
            Thus $\g_0[n] < \varphi_\d (0) < \g_0 < \b$ and so $\varphi_{\g_0 [n]} (0) < \varphi_{\varphi_\d (0)} (0) = \b [n] < \varphi_{\g_0} (0) < \varphi_\b (0) = \b$.
            Since $\varphi_\d (0) > \g_0[n]$ then $\b [n]$ is a fixed point of $\varphi_{\g_0 [n]}$ and we get $\varphi_{\g_0 [n]}^{n+1} (0) < \b [n] < \varphi_{\g_0} (0) < \b$.
            This contradicts the induction hypothesis of \ref{A}.
            \item If $\g_0 = \varphi_\d (0)$ then $\varphi^{n+1}_{\g_0 [n]} (0) < \g[n] < \b[n] = \varphi_{\g_0} (0) < \g$ contradicting the inductive hypothesis of \ref{A}.
            \item If $\g_0 < \varphi_\d (0)$ then both $\b$ and $\b[n]$ are fixed points of $\varphi_{\g_0}$.
            Hence peeling off one application of $\varphi_{\g_0}$ we get $\b[n] < \g_1 < \b$.
            Moreover, $\g_1 [n] < \g[n] < \b[n]$ contradicts the inductive hypothesis of \ref{A}.
        \end{itemize}
	\end{enumerate}
	To prove \ref{B}  we should consider only the case $\d=\g_0$ and assume towards a contradiction that $\g [n] < \varphi_{\g_0} (\b [n]) < \g < \varphi_{\g_0} (\b)$.
	Then $\b[n]<\g_1<\b$ and the induction hypothesis of \ref{A} implies $\b[n] \leq \g_1[n]$ so that $\varphi_{\g_0}(\b[n]) \leq \varphi_{\g_0}(\g_1[n]) < \g[n]$, a contradiction.
    \item[If $\g$ is of the form $\G_0$] then $\b > \G_0$ and Lemma \ref{>epsilon0} imply that $\b[n] \ge \G_0$ for all $n \ge 1$.
    Therefore we cannot have a counterexample to \ref{A}.
    On the other hand \ref{B} is immediate because we should consider only the case $\d = \G_0$, and $\varphi_{\G_0} (\b [n]) < \G_0 = \varphi_{\G_0} (0)$ is impossible.
    \item[If $\g$ is of the form $\G_\xi$] then to prove \ref{A} assume towards a contradiction that $\g[n] < \b[n] < \g < \b$ and consider different subcases depending on the form of $\b$:
    \begin{enumerate}[(i)]
        \item $\b \ne \G_\d$ for all $\d$.
        In this case $\b > \G_\xi$ and Lemma \ref{notrange} imply that $\g \le \b[n]$ a contradiction.
        \item $\b = \G_\d$ with $\d < \b$.
        In this case $\xi < \d$ and $\G_\xi [n] < \G_\d [n] < \G_\xi < \G_\d$.
        By considering $n+1$ times the index of the Veblen functions, the first inequality becomes $\G_{\xi \lceil n \rceil} < \G_{\d \lceil n \rceil}$.
        Then we get $\xi \lceil n \rceil < \d \lceil n \rceil < \xi < \d$ contradicting the smoothness of the system of fundamental sequences on $\zeta$ we started from.
    \end{enumerate}
    To prove \ref{B} we should only consider the case of $\varphi_{\G_\xi}$.
    This is immediate since $\varphi_{\G_\xi} (\b[n]) < \G_\xi = \varphi_{\G_\xi} (0)$ is impossible.\qedhere
   \end{description}
\end{proof}

%
%
%
%
%\bibliographystyle{alpha}
%\bibliography{bibliography}

%\begin{comment}

%\end{comment}
		
\end{document}